\documentclass[a4paper,11pt,reqno]{amsart}

\pdfoutput=1 
\usepackage{ifpdf}

\usepackage[hidelinks,pdfpagelabels]{hyperref}

\usepackage[british]{babel} 
\usepackage[latin1]{inputenc} 
\usepackage[T1]{fontenc} 

\usepackage{amsmath,amssymb,amsfonts,amsthm} 
\usepackage{bbm}

\usepackage[all,2cell,cmtip]{xy} 
\UseAllTwocells

\allowdisplaybreaks 
\usepackage[svgnames]{xcolor}
\numberwithin{equation}{section}

\usepackage{enumitem}

\usepackage{etoolbox}
\makeatletter 
\mathchardef\latex@simeq\simeq
\let\simeq\relax
\DeclareRobustCommand{\simeq}{\mathrel{\mathpalette\new@simeq\relax}}
\newcommand{\new@simeq}[2]{%
  \raisebox{\simeq@raise{#1}}{$\m@th#1\latex@simeq$}%
}
\newcommand{\simeq@raise}[1]{%
  \ifx#1\displaystyle .425\fontdimen22\textfont2\fi
  \ifx#1\textstyle .425\fontdimen22\textfont2\fi
  \ifx#1\scriptstyle .425\fontdimen22\scriptfont2\fi
  \ifx#1\scriptscriptstyle .425\fontdimen22\scriptscriptfont2\fi
}
\patchcmd{\@vereq}{.5}{0}{}{}
\makeatother 

\theoremstyle{plain}
\newtheorem{teo}[subsection]{Theorem}
\newtheorem{prop}[subsection]{Proposition}
\newtheorem{lem}[subsection]{Lemma}
\newtheorem{cor}[subsection]{Corollary}

\newtheorem*{teo*}{Theorem}
\theoremstyle{remark}
\newtheorem{eje}[subsection]{Example}

\newtheorem{rem}[subsection]{Remark}

\theoremstyle{definition}
\newtheorem{defi}[subsection]{Definition}

\renewenvironment{proof}{\textit{Proof.}$\;$}{\qed}

\DeclareMathOperator{\Set}{\mathsf{Set}}

\DeclareMathOperator{\Vect}{\mathsf{Vect}}

\DeclareMathOperator{\Opmon}{\mathsf{OpMon}}
\DeclareMathOperator{\SkOpmon}{\mathsf{SkOpMon}}
\DeclareMathOperator{\SkMon}{\mathsf{SkMon}}
\DeclareMathOperator{\OplaxAct}{\mathsf{OplaxAct}}

\DeclareMathOperator{\Cat}{\mathsf{Cat}}

\DeclareMathOperator{\Mod}{\mathsf{Mod}}
\DeclareMathOperator{\Comod}{\mathsf{Comod}}
\DeclareMathOperator{\rComod}{\mathsf{rComod}}
\DeclareMathOperator{\Span}{\mathsf{Span}}
\DeclareMathOperator{\Prof}{\mathsf{Prof}}
\DeclareMathOperator{\Cgb}{\mathsf{Cgb}}

\DeclareMathOperator{\id}{\mathrm{id}}

\newcommand{\ot}{^{\circ}\!}
\newcommand{\ob}{_{\circ}\!}
\renewcommand{\qedsymbol}{$\blacksquare$}
\renewcommand{\qed}{\hfill\qedsymbol}
\renewenvironment{proof}{\textit{Proof.}$\;$}{\qed}
\newcommand{\pb}[1]{\save*!/#1-1.2pc/#1:(-1,1)@^{|-}\restore} 

\title{Coalgebroids in monoidal bicategories and their comodules}
\author{Ramón Abud Alcalá}
\thanks{The results in this paper are included in the first chapter of my PhD thesis \emph{Oplax actions and enriched icons with applications to coalgebroids and quantum categories} which was written under the supervision of Steve Lack. I would like to thank Steve Lack for all his guidance and useful comments.}
\begin{document}

\begin{abstract}
Quantum categories have been recently studied because of their relation to bialgebroids, small categories, and skew monoidales. This is the first of a series of papers based on the author's PhD thesis in which we examine the theory of quantum categories developed by Day, Lack, and Street.

A quantum category is an opmonoidal monad on the monoidale associated to a biduality $R\dashv R\ot$, or enveloping monoidale, in a monoidal bicategory of modules $\Mod(\mathcal{V})$ for a monoidal category $\mathcal{V}$. Lack and Street proved that quantum categories are in equivalence with right skew monoidales whose unit has a right adjoint in $\Mod(\mathcal{V})$. Our first important result is similar to that of Lack and Street. It is a characterisation of opmonoidal \emph{arrows} on enveloping monoidales in terms of a new structure named \emph{oplax action}. We then provide three different notions of comodule for an opmonoidal arrow, and using a similar technique we prove that they are equivalent. Finally, when the opmonoidal arrow is an opmonoidal monad, we are able to provide the category of comodules for a quantum category with a monoidal structure such that the forgetful functor is monoidal.
\end{abstract}
\maketitle
\section{Motivation and Historical Context}
Bialgebroids were defined by Takeuchi in \cite{Takeuchi1977} as an alternative to the existing theory of $\times$-bialgebras over a commutative algebra due to Sweedler \cite{Sweedler1974}, so as to allow a non-commutative base algebra as well. Almost twenty years later in \cite{Day2003a}, Day and Street used 2-dimensional category theory to prove that Takeuchi's bialgebroids and small categories share a common theoretical framework which they call \emph{quantum categories}. In this paper we extend certain aspects of the existing theory of bialgebroids to a more general context, which in particular includes that of quantum categories.

While bialgebras over a commutative ring $k$ consist of a $k$-algebra and a $k$-coalgebra interacting in an appropriate way, the elementary description of a bialgebroid from the viewpoint of classical ring and module theory is quite elaborate. If $R$ is a (not necessarily commutative) $k$-algebra, the data for a $R$-bialgebroid consists of a $k$-module $B$ together with suitably compatible $R$-coring and $(R\ot\otimes R)$-ring structures; i.e. a comonoid in $R$-$\Mod$-$R$ and a monoid in $(R\ot\otimes R)$-$\Mod$-$(R\ot\otimes R)$. Some of the symmetry that bialgebras have is now lost for bialgebroids; for example, in the definition of a bialgebra one may exchange the roles of the ``algebra'' and the ``coalgebra'' structures and get a bialgebra again, whereas for a bialgebroid swapping the roles of the ``ring'' and ``coring'' structures gives a different mathematical object. Notice that there are four $R$-actions on the same $k$-module $B$ for which even choosing an adequate notation is not simple and each author does it in a different way.

In the early 2000's Szlachányi made significant contributions towards a simpler description of a bialgebroid based on the work by \cite{Moerdijk2002} and \cite{McCrudden2002} on opmonoidal monads, and later developing some categorical tools himself; namely \emph{skew monoidal categories}.

\begin{teo}\label{teo:Bialgebroid}
For a $k$-algebra $R$ the following are equivalent,
\begin{enumerate}
\item A right $R$-bialgebroid (original definition $\times_R$-bialgebra \cite[Section 4]{Takeuchi1977}).
\item An $(R\ot\otimes R)$-ring $B$ for which the category of right $B$-modules has a monoidal structure such that the forgetful functor is strong monoidal \cite[Theorem 5.1]{Schauenburg1998}.
\item A cocontinuous opmonoidal monad on the category $R$-$\Mod$-$R$ \cite[Section 4.2]{Szlachanyi2003}.
\item A monoid in a monoidal category of coalgebroids \cite[Section 2.1]{Szlachanyi2004}.
\item A closed right skew monoidal structure on the category $\Mod$-$R$ with skew unit $R$ \cite[Theorem 9.1]{Szlachanyi2012}.
\end{enumerate}
\end{teo}

Motivated by the work of Szlachányi the Australian school of category theory gives a similar account of Theorem~\ref{teo:Bialgebroid} but in a bicategorical language instead, with the concept of quantum categories for a monoidal category $\mathcal{V}$ taking the place where bialgebroids are. Quantum categories are defined for a symmetric monoidal category with equalisers of coreflexive pairs $(\mathcal{V},\otimes,I)$, but within the context of the bicategory $\Comod(\mathcal{V})$ of comonoids in $\mathcal{V}$, two sided comodules between them, and their morphisms. There is a notion of duality amongst the objects of $\Comod(\mathcal{V})$; for each comonoid $R$ in $\mathcal{V}$ there is a comonoid $R\ot$ obtained by reversing the comultiplication rule of $R$. These comonoids come equipped with ``unit'' and ``counit'' two sided comodules
\[
\vcenter{\hbox{\xymatrix{
n:I\ar[r]&R\ot\otimes R
}}}
\qquad\qquad
\vcenter{\hbox{\xymatrix{
e:R\otimes R\ot\ar[r]&I
}}}
\]
both of which have $R$ as the underlying object and whose actions are the left and right regular actions with respect to $R\ot$ and $R$ as pictured above. Furthermore, these comodules satisfy the triangle identities in $\Comod(\mathcal{V})$ up to coherent isomorphism. This concept is that of a biduality, and because every object $R$ has a right bidual $R\ot$ we say that $\Comod(\mathcal{V})$ is right autonomous. Bidualities induce a monoidal structure on the object $R\ot\otimes R$ with product $\xymatrix@1@C=5mm{1\otimes e\otimes 1:R\ot\otimes R\otimes R\ot\otimes R\ar[r]&R\ot\otimes R}$ and unit $n$ satisfying the associative and unit laws up to coherent isomorphism. We call this structure the enveloping monoidale of a biduality. The following theorem summarises the view that the Australian school of category theory gave to Theorem~\ref{teo:Bialgebroid}.

\begin{teo}\label{teo:QuantumCats}
Let $\mathcal{V}$ be a braided monoidal category which has all equalisers of coreflexive pairs, and in which these are preserved by tensoring with objects on either side. For a comonoid $R$ in $\mathcal{V}$ the following are equivalent,
\begin{enumerate}
\item A \emph{quantum category} over $R$ in $\mathcal{V}$ (original definition \cite[Section 12]{Day2003a}); that is a comonad on the enveloping monoidale $R\ot\otimes R$ in $\Comod(\mathcal{V})$ for which the coEilenberg-Moore object has a monoidal structure such that the forgetful arrow is strong monoidal.
\item A monoidal comonad on $R\ot\otimes R$ in $\Comod(\mathcal{V})$ \cite[Proposition 3.3]{Day2003a}.
\item A left skew monoidal structure on $R$ in $\Comod(\mathcal{V})$ such that the skew unit is coopmonadic \cite[Theorem 6.4]{Lack2012}.
\end{enumerate}
\end{teo}

Theorem~\ref{teo:QuantumCats} is the starting point of our research, but we will rather consider its dual statement and for a monoidal bicategory $\mathcal{M}$ taking the role of $\Mod(\mathcal{V})$; this is Theorem~\ref{teo:QuantumMonads} below. In this way, there is less notation and structure to keep track of during the proofs. After this switch of perspective, instead of talking about coopmonadic adjunctions we now talk about opmonadic adjunctions. An opmonadic adjunction in $\mathcal{M}$ (or adjunction of Kleisli-type) is an adjunction with a universal property; in particular, it is an initial adjunction amongst those that have the same associated monad. In the case of $\Cat$, opmonadic adjunctions are those whose left adjoint is essentially surjective on objects. In other words, the comparison functor with the Kleisli category of algebras for the associated monad is an equivalence \cite[Theorem IV.5.3]{MacLane1997}. In $\Mod(\mathcal{V})$ opmonadic adjunctions behave quite well since all adjunctions are monadic and opmonadic. For example, the unit arrow $\xymatrix@1@C=5mm{i:I\ar[r]&R}$ of a monoid $R$ in $\mathcal{V}$ induces an adjunction in $\Mod(\mathcal{V})$ as shown.
\[
\vcenter{\hbox{\xymatrix{
R\xtwocell[d]{}_{i}^{i^*}{'\dashv}\\
I
}}}
\]
The opmonadicity of this adjunction translates between two descriptions of left-$R$ right-$X$ modules: as arrows $\xymatrix@1@C=5mm{A:R\ar[r]&X}$ in $\Mod(\mathcal{V})$, or as right $X$-modules $A$ together with a left $R$-action $\xymatrix@1@C=5mm{R\otimes A\ar[r]&A}$ in $\mathcal{V}$.

Now, it is not surprising that the new ambient monoidal bicategory $\mathcal{M}$ must satisfy some mild conditions for the theorem to be true. We explicitly state these technical conditions in Section~\ref{sec:OpmonoidalOpmonadicity} and collect them under the name of \emph{opmonadic-friendly monoidal bicategories}. Informally, this is saying that opmonadic adjunctions behave well with respect to the tensor product and the composition of $\mathcal{M}$.

\begin{teo}\label{teo:QuantumMonads}
Let $\mathcal{M}$ be an opmonadic-friendly monoidal bicategory with Eilenberg-Moore objects for monads, $R$ an object with right bidual $R\ot$, and an opmonadic adjunction as shown,
\[
\vcenter{\hbox{\xymatrix{
R\ot\xtwocell[d]{}_{i\ob}^{i\ot}{'\dashv}\\
I
}}}
\]
then the following are equivalent:
\begin{enumerate}
\item A monoidale $B$ and an arrow $\xymatrix@1@C=5mm{B\ar[r]&R\ot\otimes R}$ which is monadic and strong monoidal.
\item An opmonoidal monad $\xymatrix@1@C=5mm{R\ot\otimes R\ar[r]&R\ot\otimes R}$ on the enveloping monoidale induced by the biduality.
\item A right skew monoidal structure with skew unit $\xymatrix@1@C=5mm{i:I\ar[r]&R}$ the opposite of $i\ot$.
\end{enumerate}
\end{teo}
Where (i)$\Leftrightarrow$(ii) is \cite[Proposition 3.3]{Day1997}, and (ii)$\Leftrightarrow$(iii) is \cite[Theorem 5.2]{Lack2012}.

\section{Aim and Structure}
We mentioned that bialgebroids are in bijection with monoidal structures on the category of modules over the underlying $(R\ot\otimes R)$-ring. It is natural to ask if a similar situation holds for the category of comodules. And it is true that, for a $k$-coalgebra, bialgebra structures are in bijection with monoidal structures on the category of comodules of the $k$-coalgebra. This is because it is possible to see $k$-coalgebras as coalgebras for a comonad. But in the case of bialgebroids it is only known that the category of comodules has a monoidal structure.

Now, right comodules for an $R$-bialgebroid are defined as right comodules in $\Mod$-$R$ for the underlying $R$-coring. Phùng proves in \cite[Lemma 1.4.1]{Hai2008} that these comodules bear an extra left $R$-module structure. This allows him to tensor comodules over $R$ to form a monoidal structure. This extra left $R$-module structure does not need to exist for comodules over an arbitrary $R$-coring, but it does for what is called an $R|R$-coalgebroid. Coalgebroids were defined by Takeuchi in \cite[Definition 3.5]{Takeuchi1987} in a slightly more general form; for two $k$-algebras $R$ and $S$, an $R|S$-coalgebroid is a module in $(R\ot\otimes R)$-$\Mod$-$(S\ot\otimes S)$ with some further structure which in particular includes an underlying $S$-coring. In \cite{Szlachanyi2004} Szlachányi proved that these $R|S$-coalgebroids are the arrows of a bicategory whose monads are $R$-bialgebroids. Thus, from this point of view the more complicated part in the definition of a bialgebroid rests within the coalgebroid.

This is the first of a series of papers based on the author's PhD thesis \cite{Abud2017} where we explore the theory of coalgebroids but in the generalised context of a monoidal bicategory $\mathcal{M}$. Hence what we really study are opmonoidal arrows between enveloping monoidales. Which in the case that $\mathcal{M}=\Mod(\Vect_k)$ such opmonoidal arrows are the coalgebroids of Takeuchi see Subsection~\ref{subsec:Coalgebroids}. The paper is organised in the following way: Section~\ref{sec:Background} is a small summary of all the background material and conventions taken. Section~\ref{sec:SkewMonoidalesBidualitiesAdjunctions} is a quick survey on formal category theory and what might be called formal monoidal category theory. The reader who feels comfortable working inside monoidal bicategories might go directly to Lemma~\ref{lem:OneRightSkewMonoidale}. In Section~\ref{sec:OpmonoidalOpmonadicity} and Seciton~\ref{sec:OplaxActionsOpmonadicity} we prove a theorem similar to Theorem~\ref{teo:QuantumMonads} where in place of right skew monoidales a new structure appears, we call it \emph{oplax right action}. These oplax right actions are a notion of action within the bicategory $\mathcal{M}$ with respect to a right skew monoidale, the associative and unit laws are witnessed by cells that are not necessarily invertible and satisfy further coherence conditions. Our first main theorem found below as Corollary~\ref{cor:Opmon_is_OplaxAct_Local} reads as follows.

\begin{teo}\label{teo:QuantumArrows}
Let $\mathcal{M}$ be an opmonadic-friendly autonomous monoidal bicategory and let $i\ob\dashv i\ot$ be an opmonadic adjunction as shown,
\[
\vcenter{\hbox{\xymatrix{
R\ot\xtwocell[d]{}_{i\ob}^{i\ot}{'\dashv}\\
I
}}}
\]
then the following are equivalent:
\begin{enumerate}
\item An opmonoidal arrow $\xymatrix@1@C=5mm{R\ot\otimes R\ar[r]&S\ot\otimes S}$ between enveloping monoidales.
\item An oplax right action $\xymatrix@1@C=5mm{S\otimes R\ar[r]&S}$, with respect to the skew monoidal structure on $R$ corresponding to the identity opmonoidal monad on $R\ot\otimes R$ under the equivalence in Theorem~\ref{teo:QuantumCats}.
\end{enumerate}
Furthermore, these structures have the same underlying comonad on $S$.
\end{teo}

Using this theorem we provide in Example~\ref{exa:CoalgebroidsAsOplaxActions} a simpler description of a coalgebroid in the language of classical ring and module theory. This new description requires only three module structures instead of four, none of which involve $k$-algebras with the reversed multiplication.

\begin{teo}\label{teo:Coalgebroid}
For two $k$-algebras $R$ and $S$ the following are equivalent,
\begin{enumerate}
\item An $R|S$-coalgebroid \cite[Original definition 3.5]{Takeuchi1987}.
\item A cocontinuous opmonoidal functor $\xymatrix@1@C=5mm{R\text{-}\Mod\text{-}R\ar[r]&S\text{-}\Mod\text{-}S}$.
\item A closed oplax right $(\Mod$-$R)$-actegory $\xymatrix@1@C=5mm{\Mod\text{-}S\times\Mod\text{-}R\ar[r]&\Mod\text{-}S}$.
\item A module in $(S\otimes R)$-$\Mod$-$S$ equipped with two module morphisms $\xymatrix@1@C=5mm{\delta:C\ar[r]&C\otimes_S C}$ and $\xymatrix@1@C=5mm{\varepsilon:C\ar[r]&S}$ subject to the equations given in Example~\ref{exa:CoalgebroidsAsOplaxActions}.
\end{enumerate}
\end{teo}

For the last section of the paper we generalise the notion of comodule for a coalgebroid so it can be interpreted in the context of a monoidal bicategory. We describe three different notions of comodule: one for opmonoidal arrows, one for oplax actions, and one more which is certain comodule for a comonad. This particular comonad may be found as an underlying structure for each of the items in Theorem~\ref{teo:QuantumArrows}. Now, by assuming similar conditions on the monoidal bicategory and using a similar technique as before, we show that these three notions of comodule are equivalent.

\begin{teo}
Let $\mathcal{M}$ be an opmonadic-friendly autonomous monoidal bicategory and let $i\dashv i^*$ be an opmonadic adjunction whose dual $i\ob\dashv i\ot$ is opmonadic too.
\[
\vcenter{\hbox{\xymatrix{
R\xtwocell[d]{}_{i}^{i^*}{'\dashv}\\
I
}}}
\qquad\quad
\vcenter{\hbox{\xymatrix{
R\ot\xtwocell[d]{}_{i\ob}^{i\ot}{'\dashv}\\
I
}}}
\]
Fix a structure of each item in Theorem~\ref{teo:QuantumArrows}; then the following are equivalent:
\begin{enumerate}
\item A comodule $\xymatrix@1@C=5mm{R\ar[r]&S}$ for the fixed opmonoidal arrow $\xymatrix@1@C=5mm{R\ot\otimes R\ar[r]&S\ot\otimes S}$ between enveloping monoidales.
\item A morphism of oplax right actions $\xymatrix@1@C=5mm{R\ar[r]&S}$ from $i^*1$ into the fixed oplax right action.
\item A comodule $\xymatrix@1@C=5mm{I\ar[r]&S}$ for the underlying comonad $\xymatrix@1@C=5mm{S\ar[r]&S}$ of any of the fixed items.
\end{enumerate}
\end{teo}

At the level of ring and module theory this equivalence provides us with three equivalent ways of describing comodules for a coalgebroid, depending on the notion of coalgebroid that we decide to use, see Corollary~\ref{cor:ComodulesAreComodules}. In particular, we obtain a generalisation of \cite[Lemma 1.4.1]{Hai2008} which induces the extra left $R$-module structure on a comodule for a coalgebroid that we mentioned at the beginning of this section.

\begin{teo}
For two $k$-algebras $R$ and $S$ fix a structure in each item of Theorem~\ref{teo:Coalgebroid},  the following are equivalent,
\begin{enumerate}
\item A comodule in $\Mod$-$S$ for the fixed $R|S$-coalgebroid.
\item A (cocontinuous) comodule $\xymatrix@1@C=5mm{\Mod\text{-}R\ar[r]&\Mod\text{-}S}$ for the fixed cocontinuous opmonoidal functor.
\item An (cocontinuous) oplax right $(\Mod$-$R)$-actegory oplax morphism $\xymatrix@1@C=5mm{\Mod\text{-}R\ar[r]&\Mod\text{-}S}$ into the fixed oplax $(\Mod$-$R)$-actegory structure.
\end{enumerate}
\end{teo}

We finish by showing that if an opmonoidal arrow has an opmonoidal monad structure, then its category of comodules has a monoidal structure such that the forgetful functor is strong monoidal. In particular, the category of comodules for a quantum category is monoidal.
\section{Background, Notation, and Conventions}\label{sec:Background}
We shall make some remarks about the general notation that we use throughout the paper.


\subsection{Categories}
We refer to the data of a category as: objects, arrows, composition, and identities. We reserve the word ``morphism'' for arrows that preserve some algebraic structure.

\subsection{Bicategories}
For a detailed account of 2-dimensional category theory we refer the reader to \cite{Benabou1967}, \cite{Kelly1974}, and \cite{Lack2010a}. A \emph{bicategory} $\mathcal{B}$ consists of several pieces of data that we call objects, arrows, and cells. We use the plain term ``cell'' instead of the standard term ``2-cell'' to avoid referring to two distinct cells as ``two 2-cells''. This shall cause no confusion since we do not use higher cells. All pasting diagrams and proofs are written as if $\mathcal{B}$ was a 2-category. Pasting diagrams have a unique interpretation as a cell within the bicategory $\mathcal{B}$ once we choose a convention to parenthesise its source and target arrows. Thus, we assume that its source has the leftmost bracketing and its target has the rightmost bracketing, although the reader is free to use their own favourite convention, see \cite{MacLane1985}, \cite{Power1990}, and \cite[Apendix~A]{Verity1992} for more details. Empty regions of pasting diagrams are assumed to be strictly commutative. The symbol for composition $\circ$ is mostly avoided; this forces us to write more pasting diagrams making our proofs more visual. Sometimes the isomorphism cells have a preferred direction which we depict by rotating the isomorphism symbol $\cong$ accordingly, so an isomorphism cell as below goes from $f$ to $f'$.
\[
\vcenter{\hbox{\xymatrix{
A\ar@/_3mm/[r]_{f}\ar@/^3mm/[r]^{f'}\ar@{}[r]|*=0[@u]{\cong}&B
}}}
\]

\subsection{Monoidal Bicategories}\label{ssec:2.5DimCatyTheory}

Our main universe of discourse is a monoidal bicategory, these appear in \cite{Gordon1995} and \cite{Gurski2013} as a particular case of the concept of tricategory. Tensor product of objects, arrows, and cells in a monoidal bicategory $\mathcal{M}$ is denoted by juxtaposition. Similar to the coherence theorem for bicategories, there is a coherence theorem for monoidal bicategories in \cite{Gordon1995}. This allows us to draw all pasting diagrams in a monoidal bicategory as if it was a Gray-monoid; that is: the underlying bicategory is a 2-category; the unit and the tensor product with objects in each variable are a 2-functors; the right and left unitors and the associator are identities; as well as the higher structural modifications. What remains is the interchange law between the tensor product and the horizontal composition which holds up to an isomorphism natural in $f$ and $f'$.
\[
\vcenter{\hbox{\xymatrix@!0@=15mm{
XX'\ar[r]^-{f1}\ar[d]_-{1f'}&YX'\ar[d]^-{1f'}\\
XY'\ar[r]_-{f1}\ar@{}[ru]|*=0[@]{\cong}&YY'
}}}
\]
These isomorphisms are subject to three axioms: two which assert that the collection of these isomorphism squares is closed under pairwise pasting along one edge; and a coherence axiom as pictured below.
\[
\vcenter{\hbox{\xymatrix@!0@C=15mm{
&YX'X''\ar[rd]^-{1f'1}&\\
XX'X''\ar[ru]^-{f11}\ar[dd]_-{11f''}\ar[rd]_-{1f'1}&&YY'X''\ar[dd]^-{11f''}\\
&XY'X''\ar[dd]|-{11f''}\ar[ru]^-{f11}\ar@{}[uu]|*=0[@]{\cong}&\\
XXY''\ar[rd]_-{1f'1}\ar@{}[ru]|*=0[@]{\cong}&&YY'Y''\\
&XY'Y''\ar[ru]_-{f11}\ar@{}[ruuu]|*=0[@]{\cong}&
}}}
\quad=\quad
\vcenter{\hbox{\xymatrix@!0@C=15mm{
&YX'X''\ar[rd]^-{1f'1}\ar[dd]|-{11f''}&\\
XX'X''\ar[ru]^-{f11}\ar[dd]_-{11f''}&&YY'X''\ar[dd]^-{11f''}\\
&YX'Y''\ar[rd]^-{1f'1}\ar@{}[ru]|*=0[@]{\cong}&\\
XXY''\ar[rd]_-{1f'1}\ar[ru]_-{f11}\ar@{}[ruuu]|*=0[@]{\cong}&&YY'Y''\\
&XY'Y''\ar[ru]_-{f11}\ar@{}[uu]|*=0[@]{\cong}&
}}}
\]
These axioms are used repeatedly without explicitly being recalled every time. The tensor product $ff'$ of two arrows $\xymatrix@1@C=5mm{f:X\ar[r]&Y}$, $\xymatrix@1@C=5mm{f':X'\ar[r]&Y'}$ in $\mathcal{M}$ always means the following composite $\xymatrix@1@C=5mm{ff':XX'\ar[r]_-{1f'}&XY'\ar[r]_-{f1}&YY'}$.

\begin{description}
\item[$\Cat$] The monoidal 2-category of categories, functors, and natural transformations. The monoidal product is the cartesian product and the monoidal unit is $\mathbbm{1}$ the terminal category.
\item[$\Mod(\mathcal{V})$] The monoidal bicategory $\Mod(\mathcal{V})$ of monoids, two sided modules between them, and their morphisms in $\mathcal{V}$. For $\mathcal{V}$ a symmetric monoidal category such that all coequalisers of reflexive pairs exist and are preserved by tensoring on both sides with an object. If $\mathcal{V}=\Vect_k$ we denote  $\Mod(\mathcal{V})=\Mod_k$ the monoidal bicategory of $k$-algebras, two-sided modules between them, and morphisms of two-sided modules.
\item[$\Span(\mathcal{C})$] The monoidal bicategory of spans in a category $\mathcal{C}$ with finite limits. The monoidal product is taken component-wise as the binary product in $\mathcal{C}$, the monoidal unit is the terminal object of $\mathcal{C}$.
\[
\vcenter{\hbox{\xymatrix@!0@C=15mm{
&EE'\ar[ld]_-{aa'}\ar[rd]^-{bb'}&\\
AA'&&BB'
}}}
\]
\item[$\Prof$] The monoidal bicategory of profunctors. Objects are categories, arrows $\xymatrix@1@C=5mm{\mathcal{C}\ar[r]&\mathcal{D}}$ are profunctors $\xymatrix@1@C=5mm{\mathcal{D}^{\textrm{op}}\times\mathcal{C}\ar[r]&\Set}$, and cells are morphisms of profunctors. The horizontal composition of profunctors $\xymatrix@1@C=5mm{F:\mathcal{C}\ar[r]&\mathcal{D}}$ and $\xymatrix@1@C=5mm{G:\mathcal{D}\ar[r]&\mathcal{E}}$ is given by the coend formula, $(G\circ F)(e,c):=\int^{d\in\mathcal{D}}F(d,c)\times G(e,d)$. The identity on a category $\mathcal{C}$ is the hom functor $\xymatrix@1@C=5mm{\mathcal{C}(\_,\_):\mathcal{C}^{\textrm{op}}\times\mathcal{C}\ar[r]&\Set}$. The monoidal unit is the terminal category $\mathbbm{1}$. The monoidal product is given on objects by the cartesian product of categories, while on arrows the monoidal product of two profunctors $\xymatrix@1@C=5mm{F:\mathcal{C}\ar[r]&\mathcal{D}}$ and $\xymatrix@1@C=5mm{F':\mathcal{C'}\ar[r]&\mathcal{D'}}$ is the composite below,
\[
\vcenter{\hbox{\xymatrix@=6mm{
**{!<6mm>}\mathcal{D}^{\textrm{op}}\times\mathcal{D'}^{\textrm{op}}\times\mathcal{C}\times\mathcal{C'}\ar[r]^-{1\times\mathsf{twist}\times 1}&\mathcal{D}^{\textrm{op}}\times\mathcal{C}\times\mathcal{D'}^{\textrm{op}}\times\mathcal{C'}\ar[r]^-{F\times F'}&\Set\times\Set\ar[r]&\Set
}}}
\]
where the last functor sends a pair of sets to their cartesian product.
\end{description}
\section{Review of Formal Monoidal Category Thoery}
\label{sec:SkewMonoidalesBidualitiesAdjunctions}
\subsection{Monads and Adjunctions}
Recall from \cite{Street1972} that one may define monads and adjunctions within a bicategory $\mathcal{B}$. Every adjunction $f\dashv g$ induces a monad structure on the composite arrow $\xymatrix@1{
t:S\ar[r]|-{f}&R\ar[r]|-{g}&S}$ with unit $\eta$ and multiplication as below.
\[
\vcenter{\hbox{\xymatrix@!0@C=12mm@R=5mm{
S\ar[r]^-{f}&R\ar[rd]_-{g}\ar@/^4mm/[rr]^-{1}\xtwocell[rr]{}<>{^\varepsilon}&&R\ar[r]^-{g}&S\\
&&S\ar[ru]_-{f}&&
}}}
\]

\begin{rem}
Every monad $\xymatrix@1@C=5mm{t:S\ar[r]&S}$ in a bicategory $\mathcal{B}$ induces two monads in $\Cat$ for each object $X$ in $\mathcal{B}$. These are obtained by using the covariant and contravariant hom functors based at $X$.
\[
\vcenter{\hbox{\xymatrix@R=0mm{
\mathcal{B}(S,X)\ar[r]^-{\mathcal{B}(t,X)}&\mathcal{B}(S,X)
}}}
\qquad
\vcenter{\hbox{\xymatrix@R=0mm{
\mathcal{B}(X,S)\ar[r]^-{\mathcal{B}(X,t)}&\mathcal{B}(X,S)
}}}
\]
One may consider the categories of Eilenberg-Moore algebras for each of these monads $\mathcal{B}(S,X)^{\mathcal{B}(t,X)}$ and $\mathcal{B}(X,S)^{\mathcal{B}(X,t)}$; we call their objects \emph{modules for the monad $t$ based at $X$}. An object in $\mathcal{B}(S,X)^{\mathcal{B}(t,X)}$ consists of an arrow $\xymatrix@1@C=5mm{x:S\ar[r]&X}$ together with an \emph{action cell} $\chi$ that is associative and unital with respect to the monad structure of $t$.
\[
\vcenter{\hbox{\xymatrix@!0@C=12mm@R=5mm{
S\ar[rd]_-{t}\ar@/^4mm/[rr]^-{x}\xtwocell[rr]{}<>{^\chi}&&X\\
&S\ar[ru]_-{x}&
}}}
\]
To avoid confusion when we refer to a module for the monad $t$, we specify which hom functor to use, or use the following notation $\xymatrix@1@C=5mm{(x,\chi):S\ar[r]&X}$.
\end{rem}

\begin{defi}
A \emph{Kleisli object} for a monad $\xymatrix@1@C=5mm{t:S\ar[r]&S}$ in a bicategory $\mathcal{B}$, if it exists, is the universal object $S_t$ that represents up to equivalence the modules $\xymatrix@1@C=5mm{(x,\chi):S\ar[r]&X}$ for the monad $t$ based at $X$ for every object $X$ in $\mathcal{B}$. In other words, there is an equivalence of categories as shown.
\[
\mathcal{B}(S_t,X)\simeq\mathcal{B}(S,X)^{\mathcal{B}(t,X)}
\]
\end{defi}

Concretely, a Kleisli object for a monad $t$ is an object $S_t$ and a ``universal module'' $\varphi$ for the monad $t$,
\[
\vcenter{\hbox{\xymatrix@!0@C=12mm@R=5mm{
S\ar[rd]_-{t}\ar@/^4mm/[rr]^-{f}\xtwocell[rr]{}<>{^\varphi}&&S_t\\
&S\ar[ru]_-{f}&
}}}
\]
in the sense that the equivalence in the definition is given by precomposition with $(f,\varphi)$.
\[
\vcenter{\hbox{\xymatrix@R=0mm@C=15mm{
\mathcal{B}(S_t,X)\ar[r]_-{\simeq}&\mathcal{B}(S,X)^{\mathcal{B}(t,X)}\\
}}}
\]
\[
\qquad\qquad
\vcenter{\hbox{\xymatrix{
S_t\ar[r]^-{\bar{x}}&X
}}}
\vcenter{\hbox{\xymatrix{
\ar@{|->}[r]&
}}}
\vcenter{\hbox{\xymatrix@!0@C=12mm@R=5mm{
S\ar[rd]_-{t}\ar@/^4mm/[rr]^-{f}\xtwocell[rr]{}<>{^\varphi}&&S_t\ar[r]^-{\bar{x}}&X\\
&S\ar[ru]_-{f}&&
}}}
\]
As part of the structure that comes together with a Kleisli object, there is an adjunction $\textsf{free}\dashv\textsf{forget}$ that has $t$ as its associated monad \cite[\textsection 1]{Street1972}.
\[
\vcenter{\hbox{\xymatrix{
S_t\dtwocell_{\textsf{free}\ \ }^{\ \quad\textsf{forget}}{'\dashv}\\
S\ar@(ru,rd)^-{t}
}}}
\]
And for every adjunction $f\dashv g$ whose induced monad is $t$,
\[
\vcenter{\hbox{\xymatrix{
R\dtwocell_{f}^{g}{'\dashv}\\
S\ar@(ru,rd)^-{t}
}}}
\]
there is a comparison arrow $\xymatrix@1@C=5mm{S_t\ar[r]&R}$ that commutes with the left and right adjoints up to isomorphism.
\begin{defi}
An adjunction $f\dashv g$ (or a left adjoint) in a bicategory $\mathcal{B}$ is called \emph{opmonadic} (or of \emph{Kleisli type}), if for every object $X$ in $\mathcal{B}$ the adjunction obtained by applying the representable functor $\mathcal{B}(\_,X)$ is monadic in $\Cat$ in the up to equivalence sense.
\[
\vcenter{\hbox{\xymatrix{
R\dtwocell_{f}^{g}{'\dashv}\\
S\ar@(ru,rd)^-{t}
}}}
\qquad\qquad\qquad
\vcenter{\hbox{\xymatrix{ \mathcal{B}(R,X)\dtwocell_{\mathcal{B}(g,X)\hspace{9mm}}^{\hspace{8mm}\mathcal{B}(f,X)}{'\dashv}\\
\mathcal{B}(S,X)\ar@(ru,rd)[]!<5mm,0mm>;[]!<5mm,0mm>^-{\mathcal{B}(t,X)}
}}}
\]
In other words, if $t$ is the monad associated to the adjunction $f\dashv g$, being opmonadic means that $R$ is a Kleisli object for the monad $t$.
\end{defi}

\subsection{Skew Monoidales}
Here we recall the concepts of skew monoidale \cite[Section 4]{Lack2012} and opmonoidal arrows between them which play a central role. To do this we need to upgrade our bicategory $\mathcal{B}$ to a monoidal bicategory $\mathcal{M}$. We also construct some monoidales and skew monoidales from bidualities and adjunctions in different ways, some of which are not standard. And finally we bring the example of a coalgebroid seen as an opmonoidal arrow using the concepts above.

\begin{defi}\label{def:SkewMonoidale}
A \emph{right skew monoidale} in $\mathcal{M}$ consists of an object $M$, a product arrow $\xymatrix@1@C=5mm{m:MM\ar[r]&M}$, a unit arrow $\xymatrix@1@C=5mm{u:I\ar[r]&M}$, an associator cell $\alpha$, a left unitor cell $\lambda$, and a right unitor cell $\rho$ (not necessarily invertible),
\[
\vcenter{\hbox{\xymatrix@!0@=15mm{
MMM\ar[r]^-{m1}\ar[d]_-{1m}\xtwocell[rd]{}<>{^\alpha}&MM\ar[d]^-{m}\\
MM\ar[r]_-{m}&M
}}}
\qquad
\vcenter{\hbox{\xymatrix@!0@=15mm{
M\ar[r]^-{u1}\ar[rd]_-{1}\xtwocell[rd]{}<>{^<-2>\lambda}&MM\ar[d]|-{m}&M\ar[l]_-{1u}\ar[ld]^-{1}\xtwocell[ld]{}<>{^<2>\rho}\\
&M&
}}}
\]
satisfying five axioms: in the same order as \cite[Section 4]{Lack2012} we will refer to them as, the pentagon \eqref{ax:SKM1}, the triangle \eqref{ax:SKM2}, \eqref{ax:SKM3}, \eqref{ax:SKM4}, and \eqref{ax:SKM5}.
\end{defi}
\begin{rem} When $\lambda$ or $\rho$ are invertible we speak of a left or right normal right skew monoidale; and we speak of a \emph{monoidale} when $\alpha$, $\lambda$, and $\rho$ are isomorphisms, and in this case, a well known argument by Kelly \cite{Kelly1964} implies that the axioms may be reduced from five to two: the pentagon \eqref{ax:SKM1} and the triangle \eqref{ax:SKM2}.
\end{rem}
The stereotypical example is to take $\mathcal{M}=\Cat$, where monoidales are monoidal categories. Skew monoidales in $\Cat$ are called \emph{skew monoidal categories}, these first appeared in \cite{Szlachanyi2003}. Skew monoidales appear first in \cite{Lack2012}, where the authors observe that skew monoidales in $\Span$ are categories, and skew monoidales in $\Mod_k$ are bialgebroids.
\begin{defi}
An \emph{opmonoidal arrow} $\xymatrix@1@C=5mm{C:M\ar[r]&N}$ between right skew monoidales $M$ and $N$ in $\mathcal{M}$ consists of an arrow $\xymatrix@1@C=5mm{C:M\ar[r]&N}$ in $\mathcal{M}$ equipped with an opmonoidal composition constraint cell $C^2$ and an opmonoidal unit constraint cell $C^0$ as shown below,
\[
\vcenter{\hbox{\xymatrix@!0@=15mm{
MM\ar[r]^-{CC}\ar[d]_-{m}\xtwocell[rd]{}<>{^C^2}&NN\ar[d]^-{m}\\
M\ar[r]_-{C}&N
}}}
\qquad\qquad
\vcenter{\hbox{\xymatrix@!0@R=15mm@C=10mm{
&I\ar[ld]_-{u}\ar[rd]^-{u}\xtwocell[rd]{}<>{^<3>C^0}&\\
M\ar[rr]_-{C}&&N
}}}
\]
satisfying three axioms.
\begin{align}
\tag{OM1}\label{ax:OM1}
\vcenter{\hbox{\xymatrix@!0@C=15mm{
&NNN\ar[rd]^-{m1}&\\
MMM\ar[ru]^-{CCC}\ar[dd]_-{1m}\ar[rd]_-{m1}\xtwocell[rddd]{}<>{^\alpha}\xtwocell[rr]{}<>{^C^2C\quad\ }&&NN\ar[dd]^-{m}\\
&MM\ar[dd]^-{m}\ar[ru]^-{CC}\xtwocell[rd]{}<>{^C^2}&\\
MM\ar[rd]_-{m}&&N\\
&M\ar[ru]_-{C}&
}}}
\quad&=\quad
\vcenter{\hbox{\xymatrix@!0@C=15mm{
&NNN\ar[rd]^-{m1}\ar[dd]_-{1m}\xtwocell[rddd]{}<>{^\alpha}&\\
MMM\ar[ru]^-{CCC}\ar[dd]_-{1m}\xtwocell[rd]{}<>{^CC^2\quad}&&NN\ar[dd]^-{m}\\
&NN\ar[rd]^-{m}&\\
MM\ar[rd]_-{m}\ar[ru]_-{CC}\xtwocell[rr]{}<>{^C^2\ }&&N\\
&M\ar[ru]_-{C}&
}}}
\\
\tag{OM2}\label{ax:OM2}
\vcenter{\hbox{\xymatrix@!0@C=15mm{
&N\ar@/^7mm/[rddd]^-{1}&\\
M\ar[ru]^-{C}\ar[dd]_-{1u}\xtwocell[rddd]{}<>{^\rho}\ar@/^7mm/[rddd]^-{1}&&\\
&&\\
MM\ar[rd]_-{m}&&N\\
&M\ar[ru]_-{C}&
}}}
\quad&=\quad
\vcenter{\hbox{\xymatrix@!0@C=15mm{
&N\ar[dd]_-{1u}\xtwocell[rddd]{}<>{^\rho}\ar@/^7mm/[rddd]^-{1}&\\
M\ar[ru]^-{C}\ar[dd]_-{1u}\xtwocell[rd]{}<>{^CC^0\quad}&&\\
&NN\ar[rd]^-{m}&\\
MM\ar[rd]_-{m}\ar[ru]_-{CC}\xtwocell[rr]{}<>{^C^2\ }&&N\\
&M\ar[ru]_-{C}&
}}}
\\
\tag{OM3}\label{ax:OM3}
\vcenter{\hbox{\xymatrix@!0@C=15mm{
&N\ar[rd]^-{u1}&\\
M\ar[ru]^-{C}\ar[rd]_-{u1}\ar@/_7mm/[rddd]_-{1}\xtwocell[rddd]{}<>{^\lambda\ }\xtwocell[rr]{}<>{^C^0C\quad}&&NN\ar[dd]^-{m}\\
&MM\ar[dd]^-{m}\ar[ru]^-{CC}\xtwocell[rd]{}<>{^C^2}&\\
&&N\\
&M\ar[ru]_-{C}&
}}}
\quad&=\quad
\vcenter{\hbox{\xymatrix@!0@C=15mm{
&N\ar[rd]^-{u1}\ar@/_7mm/[rddd]_-{1}\xtwocell[rddd]{}<>{^\lambda\ }&\\
M\ar[ru]^-{C}\ar@/_7mm/[rddd]_-{1}&&NN\ar[dd]^-{m}\\
&&\\
&&N\\
&M\ar[ru]_-{C}&
}}}
\end{align}
\end{defi}
\begin{rem}
In the case that both opmonoidal constraints are isomorphisms speak of a \emph{strong monoidal arrow}, and if they are identities we speak of a \emph{strict monoidal arrow}.
\end{rem}
\begin{defi}
An \emph{opmonoidal cell} between a parallel pair of opmonoidal arrows $C$ and $\xymatrix@1@C=5mm{C':M\ar[r]&N}$ in $\mathcal{M}$ consists of a cell $\xi$ as shown,
\[
\vcenter{\hbox{\xymatrix{
M\ar@/_3mm/[r]_-{C}\ar@/^3mm/[r]^-{C'}\xtwocell[r]{}<>{^\xi}&N
}}}
\]
satisfying two axioms.
\begin{align}
\tag{OM4}\label{ax:OM4}
\vcenter{\hbox{\xymatrix@!0@R=16mm@C=19mm{
MM\ar@/^3mm/[r]^-{C'C'}\ar[d]_-{m}\xtwocell[rd]{!<-2mm,2mm>}<>{^<-1>C'^2}&NN\ar[d]^-{m}\\
M\ar@/_3mm/[r]_-{C}\ar@/^3mm/[r]^-{C'}\xtwocell[r]{}<>{^\xi}&N
}}}
\quad&=\quad
\vcenter{\hbox{\xymatrix@!0@R=16mm@C=19mm{
MM\ar@/_3mm/[r]_-{CC}\ar@/^3mm/[r]^-{C'C'}\xtwocell[r]{}<>{^\xi\xi\ }\ar[d]_-{m}\xtwocell[rd]{!<2mm,-2mm>}<>{^<1>C^2}&NN\ar[d]^-{m}\\
M\ar@/_3mm/[r]_-{C}&N
}}}
\\
\tag{OM5}\label{ax:OM5}
\vcenter{\hbox{\xymatrix@!0@R=19mm@C=12mm{
{\xtwocell[rrd]{}<>{^C'^0}}&I\ar[ld]_-{u}\ar[rd]^-{u}&\\
M\ar@/_3mm/[rr]_-{C}\ar@/^3mm/[rr]^-{C'}\xtwocell[rr]{}<>{^\xi}&&N
}}}
\quad&=\quad
\vcenter{\hbox{\xymatrix@!0@R=19mm@C=12mm{
&I\ar[ld]_-{u}\ar[rd]^-{u}\xtwocell[rd]{}<>{^<4>C^0}&\\
M\ar@/_3mm/[rr]_-{C}&&N
}}}
\end{align}
\end{defi}

\begin{rem}
As usual, the process of taking dual bicategories does the job of switching between these notions, we name all of them as follows:
\begin{itemize}
\item $\SkOpmon_{\text{r}}(\mathcal{M})=\SkOpmon(\mathcal{M})$ is the bicategory of right skew monoidales, opmonoidal arrows, and opmonoidal cells between them. This bicategory is going to be used the most throughout the document, and in order to have a lighter notation we omit the subscript.
\item $\SkOpmon_{\text{l}}(\mathcal{M})=\SkOpmon_{\text{r}}(\mathcal{M}^{\text{rev}})$ is the bicategory of left skew monoidales, opmonoidal arrows, and opmonoidal cells between them.
\item $\SkMon_{\text{l}}(\mathcal{M})=\SkOpmon_{\text{r}}(\mathcal{M}^{\text{co}})^{\text{co}}$
is the bicategory of left skew monoidales, monoidal arrows, and monoidal cells between them.
\item $\SkMon_{\text{r}}(\mathcal{M})=\SkOpmon_{\text{r}}(\mathcal{M}^{\text{rev co}})^{\text{co}}$ is the bicategory of right skew monoidales, monoidal arrows, and monoidal cells between them.
\item $\Opmon(\mathcal{M})$ is the bicategory of monoidales, opmonoidal arrows, and opmonoidal cells between them, which can be seen as the full subbicategory of $\SkOpmon_{\text{r}}(\mathcal{M})$ whose objects are monoidales.
\end{itemize}
When there is no room for ambiguity, we omit the ambient monoidal bicategory $\mathcal{M}$ from the hom categories and write them as $\SkOpmon(M,N)$ and $\Opmon(M,N)$.
\end{rem}

In the following lemma we construct right skew monoidales from adjunctions whose left adjoint has domain $I$.
\begin{lem}\label{lem:OneRightSkewMonoidale}
For every adjunction as shown below in a monoidal bicategory $\mathcal{M}$, there is a right skew monoidal structure on $R$.
\[
\vcenter{\hbox{\xymatrix{
R\dtwocell_{i}^{i^*}{'\dashv}\\
I
}}}
\]
\end{lem}
\begin{proof}

The structure is given as follows:
\begin{center}
\begin{tabular}{rlrl}
\centering
Product&
$
\xymatrix{
RR\ar[r]^-{i^*1}&R
}
$&
Unit&
$
\xymatrix{
I\ar[r]^-{i}&R
}$
\\
Associator&
$
\vcenter{\hbox{\xymatrix@!0@=15mm{
RRR\ar[r]^-{i^*11}\ar[d]_-{1i^*1}&RR\ar[d]^-{i^*1}\\
RR\ar[r]_-{i^*1}\ar@{}[ru]|-*[@ru]{\cong}&R
}}}
$&
Unitors&
$
\vcenter{\hbox{\xymatrix@!0{
R\ar[rr]^-{i1}\ar[rrdd]_-{1}\xtwocell[rrdd]{}<>{^<-2>\eta1}&&RR\ar[dd]|-{i^*1}\ar@{}[rd]|-<<<<*[@rd]{\cong}&&R\ar[ll]_-{1i}\ar[ld]^-{i^*}\ar@/^1cm/[lldd]^-{1}\xtwocell[lldd]{}<>{^<-3>\varepsilon}\\
&&&I\ar[ld]^-{i}&\\
&&R&&
}}}
$
\end{tabular}
\end{center}
Note that the associator is an instance of an interchange isomorphism, therefore the pentagon \eqref{ax:SKM1} holds as an instance of the interchange coherence. The $\alpha$-$\lambda$ compatibility \eqref{ax:SKM3} holds by naturality of the interchanger. The $\alpha$-$\rho$ compatibility \eqref{ax:SKM4} is also an instance of the naturality of the interchanger (regardless of the definition of $\rho$). The $\alpha$-$\lambda$-$\rho$ compatibility \eqref{ax:SKM2} and the $\lambda$-$\rho$ compatibility \eqref{ax:SKM5} are a consequence of the snake equations of the adjunction $i\dashv i^*$.
\[
\vcenter{\hbox{\xymatrix@!0@C=15mm{
&&\\
RR\ar[rd]_-{1i1}\ar@/_7mm/[rddd]_-{1}\ar@/^7mm/[rr]^-{1}\ar[r]^-{i^*1}\xtwocell[rddd]{}<>{^1\eta 1\ }\xtwocell[rr]{}<>{^<-2>\varepsilon 1\ }&I\ar[r]^-{i1}&RR\ar[dd]^-{i^*1}\\
&RRR\ar[dd]^-{1i^*1}\ar[ru]_-{i^*11}\ar@{}[u]|*[@]{\cong}&\\
&&R\\
&RR\ar[ru]_-{i^*1}\ar@{}[ruuu]|*[@]{\cong}&
}}}
=
\vcenter{\hbox{\xymatrix@!0@C=15mm{
&&\\
RR\ar@/_7mm/[rddd]_-{1}\ar@/^7mm/[rr]^-{1}\ar[r]^-{i^*1}\xtwocell[rr]{}<>{^<-2>\varepsilon 1\ }&I\ar[r]^-{i1}\ar@/_4mm/[rdd]_-1\xtwocell[rdd]{}<>{^\eta 1\ }&RR\ar[dd]^-{i^*1}\\
&&\\
&&R\\
&RR\ar[ru]_-{i^*1}&
}}}
=\!\!
\vcenter{\hbox{\xymatrix@!0@C=15mm{
&&\\
RR\ar@/_7mm/[rrdd]_-{i^*1}\ar@/^7mm/[rrdd]^-{i^*1}&&\\
&&\\
&&R\\
&&
}}}
\]
\[
\vcenter{\hbox{\xymatrix@!0@C=15mm{
&&\\
I\ar@/_7mm/[rrdd]_-{i}\ar@/^7mm/[rrdd]^-{i}&&\\
&&\\
&&R\\
&&
}}}
\!=
\vcenter{\hbox{\xymatrix@!0@C=14mm{
&R\ar@/^15mm/[rddd]^-{1}\ar[rd]^-{i^*}\xtwocell[rddd]{}<>{^<-7>\varepsilon}&\\
I\ar[ru]^-{i}\ar[dd]_-{i}\ar@/_7mm/[rr]_-{1}\xtwocell[rr]{}<>{^\eta\ }&&I\ar[dd]^-{i}\\
&&\\
R\ar@/_7mm/[rr]_-{1}&&R\\
&&
}}}
\!\!=
\vcenter{\hbox{\xymatrix@!0@C=15mm{
&R\ar[dd]_-{1i}\ar@/^15mm/[rddd]^-{1}\ar[rd]^-{i^*}\xtwocell[rddd]{}<>{^<-7>\varepsilon}&\\
I\ar[ru]^-{i}\ar[dd]_-{i}&&I\ar[dd]^-{i}\\
&RR\ar[rd]^-{i^*1}\ar@{}[ru]|*[@]{\cong}&\\
R\ar[ru]_-{i1}\ar@/_7mm/[rr]_-{1}\ar@{}[ruuu]|*[@]{\cong}\xtwocell[rr]{}<>{^\eta 1\ }&&R\\
&&
}}}
\]
\end{proof}
\begin{eje}
We can now provide various examples of skew monoidal structures.
\begin{itemize}
\item In the case that $\mathcal{M}=\Cat$ a left adjoint $\xymatrix@1@C=5mm{i:1\ar[r]&R}$ is the same as an initial object $i$ in $R$. The right skew monoidal structure $\odot$ on $R$ induced by $i$ is given by the second projection $a\odot b=b$, it is strictly associative and left unital. The right unitor is the unique arrow $\xymatrix@1@C=5mm{a\odot i=i\ar[r]&a}$ in $R$.
\item When $\mathcal{M}$ is a locally discrete monoidal bicategory, in other words, a monoidal category regarded as a monoidal bicategory, the only example is $I$ itself since adjunctions in $\mathcal{M}$ are the isomorphisms.
\item In the case of $\mathcal{M}=\Mod_k$ an adjunction as in the Lemma~\ref{lem:OneRightSkewMonoidale} amounts to a finitely generated and projective module $P$ in $\Mod$-$R$ \cite[Section 5]{Street2007}. This module $P$ induces a right skew monoidal structure $\odot_P$ on $\Mod$-$R$. The case where $P=R_R$ is simple yet illuminating; the skew monoidal product on $\Mod$-$R$ is $A\odot_R B:=A\otimes_RR\otimes B\cong A\otimes B$, the unit is $R$, the associator is an invertible transformation, and the unitors are given below.
\[
\vcenter{\hbox{\xymatrix@R=0mm{
\lambda:B\ar[r]^-{\eta\otimes 1}&R\otimes B\cong R\underset{R}\odot B\\
b\ar@{|->}[r]&1\otimes b
}}}
\qquad
\vcenter{\hbox{\xymatrix@R=0mm{
\rho:A\underset{R}\odot R\cong A\otimes R\ar[r]^-{A\underset{R}\otimes\varepsilon}&A\\
a\otimes r\ar@{|->}[r]&ar
}}}
\]
Furthermore, under Szlachányi's equivalence, this skew monoidal category corresponds to the simplest $R$-bialgebroid $B=R\ot\otimes R$ \cite[Example 3.2.3.]{Bohm2009}.

\subitem In the general case, the skew monoidal product $\odot_P$ on $\Mod$-$R$ is
$A\odot_P B:=A\otimes_RP^*\otimes B$, the skew unit is $P$, the associator is an invertible transformation, and the unitors are given below.
\[
\vcenter{\hbox{\xymatrix@R=0mm{
\lambda:B\ar[r]^-{\eta\otimes 1}&P{\underset{R}\otimes}P^*\otimes B\cong P\underset{P}\odot B
}}}
\qquad
\vcenter{\hbox{\xymatrix@R=0mm{
\rho:A\underset{P}\odot P\cong A{\underset{R}\otimes}P^*\otimes P\ar[r]_-{A\underset{R}\otimes\varepsilon}&A
}}}
\]
And if $P\neq R$, then the skew monoidal structure $\odot_P$ does not correspond to an $R$-bialgebroid under \cite[Theorem 9.1]{Szlachanyi2012} since a necessary condition is that the skew unit is equal to $R_R$.
\end{itemize}
\end{eje}

\subsection{Bidualities}
\begin{defi}
A \emph{right bidual} of an object $R$ of $\mathcal{M}$ is an object $R\ot$ equipped with a unit $\xymatrix@1@C=5mm{n:I\ar[r]&R\ot R}$ and a counit $\xymatrix@1@C=5mm{e:RR\ot\ar[r]&I}$ arrows, and two cells $\varsigma_l$ and $\varsigma_r$ called left and right triangle (or snake) isomorphisms,
\[
\vcenter{\hbox{\xymatrix@!0@=18mm{
R\ot\ar[r]^-{n1}\ar[rd]_-{1}&R\ot RR\ot\ar[d]^-{1e}\\
\ar@{}[ru]|>>>>>>*[@ru]{\cong}^>>>>>>{\varsigma_l}&R\ot\\
}}}
\qquad\qquad\qquad\qquad
\vcenter{\hbox{\xymatrix@!0@=18mm{
RR\ot R\ar[d]_-{e1}\ar@{}[rd]|<<<<<<*[@]{\cong}^<<<<<<{\varsigma_r}&R\ar[l]_-{1n}\ar[ld]^-{1}\\
R&
}}}
\]
satisfying the swallowtail equations below.
\begin{align*}
\vcenter{\hbox{\xymatrix@!0@C=15mm@R=8.45mm{
&&\\
RR\ot\ar[rd]_-{1n1}\ar@/_7mm/[rddd]_-{1}\ar@/^7mm/[rr]^-{1}&&RR\ot\ar[dd]^-{e}\\
&RR\ot RR\ot\ar[dd]^-{11e}\ar[ru]^-{e11}\ar@{}[uu]|*[@]{\cong}^-{\varsigma_r1\ }&\\
\ar@{}[ru]|*[@]{\cong}^-{1\varsigma_l}&&I\\
&RR\ot\ar[ru]_-{e}\ar@{}[ruuu]|*[@]{\cong}&
}}}
\quad&=\quad
\vcenter{\hbox{\xymatrix@!0@C=15mm{
&&\\
RR\ot\ar@/_7mm/[rrdd]_-{e}\ar@/^7mm/[rrdd]^-{e}&&\\
&&\\
&&I\\
&&
}}}\\
\vcenter{\hbox{\xymatrix@!0@C=15mm{
&&\\
I\ar@/_7mm/[rrdd]_-{n}\ar@/^7mm/[rrdd]^-{n}&&\\
&&\\
&&R\ot R\\
&&
}}}
\quad&=\quad
\vcenter{\hbox{\xymatrix@!0@C=15mm{
&R\ot R\ar[dd]_-{11n}\ar@/^7mm/[rddd]^-{1}&\\
I\ar[ru]^-{n}\ar[dd]_-{n}&&\\
&R\ot RR\ot R\ar[rd]^-{1e1}\ar@{}[ru]|*[@]{\cong}^-{1\varsigma_r\ }&\\
R\ot R\ar[ru]_-{n11}\ar@/_7mm/[rr]_-{1}\ar@{}[ruuu]|*[@]{\cong}&&R\ot R\\
&\ar@{}[uu]|*[@]{\cong}^-{\varsigma_l1\:}&
}}}
\end{align*}
This situation is denoted by $R\dashv R\ot$ and called \emph{a biduality} in $\mathcal{M}$. \emph{Left biduals} are defined as right biduals in $\mathcal{M}^{\text{op rev}}$. A monoidal bicategory $\mathcal{M}$ that has right biduals for every object is called \emph{right autonomous} (or \emph{right rigid}); if instead $\mathcal{M}$ has left biduals it is called \emph{left autonomous}, and if it has both left and right biduals $\mathcal{M}$ is called \emph{autonomous}.
\end{defi}

\begin{eje} This is what bidualities look like in our prototypical monoidal bicategories.
\begin{itemize}
\item In $\Mod_k$ for a commutative ring $k$, the bidual of a $k$-algebra $R$ is the opposite algebra $R\ot$, which has the same underlying $k$-vector space but the reverse multiplication.
\item In $\Span^{\text{co}}$ every set is self-bidual, the unit and counit of the biduality are constructed with the unique comonoid structure (duplicate/discard) that every set has.
\item In $\mathcal{V}$-$\Prof$ for a symmetric monoidal category $\mathcal{V}$, the bidual of a $\mathcal{V}$-category $\mathcal{A}$ is the opposite category $\mathcal{A}^{\text{op}}$.
\item In $\Cat$ with the cartesian product (and in fact in any cartesian monoidal bicategory) a biduality is far too restrictive, because for a biduality $\mathcal{C}\dashv\mathcal{C}\ot$ to exist both categories $\mathcal{C}$ and $\mathcal{C}\ot$ have to be equivalent to the terminal category.
\end{itemize}
\end{eje}

\begin{rem}\label{rem:PantsMonoidale}
Right biduals are unique up to equivalence and a biduality $R\dashv R\ot$ induces a monoidale $R\ot R$ with the structure below. We call \emph{enveloping monoidales} those monoidales that arise from bidualities.
\begin{center}\begin{tabular}{cc}
\centering
Product&
Unit\\
$
\vcenter{\hbox{\xymatrix{
R\ot RR\ot R\ar[r]^-{1e1}&R\ot R
}}}
$&
$
\vcenter{\hbox{\xymatrix@1@C=5mm{I\ar[r]^-{n}&R\ot R
}}}
$\\
&
\\
Associator&
Left and right unitors\\
$
\vcenter{\hbox{\xymatrix@!0@=15mm{
**[l]R\ot RR\ot RR\ot R\ar[r]^-{1e11}\ar[d]_-{11e1}&**[r]R\ot RR\ot R\ar[d]^-{1e1}\\
**[l]R\ot RR\ot R\ar[r]_-{1e1}\ar@{}[ru]|*[@ru]{\cong}&R\ot R
}}}
$&
$
\vcenter{\hbox{\xymatrix@!0@=15mm{
**[l]R\ot R\ar[r]^-{n1}\ar[rd]_-{1}&R\ot RR\ot R\ar[d]|-{1e1}\ar@{}[rd]|<<<<<*[@rd]{\cong}^<<<<<{1\varsigma_r}&**[r]R\ot R\ar[l]_-{1n}\ar[ld]^-{1}\\
\ar@{}[ru]|>>>>>*[@ru]{\cong}^>>>>>{\varsigma_l1}&R\ot R&
}}}
$
\end{tabular}\end{center}
The pentagon axiom~\eqref{ax:SKM1} is an instance of the coherence of the interchange law in $\mathcal{M}$, and the triangle axiom~\eqref{ax:SKM2} is exactly one of the swallowtail equations of the biduality.
\end{rem}

\begin{rem}\label{rem:EndoHomMonoidale}
It is possible to generalise the fact that an adjunction $F\dashv G$ in $\Cat$ is also given by a pair of functors $F$, $G$ and a natural isomorphism of sets $\hom(Fx,y)\cong\hom(x,Gy)$ to the case of biadjunctions in a tricategory (see \cite[Example 1.1.7]{Verity1992} for biadjunctions between bicategories). Given a biduality $R\dashv R\ot$ in a monoidal bicategory $\mathcal{M}$ there exists an adjoint equivalence $\mathcal{M}(RX,Y)\simeq\mathcal{M}(X,R\ot Y)$ equipped with some other structure that satisfies some coherence equations. Hence, every autonomous monoidal bicategory $\mathcal{M}$ is a right closed monoidal bicategory \cite[Section 2]{Day1997} with right internal hom given by $[X,Y]:=X\ot Y$. This allows us to think of the enveloping monoidale $R\ot R$ as the endo-hom monoidale.

An opmonoidal arrow whose source is the monoidal unit $I$ may be called an \emph{internal comonoid} of the target skew monoidale. If $\mathcal{M}=\Cat$, internal comonoids of a (skew) monoidal category are precisely comonoids in the (skew) monoidal category. Now, even if the monoidal bicategory $\mathcal{M}$ is not right closed monoidal or autonomous, for an object $R$ with a bidual $R\ot$ we may still talk about internal comonoids of the enveloping monoidale $R\ot R$. These are opmonoidal arrows $\xymatrix@1@C=5mm{I\ar[r]&R\ot R}$, and it is not hard to see that under transposition internal comonoids of $R\ot R$ correspond to comonads on $R$. We will continue discussing these concepts in Remark~\ref{rem:ComonadsAreIActions} and Remark~\ref{rem:InducedComonad} below.
\end{rem}
\begin{lem}\label{lem:LocalEquivalence}
For every two bidualities $R\dashv R\ot$ and $S\dashv S\ot$ in $\mathcal{M}$ there is an adjoint equivalence of categories
\[
\mathcal{M}(R,S)\simeq\mathcal{M}(S\ot,R\ot).
\]
More generally,
\[
\mathcal{M}(RX,YS)\simeq\mathcal{M}(XS\ot,R\ot Y).
\]
\end{lem}\qed

\begin{rem}\label{rem:DualsAsALocalEquivalence}
If $\mathcal{M}$ is right autonomous, the axiom of choice allows us to choose a bidual $R\ot$ for every object $R$, thus the equivalence of Lemma~\ref{lem:LocalEquivalence} gives rise to a strong monoidal pseudofunctor in the up to equivalence sense since $(XY)\ot\simeq Y\ot X\ot$.
\[
\vcenter{\hbox{\xymatrix{
\mathcal{M}^{\text{rev op}}\ar[r]^-{(\ )\ot}&\mathcal{M}
}}}
\]
This pseudofunctor is also locally an equivalence in the sense that, for every pair of objects, its action on homs is an equivalence. Furthermore, if $\mathcal{M}$ is autonomous $(\ )\ot$ is a strong monoidal biequivalence of monoidal bicategories \cite[1.33 for definition]{Street1980}. Its pseudoinverse takes an object to its chosen left bidual, thus $(\ )\ot$ is essentially surjective on objects (in the up to equivalence sense) by the existence of left biduals and the uniqueness up to equivalence of right biduals. This appears first in \cite[Section 2]{Day1997} but the authors forget to mention left autonomy. This biequivalence allows us to transpose many structures without losing information, for example, adjunctions.
\end{rem}

\begin{lem}
For every two bidualities $S\dashv S\ot$ and $R\dashv R\ot$ in $\mathcal{M}$, adjunctions $f_*\dashv f^*:\xymatrix@1@C=5mm{S\ar[r]&R}$ are in correspondence with adjunctions $f\ob\dashv f\ot:\xymatrix@1@C=5mm{S\ot\ar[r]&R\ot}$.
\end{lem}\qed

The adjunction $f\ob\dashv f\ot$ is called the \emph{opposite or mate} adjunction of $f_*\dashv f^*$. In what follows adjunctions where $S=I$ and their opposites are constantly used, so we spell out the opposite adjunction to have at hand for future calculations.
\[
\vcenter{\hbox{\xymatrix{
R\dtwocell_{i}^{i^*}{'\dashv}\\
I
}}}
\qquad
\vcenter{\hbox{\xymatrix@R=5mm{
i\ob:I\ar[r]^-{n}&R\ot R\ar[r]^-{1i^*}&R\ot\\
i\ot:R\ot\ar[r]^-{i1}&RR\ot\ar[r]^-{e}&I
}}}
\qquad\vcenter{\hbox{\xymatrix{
R\ot\dtwocell_{i\ob}^{i\ot}{'\dashv}\\
I
}}}
\]
Note that the associated monad of $i\ob\dashv i\ot$ has the same underlying arrow as the monad for $i\dashv i^*$ up to isomorphism, but the multiplication is the opposite one.

\subsection{Coalgebroids}\label{subsec:Coalgebroids}
We close this section with the example that gave this paper its name. We use Sweedler's notation $\delta(c)=\sum c_{(1)}\otimes c_{(2)}$ for the image of an element $c\in C$ under a morphism of modules $\xymatrix@1@C=5mm{\delta:C\ar[r]&C\otimes C}$. Recall the definition of an $R|S$-coalgebroid (\cite[Definition 1.1]{Szlachanyi2004} or \cite[pp. 185]{Bohm2009}) which first appeared under the name $R|S$-coring in \cite[Definition 3.5]{Takeuchi1987}.
\begin{defi}\label{def:coalgebroid}
Let $R$ and $S$ be $k$-algebras for a commutative ring $k$. An \emph{$R|S$-coalgebroid} consists of a module $C$ in $RS$-$\Mod$-$RS$, a morphism called comultiplication $\xymatrix@1@C=5mm{\delta:C\ar[r]&C\otimes_S C}$ in $RS$-$\Mod$-$RS$ in which $C\otimes_SC$ uses the two-sided $R$-module structure given by $r.(c\otimes c').r'=cr'\otimes rc'$, that is
\begin{enumerate}
\item $\delta(scs')=\sum sc_{(1)}\otimes c_{(2)}s'$
\item $\delta(rcr')=\sum c_{(1)}r'\otimes rc_{(2)}$
\end{enumerate}
and a morphism called counit $\xymatrix@1@C=5mm{\varepsilon:C\ar[r]&S}$ in $S$-$\Mod$-$S$, that is
\begin{enumerate}\setcounter{enumi}{2}
\item $\varepsilon(scs')=s\varepsilon(c)s'$
\end{enumerate}
subject to the following axioms.
\begin{enumerate}\setcounter{enumi}{3}
\item $\sum rc_{(1)}\otimes c_{(2)}=\sum c_{(1)}\otimes c_{(2)}r$
\item $\varepsilon(rc)=\varepsilon(cr)$
\item $(C,\varepsilon,\delta)$ forms a comonoid in the monoidal category $S$-$\Mod$-$S$
\end{enumerate}
Note that axiom (iv) may be rewritten using the two-sided $R$-action on $C\otimes_SC$ given by $r\mathord{\cdot}(c\otimes c')\mathord{\cdot}r'=(rc)\otimes(c'r')$, which is different than the action used in (ii).
\begin{itemize}
\item[(iv')] $r\mathord{\cdot}\delta(c)=\delta(c)\mathord{\cdot}r$, the image of the comultiplication $\delta$ is in the $R$-centralizer of $C\otimes_SC$.
\end{itemize}
\end{defi}
According to \cite{Takeuchi1987}, \cite{Szlachanyi2004}, or \cite{Hai2008}, conditions (iv) and (v) are logically equivalent. In Example~\ref{exa:CoalgebroidsAsOplaxActions} at the end of Section~\ref{sec:OplaxActionsOpmonadicity}, we give another equivalent and simpler definition of a coalgebroid by using the tool developed in that section: oplax actions.

It is immediate from the definition of an $R|S$-coalgebroid that if $R=k$ then conditions (ii), (iv), and (v) are trivial, thus a $k|S$-coalgebroid is nothing but a comonoid in $S$-$\Mod$-$S$, i.e. an $S$-coring. Going up one dimension, since $S$-$\Mod$-$S$ is a hom category of the monoidal bicategory $\Mod_k$, then an $S$-coring is a comonad in $\Mod_k$ on $S$. And, as mentioned in Remark~\ref{rem:EndoHomMonoidale}, comonads correspond to opmonoidal arrows by transposition, which implies that $k|S$-coalgebroids correspond to opmonoidal arrows $\xymatrix@1@C=5mm{k\ar[r]&S\ot S}$.

In fact, all $R|S$-coalgebroids \emph{are} opmonoidal arrows; the following lemma is the behaviour on objects of an isomorphism of bicategories between the full subbicategory $\Opmon^{\mathrm{e}}(\Mod_k)$ of $\Opmon(\Mod_k)$ on the enveloping monoidales in $\Mod_k$, and the bicategory $\Cgb_k$, defined in \cite{Szlachanyi2004}, whose objects are $k$-algebras and arrows $\xymatrix@1@C=5mm{R\ar[r]&S}$ are $R|S$-coalgebroids.
\[
\Opmon^{\mathrm{e}}(\Mod_k)\cong\Cgb_k
\]
In the proof, there are modules that have more than two actions with respect to the same $k$-algebra as well as tensor products of these modules over one or more of these actions. To avoid confusion, we use coloured $k$-algebras as subscripts for modules and tensor products to distinguish which actions are being used while tensoring. For example, ${_{\textcolor{Red}{R}}M_{\textcolor{ForestGreen}{S}}}$ is a module in $\textcolor{Red}{R}$-$\Mod$-$\textcolor{ForestGreen}{S}$, and with another module ${_{\textcolor{ForestGreen}{S}}N_{\textcolor{BlueViolet}{T}}}$ we can form the tensor product ${_{\textcolor{Red}{R}}M_{\textcolor{ForestGreen}{S}}}\underset{\textcolor{ForestGreen}{S}}\otimes{_{\textcolor{ForestGreen}{S}}N_{\textcolor{BlueViolet}{T}}}$ to get a module ${_{\textcolor{Red}{R}}L_{\textcolor{BlueViolet}{T}}}$.
\begin{lem}\label{lem:CoalgebroidsAsQuantum}
For a commutative ring $k$, opmonoidal arrows in the bicategory $\Mod_k$ of the form $\xymatrix@1@C=5mm{C:R\ot R\ar[r]&S\ot S}$ are $R|S$-coalgebroids.
\end{lem}
\begin{proof}

The isomorphism $R\ot R$-$\Mod$-$S\ot S\cong RS$-$\Mod$-$RS$ is used without changing the name of the modules. Let $C$ be an opmonoidal arrow as in the statement. One may rewrite the structure cell $C^0$ in the language of the category $\Mod$-$\textcolor{Red}{S\ot S}$ instead of the language of the monoidal bicategory $\Mod_k$. Both notations are shown below.
\[
\vcenter{\hbox{\xymatrix@!0@R=15mm@C=10mm{
&k\ar[ld]_-{n}\ar[rd]^-{n}\xtwocell[rd]{}<>{^<3>C^0}&\\
\textcolor{ForestGreen}{R\ot R}\ar[rr]_-{C}&&\textcolor{Red}{S\ot S}
}}}
\qquad\quad
\vcenter{\hbox{\xymatrix{
R_{\textcolor{ForestGreen}{R\ot R}}
\underset{\textcolor{ForestGreen}{R\ot R}}{\otimes}
{_{\textcolor{ForestGreen}{R\ot R}}C}_{\textcolor{Red}{S\ot S}}
\ar[r]^-{C^0}&
S_{\textcolor{Red}{S\ot S}}
}}}
\]
And module morphisms $C^0$ are in bijective correspondence with morphisms $\xymatrix@1@C=5mm{\varepsilon:C\ar[r]&S}$ in $\textcolor{Red}{S}$-$\Mod$-$\textcolor{Red}{S}$ for which the condition (v) $\varepsilon(\textcolor{ForestGreen}{r}c)=\varepsilon(c\textcolor{ForestGreen}{r})$ is satisfied. Now, one needs to be more careful with the structure cell $C^2$ as there are several $R$-actions which may be confusing. Here is where the colours are most helpful; $C^2$ is a cell in $\Mod_k$ as follows.
\[
\vcenter{\hbox{\xymatrix@!0@=15mm{
**{!<3mm>}\textcolor{Red}{R\ot\textcolor{ForestGreen}{RR\ot}R}\ar[r]^-{CC}\ar[d]_-{1e1}\xtwocell[rd]{}<>{^C^2}&**{!<-3mm>}\textcolor{olive}{S\ot\textcolor{BlueViolet}{SS\ot}S}\ar[d]^-{1e1}\\
\textcolor{BlueViolet}{R\ot R}\ar[r]_-{C}&\textcolor{Red}{S\ot S}
}}}
\]
But now, one may rewrite it in the language of $\textcolor{Red}{R\ot\textcolor{ForestGreen}{RR\ot}R}$-$\Mod$-$\textcolor{Red}{S\ot S}$, hence $C^2$ is a module morphism with source and target as shown below.
\[
\vcenter{\hbox{\xymatrix@!0@C=50mm@R=15mm{
({_{\textcolor{Red}{R\ot}}R_{\textcolor{BlueViolet}{R\ot}}}
\otimes
{_{\textcolor{ForestGreen}{RR\ot}}R}
\otimes
{_{\textcolor{Red}{R}}R_{\textcolor{BlueViolet}{R}}})
\underset{\textcolor{BlueViolet}{R\ot R}}\otimes
{_{\textcolor{BlueViolet}{R\ot R}}C_{\textcolor{Red}{S\ot S}}}
\ar[rd]^-{C^2}&\\
&({_{\textcolor{Red}{R\ot}\textcolor{ForestGreen}{R}}C_{\textcolor{olive}{S\ot}\textcolor{BlueViolet}{S}}}
\otimes
{_{\textcolor{ForestGreen}{R\ot}\textcolor{Red}{R}}C_{\textcolor{BlueViolet}{S\ot}\textcolor{olive}{S}}})
\underset{\textcolor{olive}{S\ot\textcolor{BlueViolet}{SS\ot}S}}\otimes
({_{\textcolor{olive}{S\ot}}S_{\textcolor{Red}{S\ot}}}
\otimes
{_{\textcolor{BlueViolet}{SS\ot}}S}
\otimes
{_{\textcolor{olive}{S}}S_{\textcolor{Red}{S}}})
}}}
\]
The source may be simplified as follows,
\[
({_{\textcolor{Red}{R\ot}}R_{\textcolor{BlueViolet}{R\ot}}}
\otimes
{_{\textcolor{ForestGreen}{RR\ot}}R}
\otimes
{_{\textcolor{Red}{R}}R_{\textcolor{BlueViolet}{R}}})
\underset{\textcolor{BlueViolet}{R\ot R}}\otimes
{_{\textcolor{BlueViolet}{R\ot R}}C_{\textcolor{Red}{S\ot S}}}
\cong
{_{\textcolor{ForestGreen}{RR\ot}}R}
\otimes
{_{\textcolor{Red}{R\ot R}}C_{\textcolor{Red}{S\ot S}}}
\]
and the target is simplified as below.
\begin{align*}
&({_{\textcolor{Red}{R\ot}\textcolor{ForestGreen}{R}}C_{\textcolor{olive}{S\ot}\textcolor{BlueViolet}{S}}}
\otimes
{_{\textcolor{ForestGreen}{R\ot}\textcolor{Red}{R}}C_{\textcolor{BlueViolet}{S\ot}\textcolor{olive}{S}}})
\underset{\textcolor{olive}{S\ot\textcolor{BlueViolet}{SS\ot}S}}\otimes
({_{\textcolor{olive}{S\ot}}S_{\textcolor{Red}{S\ot}}}
\otimes
{_{\textcolor{BlueViolet}{SS\ot}}S}
\otimes
{_{\textcolor{olive}{S}}S_{\textcolor{Red}{S}}})\\
\cong\ &
({_{\textcolor{Red}{R\ot}\textcolor{ForestGreen}{R}}C_{\textcolor{Red}{S\ot}\textcolor{BlueViolet}{S}}}
\otimes
{_{\textcolor{ForestGreen}{R\ot}\textcolor{Red}{R}}C_{\textcolor{BlueViolet}{S\ot}\textcolor{Red}{S}}})
\underset{\textcolor{BlueViolet}{SS\ot}}\otimes
{_{\textcolor{BlueViolet}{SS\ot}}S}\\
\cong\ &
{_{\textcolor{Red}{R\ot}\textcolor{ForestGreen}{R}}C_{\textcolor{Red}{S\ot}\textcolor{BlueViolet}{S}}}
\underset{\textcolor{BlueViolet}{S}}\otimes
{_{\textcolor{ForestGreen}{R\ot}\textcolor{Red}{R}}C_{\textcolor{BlueViolet}{S\ot}\textcolor{Red}{S}}}
\end{align*}
Thus in $\textcolor{Red}{R\ot\textcolor{ForestGreen}{RR\ot}R}$-$\Mod$-$\textcolor{Red}{S\ot S}$, module morphisms $C^2$ are in bijection with module morphisms of the following form,
\[
\vcenter{\hbox{\xymatrix{
{_{\textcolor{ForestGreen}{RR\ot}}R}
\otimes
{_{\textcolor{Red}{R\ot R}}C_{\textcolor{Red}{S\ot S}}}
\ar[r]&
{_{\textcolor{Red}{R\ot}\textcolor{ForestGreen}{R}}C_{\textcolor{Red}{S\ot}\textcolor{BlueViolet}{S}}}
\underset{\textcolor{BlueViolet}{S}}\otimes
{_{\textcolor{ForestGreen}{R\ot}\textcolor{Red}{R}}C_{\textcolor{BlueViolet}{S\ot}\textcolor{Red}{S}}}
}}}
\]
which in turn are in bijection with module morphisms
\[
\xymatrix@1@C=5mm{\delta:{_{\textcolor{Red}{R\ot R}}C_{\textcolor{Red}{S\ot S}}}\ar[r]&{_{\textcolor{Red}{R\ot}}C_{\textcolor{Red}{S\ot}\textcolor{BlueViolet}{S}}}\underset{\textcolor{BlueViolet}{S}}\otimes{_{\textcolor{Red}{R}}C_{\textcolor{BlueViolet}{S\ot}\textcolor{Red}{S}}}}
\]
in $\textcolor{Red}{R\ot R}$-$\Mod$-$\textcolor{Red}{S\ot S}$ which satisfy (iv) $\sum \textcolor{ForestGreen}{r}c_{(1)}\otimes c_{(2)}=\sum c_{(1)}\otimes c_{(2)}\textcolor{ForestGreen}{r}$, by using the $\textcolor{ForestGreen}{R}$-actions. Now that we have translated the data, the three axioms of a comonoid for $(C,\varepsilon,\delta)$ translate exactly into the those of an opmonoidal arrow for $(C,C^0,C^2)$.
\end{proof}
\section{Opmonoidal \texorpdfstring{$\dashv$}{-|} Monoidal Adjunctions and Opmonadicity}\label{sec:OpmonoidalOpmonadicity}
We dedicate this section to investigating the interaction between opmonoidal arrows $\xymatrix@1@C=5mm{R\ot R\ar[r]&N}$ and certain opmonadic adjunctions. In particular, we study opmonadic adjunctions where the left adjoint is opmonoidal and the right adjoint is monoidal. This ``opmonoidal $\dashv$ monoidal opmonadicity'' is one of the most powerful tools used throughout the paper, providing us with non-trivial equivalences of categories. But for that, we require that opmonadic adjunctions behave well with respect to the overall structure of the monoidal bicategory.

\begin{defi}\label{def:OpmonadicFriendly}
An \emph{opmonadic-friendly monoidal bicategory} $\mathcal{M}$, is a monoidal bicategory such that
\begin{itemize}
\item Tensoring with objects on either side preserves opmonadicity.
\item Composing with arrows on either side preserves any existing reflexive coequaliser in the hom categories.
\end{itemize}
\end{defi}

A fairly common behaviour of an adjunction in a monoidal bicategory $\mathcal{M}$ between objects that have a (skew) monoidal structure is that the left adjoint is opmonoidal while the right adjoint is monoidal. Surprisingly, these two properties are logically equivalent: for if an opmonoidal arrow has a right adjoint, then the mates of its opmonoidal constraints provide the right adjoint with a monoidal structure and vice versa. Moreover, the right adjoint is strong monoidal if and only if the left adjoint, the unit, and the counit are all opmonoidal, in which case the whole adjunction is in $\SkOpmon(\mathcal{M})$. All of this fits along with a phenomenon called doctrinal adjunction described in \cite{Kelly1974a}.
\begin{defi}
An \emph{opmonoidal $\dashv$ monoidal adjunction} $f\dashv g$ in a monoidal category $\mathcal{M}$, is an adjunction between (skew) monoidales where the left adjoint is opmonoidal and the right adjoint is monoidal.
\end{defi}

Here are some examples of opmonoidal $\dashv$ monoidal adjunctions.

\begin{lem}\label{lem:OpmonoidalUnit}
For every right skew monoidale $(M,m,u,\alpha,\lambda,\rho)$, the unit $\xymatrix@1@C=5mm{u:I\ar[r]&M}$ is an opmonoidal arrow, where $I$ has the trivial monoidal structure. The opmonoidal constraints are given by the diagrams below.
\[
\vcenter{\hbox{\xymatrix@!0@=15mm{
I\ar[r]^{uu}\ar[d]_{1}\ar[rd]_-{u}\xtwocell[rd]{}<>{^<-2>\lambda u\ }&MM\ar[d]^{m}\\
I\ar[r]_{u}&M
}}}
\qquad\qquad
\vcenter{\hbox{\xymatrix@!0@R=15mm@C=10mm{
&I\ar[dl]_-1\ar[rd]^-u&\\I\ar[rr]_-u&&M
}}}
\]\qed
\end{lem}
\begin{rem}
As a consequence, every arrow $\xymatrix@1@C=5mm{i:I\ar[r]&R}$ that has a right adjoint $i^*$ is automatically opmonoidal, taking the skew monoidal structure on $R$ induced by the adjunction $i\dashv i^*$ in Lemma~\ref{lem:OneRightSkewMonoidale}. In other words, every adjunction such that the source of the left adjoint is $I$ is automatically an ``opmonoidal $\dashv$ monoidal adjunction''. In general, the unit and counit are neither monoidal nor opmonoidal.
\end{rem}
\begin{prop}\label{prop:OneOpmonoidalArrow}
For every biduality $R\dashv R\ot$ and every adjunction $i\dashv i^*$ the equality between the triangles below holds.
\[
\vcenter{\hbox{\xymatrix@!0@C=8mm@R=7mm{
&&I\ar[ddddll]_-i\ar[rdd]^-n\ar[ddd]_-{i\ob}&&\\
&&&&\\
&&&R\ot R\ar[dl]_<<{1i^*}\ar[rdd]^-{1}\xtwocell[dd]{}<>{^<1>1\varepsilon }&\\
&&R\ot\ar[rrd]_<<<<{1i}&&\\
R\ar[rrrr]_-{i\ob 1}\ar@{}[rru]|*[@]{\cong}&&&&R\ot R
}}}
\quad=\quad
\vcenter{\hbox{\xymatrix@!0@C=8mm@R=7mm{
\ar@{}[rrrrddd]|*[@]{\cong}&&I\ar[ddddll]_i\ar[rdd]^n&&\\
&&&&\\
&&{\xtwocell[rdd]{}<\omit>{^<2>\varepsilon\ob 1}}&R\ot R\ar[ddlll]_{i\ot 1}\ar[rdd]^1&\\
&&&&\\
R\ar[rrrr]_{i\ob 1}&&&&R\ot R
}}}
\]
Furthermore, taking the skew monoidal structure on $R$ induced by the adjunction $i\dashv i^*$ as in Lemma~\ref{lem:OneRightSkewMonoidale}, and the enveloping monoidale $R\ot R$ induced by the biduality $R\dashv R\ot$ as in Remark~\ref{rem:PantsMonoidale}, the arrow $\xymatrix@1@C=5mm{i\ob 1:R\ar[r]&R\ot R}$ is an opmonoidal arrow and its structure cells are the triangle above and the square below.
\[
\vcenter{\hbox{\xymatrix@!0@=13mm{
RR\ar[r]^-{1i\ob 1}\ar[dd]_-{i^*1}&RR\ot R\ar[r]^-{i\ob 111}\ar[ldd]^-{e1}&**[r]R\ot RR\ot R\ar[dd]^-{1e1}\\
\ar@{}[ru]|<<<<<<*[@ru]{\cong}&&\\
R\ar[rr]_-{i\ob 1}\ar@{}[rruu]|*[@]{\cong}&&R\ot R
}}}
\]
\end{prop}
\begin{proof}

The equality between the triangular cells in the statement follows either by direct calculation using the definition of $\varepsilon\ob$ in terms of $\varepsilon$, or by transposing both triangles along the equivalence $\mathcal{M}(I,R\ot R)\simeq\mathcal{M}(R,R)$, and noticing that this yields the cell $\varepsilon$ in each case. Now we prove that $i\ob 1$ is opmonoidal; axiom \eqref{ax:OM1} follows from the calculation below, and axioms \eqref{ax:OM2} and \eqref{ax:OM3} are verified in a similar way.

{\scriptsize
\begin{align*}
\vcenter{\hbox{\xymatrix@!0@C=12mm@R=7mm{
&&R\ot RR\ot RR\ot R\ar[rrdd]^-{1e111}&&\\
&&RR\ot RR\ot R\ar@<2.5mm>[u]_<<{i\ob 11111}\ar[rdd]|-{e111}&&\\
RRR\ar[rruu]^-{i\ob 1i\ob 1i\ob 1}\ar[dddd]_-{1i^*1}\ar[rrdd]_-{i^*11}\ar[r]_-{11i\ob 1}&**[r]RRR\ot R\ar[ru]^<<<{1i\ob 111}\ar[rrd]_-{i^*111}\ar@{}[luu]|<*=0[@]{\cong}&&&R\ot RR\ot R\ar[dddd]_-{1e1}\\
&&\ar@{}[uu]|*[@]{\cong}&RR\ot R\ar[dddddl]^-{e1}\ar[ru]^<<<{i\ob 111}\ar@{}[uu]|>>>>*=0[@]{\cong}&\\
&&RR\ar[dddd]_{i^*1}\ar[ru]^-{1i\ob 1}\ar@{}[luu]|*=0[@]{\cong}&&\\
&&\ar@{}[ruu]|<<<<<<<*[@]{\cong}&&\\
RR\ar[rrdd]_-{i^*1}\ar@{}[rruu]|*[@]{\cong}&&&&R\ot R\\
&&&&\\
&&R\ar[rruu]_-{i\ob 1}\ar@{}[rruuuuuu]|*[@]{\cong}&&
}}}
\!\!\!\!\!\!\!\!\!&=
\vcenter{\hbox{\xymatrix@!0@C=12mm@R=7mm{
&&R\ot RR\ot RR\ot R\ar[rrdd]^-{1e111}&&\\
&&RR\ot RR\ot R\ar@<2.5mm>[u]_<<{i\ob 11111}\ar[rdd]|-{e111}&&\\
RRR\ar[rruu]^-{i\ob 1i\ob 1i\ob 1}\ar[dddd]_-{1i^*1}\ar[r]_-{11i\ob 1}&**[r]RRR\ot R\ar[ru]^<<<{1i\ob 111}\ar[rrd]_-{i^*111}\ar[ldddd]^-{11e1}\ar@{}[luu]|<*=0[@]{\cong}&&&R\ot RR\ot R\ar[dddd]_-{1e1}\\
\ar@{}[ru]|<<<<<*[@]{\cong}&&\ar@{}[uu]|*[@]{\cong}&RR\ot R\ar[dddddl]^-{e1}\ar[ru]^<<<{i\ob 111}\ar@{}[uu]|>>>>*=0[@]{\cong}&\\
&&&&\\
&&&&\\
RR\ar[rrdd]_-{i^*1}\ar@{}[rrruuu]|*[@]{\cong}&&&&R\ot R\\
&&&&\\
&&R\ar[rruu]_-{i\ob 1}\ar@{}[rruuuuuu]|*[@]{\cong}&&
}}}
\\
=
\vcenter{\hbox{\xymatrix@!0@C=6mm@R=3.5mm{
&&&&R\ot RR\ot RR\ot R\ar[rrrrdddd]^-{1e111}&&&&\\
&&&&&&&&\\
&&&&&&&&\\
&&&&&&&&\\
RRR\ar[rrrruuuu]^-{i\ob 1i\ob 1i\ob 1}\ar[dddddddd]_-{1i^*1}\ar[rdd]^{11i\ob 1}&&\ar@{}[lluuuu]|<<<*=0[@]{\cong}&RR\ot RR\ot R\ar[ruuuu]|<<<{i\ob 11111}\ar@<1mm>[ldddddd]^-{11e1}\ar[rrrdd]^-{e111}&&&&&R\ot RR\ot R\ar[dddddddd]_-{1e1}\\
&&&&&&&&\\
&RRR\ot R\ar[rruu]_<<<<{1i\ob 111}\ar[ldddddd]^-{1e1}&&&&&RR\ot R\ar[ddddddddddll]^-{e1}\ar[rruu]^<<<{i\ob 111}\ar@{}[luuuu]|>>>>*=0[@]{\cong}&&\\
&&&&&&&&\\
\ar@{}[ru]|<<<*[@]{\cong}&&&&&&&&\\
&&&&&&&&\\
&&RR\ot R\ar[rrdddddd]^-{e111}\ar@{}[rrrruuuu]|*[@]{\cong}&&&&&&\\
&&&&&&&&\\
RR\ar[rrrrdddd]_-{i^*1}\ar[rruu]_-{1i\ob 1}\ar@{}[rrruuuuuuuu]|*[@]{\cong}&&&&&&&&R\ot R\\
&&&&&&&&\\
&&\ar@{}[uuu]|*[@]{\cong}&&&&&&\\
&&&&&&&&\\
&&&&R\ar[rrrruuuu]_-{i\ob 1}\ar@{}[rrrruuuuuuuuuuuu]|*[@]{\cong}&&&&
}}}
\!\!\!\!\!\!\!\!\!&=
\vcenter{\hbox{\xymatrix@!0@C=6mm@R=3.5mm{
&&&&R\ot RR\ot RR\ot R\ar[rrrrdddd]^-{1e111}\ar[dddddddd]^-{111e1}&&&&\\
&&&&&&&&\\
&&&&&&&&\\
&&&&&&&&\\
RRR\ar[rrrruuuu]^-{i\ob 1i\ob 1i\ob 1}\ar[dddddddd]_-{1i^*1}\ar[rdd]^{11i\ob 1}&&\ar@{}[lluuuu]|<<<*=0[@]{\cong}&RR\ot RR\ot R\hspace{3mm}\ar[ruuuu]^<<<{i\ob 11111}\ar@<1mm>[ldddddd]|-{11e1}&&&&&R\ot RR\ot R\ar[dddddddd]_-{1e1}\\
&&&&&&&&\\
&RRR\ot R\ar[rruu]_<<<<{1i\ob 111}\ar[ldddddd]^-{1e1}&&&&&&&\\
&&&&&&&&\\
\ar@{}[ru]|<<<*[@]{\cong}&&&&R\ot RR\ot R\ar[rrrrdddd]^-{1e1}\ar@{}[rrrruuuu]|*[@]{\cong}&&&&\\
&&&&&&&&\\
&&RR\ot R\ar[rrdddddd]^-{e111}\ar[rruu]_{i\ob 11}\ar@{}[rrruuuuuuuu]|*[@]{\cong}&&&&&&\\
&&&&&&&&\\
RR\ar[rrrrdddd]_-{i^*1}\ar[rruu]_-{1i\ob 1}\ar@{}[rrruuuuuuuu]|*[@]{\cong}&&&&&&&&R\ot R\\
&&&&&&&&\\
&&\ar@{}[uuu]|*[@]{\cong}&&&&&&\\
&&&&&&&&\\
&&&&R\ar[rrrruuuu]_-{i\ob 1}\ar@{}[uuuuuuuu]|*[@]{\cong}&&&&
}}}
\end{align*}}
\end{proof}

In Lemma~\ref{lem:OpmonoidalUnit} and Proposition~\ref{prop:OneOpmonoidalArrow} we exhibit two opmonoidal left adjoints $\xymatrix@1@C=5mm{i:I\ar[r]&R}$ and $\xymatrix@1@C=5mm{i\ob\, 1:R\ar[r]&R\ot R}$ which may be composed into a new opmonoidal left adjoint $\xymatrix@1@C=5mm{i\ob\, i:I\ar[r]&R\ot R}$. And by a doctrinal adjunction argument, the opmonoidal structures on the left adjoints $i$, $i\ob 1$ and $i\ob i$ induce monoidal structures on the right adjoints $i^*$, $i\ot\, 1$ and $i\ot\, i^*$ which in general are not strong monoidal, hence these adjunctions do not belong to $\Opmon(\mathcal{M})$.

\subsection{A Bicategorical Theorem}
We proceed with one of the main results: in an opmonadic-friendly monoidal bicategory $\mathcal{M}$ the functor
\[
\xymatrix@!0@C=50mm{
\SkOpmon(i\ob 1,N):\Opmon(R\ot R,N)\ar[r]_-{\simeq}&\SkOpmon(R,N)
}
\]
is an equivalence of categories, provided that the opmonoidal arrow $i\ob 1$ in Proposition~\ref{prop:OneOpmonoidalArrow} is opmonadic, and $N$ is a genuine monoidale (not just a skew one). This is stated formally as Theorem~\ref{teo:OpmonadicityOpmonoidal} below. Its proof uses some of the important techniques employed throughout this paper, and it naturally breaks down into two parts: an isomorphism followed by an equivalence of categories, therefore, to gain some clarity we present these separately in Lemma~\ref{lem:OpmonadicityOpmonoidal} and Theorem~\ref{teo:OpmonadicityOpmonoidal} below. Taking the middle step and most of the technicalities, there is a category that we denote by $\mathcal{X}(R,N)$. One way to informally interpret the category $\mathcal{X}(R,N)$ is as follows: its objects are opmonoidal arrows $\xymatrix@1@C=5mm{R\ar[r]&N}$ equipped with a module structure for the monad induced by the adjunction
\[
\vcenter{\hbox{\xymatrix{
R\ot R\xtwocell[d]{}_{i\ob 1\ }^{\ i\ot 1}{'\dashv}\\
R
}}}
\]
together with compatibility conditions between the opmonoidal and the module structures which involve the ``opmonoidal $\dashv$ monoidal'' structure of the adjunction $i\ob 1\dashv i\ot 1$. What we show in Lemma~\ref{lem:OpmonadicityOpmonoidal} is that this extra module structure on the opmonoidal arrows $\xymatrix@1@C=5mm{R\ar[r]&N}$ is in fact redundant, hence the isomorphism $\mathcal{X}(R,N)\cong\SkOpmon(R,N)$. And when $i\ob 1\dashv i\ot 1$ is opmonadic the category $\mathcal{X}(R,N)$ of ``opmonoidal $\dashv$ monoidal modules'' (as we may informally call them) is equivalent to $\Opmon(R\ot R,N)$, as some sort of ``opmonoidal $\dashv$ monoidal opmonadicity''.
\[
\vcenter{\hbox{\xymatrix{
R\ot R\dtwocell_{i\ob 1\ }^{\ i\ot 1}{'\dashv}\\
R
}}}
\vcenter{\hbox{\xymatrix@!0{
\ar@{|~>}[r]&
}}}
\!\!\!\!\!\!\!\!\!\!\!\!\!\!\!\!\!\!\!\!\!\!\!\!\!\!\!\!\!\!\!\!\!\!\!\!\!\!\!\!\!
\vcenter{\hbox{\xymatrix{
\qquad\qquad\qquad\qquad\mathcal{M}(R\ot R,N)\simeq\mathcal{M}(R,N)^{\mathcal{M}(t1,N)}\dtwocell_{\mathcal{M}(i\ot 1,N)\hspace{11mm}}^{\hspace{11mm}\mathcal{M}(i\ob 1,N)}{'\dashv}\\
\mathcal{M}(R,N)
}}}
\!\!\!\!\!\!\!\!\!\!\!\!\!\!\!\!
\vcenter{\hbox{\xymatrix@!0{
\ar@{|~>}[r]&
}}}
\!\!\!
\vcenter{\hbox{\xymatrix{
\quad\Opmon(R\ot R,N)\simeq\mathcal{X}(R,N)\ar[d]_\simeq^{\Opmon(i\ob 1,N)}\\
\SkOpmon(R,N)
}}}
\]
We now make this precise.
\begin{defi}
For a right skew monoidale $(N,m,u)$, a biduality $R\dashv R\ot$, and an adjunction $i\ob\dashv i\ot$ in $\mathcal{M}$,
\[
\vcenter{\hbox{\xymatrix{
R\ot\xtwocell[d]{}_{i\ob}^{i\ot}{'\dashv}\\
I
}}}
\]
the category $\mathcal{X}(R,N)$ has objects pairs $(D,\varphi)$ where $\xymatrix@1@C=5mm{D:R\ar[r]&N}$ is an opmonoidal arrow in $\mathcal{M}$ and $\varphi$ is a cell
\[
\vcenter{\hbox{\xymatrix{
R\ar[r]_-{i\ob 1}\ar@/^9mm/[rrr]^-{D}&R\ot R\ar[r]_-{i\ot 1}\rtwocell<\omit>{^<-3>\varphi}&R\ar[r]_-{D}&N
}}}
\]
satisfying five axioms: two which assert that $\varphi$ is an action for the monad induced by the adjunction $i\ob 1\dashv i\ot 1$, and three of which express the following compatibility between $\varphi$ and the opmonoidal constraints of $D$.
\begin{align}
\tag{X1}\label{ax:X1}
\vcenter{\hbox{\xymatrix@!0@R=16mm@C=8mm{
&{\xtwocell[rrrdd]{}<>{^<-2>D^0}}&&I\ar[llldd]_-{i}\ar[rrrdd]^-{u}&&&\\
&&&&&&\\
R\ar[rr]_-{i\ob 1}\ar@/^9mm/[rrrrrr]^-{D}\xtwocell[rrrrrr]{}<>{^<-3>\varphi}&&R\ot R\ar[rr]_-{i\ot 1}&&R\ar[rr]_-{D}&&N
}}}
&=
\vcenter{\hbox{\xymatrix@!0@R=16mm@C=8mm{
&\ar@{}[rrdd]|*[@]{\cong}&&I\ar[llldd]_-{i}\ar[d]_-{n}\ar[rdd]^-{i}\ar[rrrdd]^-{u}&&&\\
&&{\xtwocell[rd]{}<>{^<3>\ \varepsilon\ob 1}}&**[l]R\ot R\ar[llld]_-{i\ot 1}\ar[ld]^-{1}\xtwocell[rrd]{}<>{^<-2>D^0}&&&\\
R\ar[rr]_-{i\ob 1}&&R\ot R\ar[rr]_-{i\ot 1}\ar@{}[rru]|*[@ru]{\quad\cong}&&R\ar[rr]_-{D}&&N
}}}
\\
\tag{X2}\label{ax:X2}
\vcenter{\hbox{\xymatrix@!0@R=7mm@C=7mm{
&&&&NN\ar[rrrd]^-{m}&&&\\
RR\ar[rrrru]^-{DD}\ar[rddd]_-{i\ob 11}\xtwocell[rrrr]{}<>{^<8>\varphi D\quad}&&&&&&&N\\
&&&&&&&\\
&&&&&&&\\
&R\ot RR\ar[rddd]_-{i\ot 11}&&{\xtwocell[rrru]{}<>{^<2>D^2\quad}}&&&&\\
&&&&&&&\\
&&&&&&&\\
&&RR\ar[rrrd]_-{i^*1}\ar[rruuuuuuu]^-{DD}&&&&&\\
&&&&&R\ar[rruuuuuuu]_-{D}&&
}}}
\!\!&=\!\!
\vcenter{\hbox{\xymatrix@!0@R=7mm@C=7mm{
&&&&NN\ar[rrrd]^-{m}&&&\\
RR\ar[rrrru]^-{DD}\ar[rrrd]^-{i^*1}\ar[rddd]_-{i\ob 11}\xtwocell[rrrrrrr]{}<>{^D^2\ \ }&&&&&&&N\\
&&&R\ar[rrrru]^-{D}\ar[rddd]_-{i\ob 1}\xtwocell[rrrr]{}<>{^<8>\varphi}&&&&\\
&&&&&&&\\
&R\ot RR\ar[rrrd]^-{1i^*1}\ar[rddd]_-{i\ot 11}\ar@{}[rruu]|*[@]{\cong}&&&&&&\\
&&&&R\ot R\ar[rddd]_-{i\ot 1}&&&\\
&&&&&&&\\
&&RR\ar[rrrd]_-{i^*1}\ar@{}[rruu]|*[@]{\cong}&&&&&\\
&&&&&R\ar[rruuuuuuu]_-{D}&&
}}}
\\
\tag{X3}\label{ax:X3}
\vcenter{\hbox{\xymatrix@!0@R=7mm@C=7mm{
&&&&NN\ar[rrrd]^-{m}&&&\\
RR\ar[rrrru]^-{DD}\ar[rddd]_-{1i\ob 1}\xtwocell[rrrr]{}<>{^<8>D\varphi\quad}&&&&&&&N\\
&&&&&&&\\
&&&&&&&\\
&RR\ot R\ar[rddd]_-{1i\ot 1}&&{\xtwocell[rrru]{}<>{^<2>D^2\quad}}&&&&\\
&&&&&&&\\
&&&&&&&\\
&&RR\ar[rrrd]_-{i^*1}\ar[rruuuuuuu]_-{DD}&&&&&\\
&&&&&R\ar[rruuuuuuu]_-{D}&&
}}}
\!\!&=\!\!
\vcenter{\hbox{\xymatrix@!0@R=7mm@C=7mm{
&&&&NN\ar[rrrd]^-{m}&&&\\
RR\ar[rrrru]^-{DD}\ar@/^8mm/[rrrrrddddddd]^-{i^*1}\ar[rddd]_-{1i\ob 1}&&&&&&&N\\
&&&{\xtwocell[rrr]{}<>{^<4>D^2\quad}}&&&&\\
&&&&&&&\\
&RR\ot R\ar[rddd]_-{1i\ot 1}\ar@/^2mm/[rrd]^-{1}\ar@{}[rru]|*[@]{\cong}\xtwocell[rrrd]{}<>{^<2>1\varepsilon\ob 1\quad}&&&&&&\\
&&&\quad RR\ot R\ar[rrddd]^-{e1}\ar@{}[ddd]|*[@u]{\cong}&&&&\\
&&&&&&&\\
&&RR\ar[rrrd]_-{i^*1}\ar[ruu]|-{1i\ob 1}&&&&&\\
&&&&&R\ar[rruuuuuuu]_-{D}&&
}}}
\end{align}
And an arrow $\xymatrix@1@C=5mm{\gamma:(D,\varphi)\ar[r]&(D',\varphi')}$ in $\mathcal{X}(R,N)$ is defined as an opmonoidal cell $\xymatrix@1@C=5mm{\gamma:D\ar[r]&D'}$ in $\mathcal{M}$ which preserves the actions $\varphi$ and $\varphi'$, in the sense of the equation below.
\begin{equation}\tag{X4}\label{ax:X4}
\vcenter{\hbox{\xymatrix{
R\ar[r]_{i\ob 1}\ar@/^20mm/[rrr]<1mm>^{D'}&R\ot R\ar[r]_{i\ot 1}\rtwocell<\omit>{^<-6>\varphi'}&R\ar[r]_D\ar@/^7mm/[r]^<<<<<{D'}\rtwocell<>{^<-2>\gamma}&N
}}}
\quad=\quad
\vcenter{\hbox{\xymatrix{
R\ar[r]_{i\ob 1}\ar@/^9mm/[rrr]^D\ar@/^20mm/[rrr]<1mm>^{D'}&R\ot R\ar[r]_{i\ot 1}\rtwocell<\omit>{^<-3>\varphi}\rtwocell<\omit>{^<-9>\gamma}&R\ar[r]_D&N
}}}
\end{equation}
Composition and identities are defined as in $\mathcal{M}(R,N)$.
\end{defi}
\begin{lem}\label{lem:OpmonadicityOpmonoidal}
For every monoidale $(N,m,u)$, every biduality $R\dashv R\ot$, and every adjunction $i\ob\dashv i\ot$ in $\mathcal{M}$
\[
\vcenter{\hbox{\xymatrix{
R\ot\dtwocell_{i\ob}^{i\ot}{'\dashv}\\
I
}}}
\]
the forgetful functor $\xymatrix@1@C=5mm{F:\mathcal{X}(R,N)\ar[r]&\SkOpmon(R,N)}$ is an isomorphism of categories.
\end{lem}
\begin{proof}

It is clear that $F$ is faithful. To see that $F$ is injective on objects observe that for an object $(D,\varphi)$ of $\mathcal{X}(R,N)$ the following calculation exhibits $\varphi$ purely in terms of the opmonoidal constraints $D^0$ and $D^2$ of $D$ (note the need for $N$ to be left normal).
\[
\vcenter{\hbox{\xymatrix{
R\ar[r]_-{i\ob 1}\ar@/^9mm/[rrr]^-{D}&R\ot R\ar[r]_-{i\ot 1}\rtwocell<\omit>{^<-3>\varphi}&R\ar[r]_-{D}&N
}}}
\]
\[
\stackrel{\eqref{ax:OM3}}{=}\!
\vcenter{\hbox{\xymatrix@!0@R=7mm@C=15mm{
R\ar[dd]_-{i1}\ar@/^9mm/[rrr]^-{D}\ar@/_9mm/[dddd]_1\xtwocell[rrrd]{}<>{^<-3>D^0D\quad\ }\xtwocell[dddd]{}<>{^<3>\eta 1}&&&N\ar[dd]^-{u1}\ar@/^9mm/[dddd]^-{1}&\\
&&&&\\
RR\ar[dd]_-{i^*1}\ar@/^9mm/[rrr]^-{DD}\xtwocell[rrrd]{}<>{^<-3>D^2\ }&&&NN\ar[dd]^-{m}\ar@{}[r]|<<*=0[@l]{\cong}&\\
&&&&\\
R\ar[r]_-{i\ob 1}\ar@/^9mm/[rrr]^-{D}&R\ot R\ar[r]_-{i\ot 1}\rtwocell<\omit>{^<-3>\varphi}&R\ar[r]_-{D}&N&
}}}
\]
\[
\stackrel{\eqref{ax:X2}}{=}
\vcenter{\hbox{\xymatrix@!0@R=7mm@C=15mm{
R\ar[dd]_-{i1}\ar@/^9mm/[rrr]^-{D}\ar@/_9mm/[dddd]_-{1}\xtwocell[rrrd]{}<>{^<-3>D^0D\quad\ }\xtwocell[dddd]{}<>{^<3>\eta 1}&&&N\ar[dd]^-{u1}\ar@/^9mm/[dddd]^-{1}&\\
&&&&\\
RR\ar[dd]_-{i^*1}\ar[r]_-{i\ob 11}\ar@/^9mm/[rrr]^-{DD}&R\ot RR\ar[r]_-{i\ot 1}\ar[dd]|-{1i^*1}\rtwocell<\omit>{^<-3>\varphi D\ \ }&RR\ar[r]_-{DD}\ar[dd]|-{i^*1}\xtwocell[rdd]{}<>{^D^2\ }&NN\ar[dd]^-{m}\ar@{}[r]|<<*=0[@l]{\cong}&\\
&&&&\\
R\ar[r]_-{i\ob 1}\ar@{}[ruu]|-*[@]{\cong}&R\ot R\ar[r]_-{i\ot 1}\ar@{}[ruu]|-*[@]{\cong}&R\ar[r]_-{D}&N&
}}}
\]
\[
\stackrel{\eqref{ax:X1}}{=}
\vcenter{\hbox{\xymatrix@!0@R=7mm@C=15mm{
R\ar@/_9mm/[dddd]_-{1}\ar[dd]_-{i1}\ar[r]|-{n1}\ar@/^13mm/[rrdd]^-{i1}\ar@/^9mm/[rrr]^-{D}\xtwocell[dddd]{}<>{^<3>\eta 1}\xtwocell[rddd]{}<>{^<-1>\varepsilon\ob 11}&**[r]R\ot RR\ar@/_/[ldd]|-{i\ot 11\quad}\ar[dd]^-{1}&{\xtwocell[rd]{}<>{^D^0D\quad\ }}&N\ar[dd]^-{u1}\ar@/^9mm/[dddd]^-{1}\\
&&&\\
RR\ar[dd]_-{i^*1}\ar[r]_-{i\ob 11}&R\ot RR\ar[r]_-{i\ot 1}\ar[dd]|-{1i^*1}\ar@{}[ruu]|-*[@]{\cong}&RR\ar[r]^-{DD}\ar[dd]|-{i^*1}\xtwocell[rdd]{}<>{^D^2\ }&NN\ar[dd]^-{m}\ar@{}[r]|<<*=0[@l]{\cong}&\\
&&&\\
R\ar[r]_-{i\ob 1}\ar@{}[ruu]|-*[@]{\cong}&R\ot R\ar[r]_-{i\ot 1}\ar@{}[ruu]|-*[@]{\cong}&R\ar[r]_-{D}&N
}}}
\]
\[
=
\vcenter{\hbox{\xymatrix@!0@R=10mm@C=8mm{
&&N\ar[rrd]^-{u1}\ar@/^18mm/[rrrrddd]^-{1}&&&&\ar@{}[lld]|*=0[@]{\cong}\\
&&&&NN\ar[rrdd]^-{m}&&\\
{\xtwocell[rrrru]{}<>{^<-2>D^0D\quad}}&&RR\ar[rrd]|-{i^*1}\ar[rru]^-{DD}\xtwocell[rrrrd]{}<>{^<-2>D^2\ }&&&&\\
R\ar[rr]_-{i\ob 1}\ar[rru]^-{i1}\ar@/^/[rruuu]^-{D}&&R\ot R\ar[rr]_-{i\ot 1}\ar@{}[u]|*[@]{\cong}&&R\ar[rr]_-{D}&&N
}}}
\]
Now, to see that $F$ is surjective on objects take an arbitrary opmonoidal arrow $D$ in $\Opmon(R,N)$, then let $\varphi$ be the cell below.
\begin{equation}\label{eq:redundant_action1}
\vcenter{\hbox{\xymatrix{
R\ar[r]_-{i\ob 1}\ar@/^9mm/[rrr]^-{D}&R\ot R\ar[r]_-{i\ot 1}\xtwocell[r]{}<>{^<-3>\varphi}&R\ar[r]_-{D}&N
}}}
:=
\vcenter{\hbox{\xymatrix@!0@R=10mm@C=8mm{
&&N\ar[rrd]^-{u1}\ar@/^18mm/[rrrrddd]^-{1}&&&&\ar@{}[lld]|*=0[@]{\cong}\\
&&&&NN\ar[rrdd]^-{m}&&\\
{\xtwocell[rrrru]{}<>{^<-2>D^0D\quad}}&&RR\ar[rrd]|-{i^*1}\ar[rru]^-{DD}\xtwocell[rrrrd]{}<>{^<-2>D^2\ }&&&&\\
R\ar[rr]_-{i\ob 1}\ar[rru]^-{i1}\ar@/^/[rruuu]^-{D}&&R\ot R\ar[rr]_-{i\ot 1}\ar@{}[u]|*[@]{\cong}&&R\ar[rr]_-{D}&&N
}}}
\end{equation}
This cell $\varphi$ exhibits $D$ as an object of $\mathcal{X}(R,N)$. We shall prove the five axioms that make it happen, starting with the two that make $\varphi$ into an action for the monad induced by $i\ob 1\dashv i\ot 1$.
\begin{gather*}
\vcenter{\hbox{\xymatrix@!0@R=10mm@C=8mm{
&&N\ar[rrd]^-{u1}\ar@/^18mm/[rrrrddd]^-{1}&&&&\ar@{}[lld]|*=0[@]{\cong}\\
&&&&NN\ar[rrdd]^-{m}&&\\
{\xtwocell[rrrru]{}<>{^<-2>D^0D\quad}}&&RR\ar[rrd]|-{i^*1}\ar[rru]^-{DD}\xtwocell[rrrrd]{}<>{^<-2>D^2\ }&&&&\\
R\ar[rr]_-{i\ob 1}\ar[rru]^-{i1}\ar@/^/[rruuu]^-{D}\ar@/_10mm/[rrrr]_-{1}\xtwocell[rrrr]{}<>{^<3>\eta\ob 1\quad}&&R\ot R\ar[rr]_-{i\ot 1}\ar@{}[u]|*[@]{\cong}&&R\ar[rr]_-{D}&&N
}}}
\quad=\quad
\vcenter{\hbox{\xymatrix@!0@R=10mm@C=8mm{
&&N\ar[rrd]^-{u1}\ar@/^18mm/[rrrrddd]^-{1}&&&&\ar@{}[lld]|*=0[@]{\cong}\\
&&&&NN\ar[rrdd]^-{m}&&\\
{\xtwocell[rrrru]{}<>{^<-2>D^0D\quad}}&&RR\ar[rrd]|-{i^*1}\ar[rru]^-{DD}\xtwocell[rrrrd]{}<>{^<-2>D^2\ }&&&&\\
R\ar[rru]^-{i1}\ar@/^/[rruuu]^-{D}\ar@/_10mm/[rrrr]_-{1}\xtwocell[rrrr]{}<>{^\eta 1\ }&&&&R\ar[rr]_-{D}&&N
}}}
\\
\stackrel{\eqref{ax:OM3}}{=}
\vcenter{\hbox{\xymatrix@!0@R=10mm@C=8mm{
&&N\ar[rrd]^-{u1}\ar@/^18mm/[rrrrddd]^-{1}\ar@/_7mm/[rrrrddd]_-{1}&&&&\ar@{}[lld]|*=0[@]{\cong}\\
&&&&NN\ar[rrdd]^-{m}&&\\
&&&\ar@{}[ru]|*[@]{\cong}&&&\\
R\ar@/^/[rruuu]^-{D}\ar[rrrr]_-{1}&&&&R\ar[rr]_-{D}&&N
}}}
\quad=\quad
\id_D
\end{gather*}
The proof of the second axiom requires $N$ to be a genuine monoidale (not just a skew left normal one).
\[
\vcenter{\hbox{\xymatrix@!0@R=10mm@C=8mm{
&&N\ar[rrd]^-{u1}\ar@/^5mm/[rrrr]^-{1}&&&&N\ar[rrd]^-{u1}\ar@/^18mm/[rrrrddd]^-{1}&&&&\ar@{}[lld]|*=0[@]{\cong}\\
&&&&NN\ar[rru]^-{m}\ar@{}[u]|>*[@d]{\cong}&&&&NN\ar[rrdd]^-{m}&&\\
{\xtwocell[rrrru]{}<>{^<-2>D^0D\quad}}&&RR\ar[rrd]|-{i^*1}\ar[rru]^-{DD}\xtwocell[rrrd]{}<>{^<-3>D^2\ }&&{\xtwocell[rrrru]{}<>{^<-2>D^0D\quad}}&&RR\ar[rrd]|-{i^*1}\ar[rru]^-{DD}\xtwocell[rrrrd]{}<>{^<-2>D^2\ }&&&&\\
R\ar[rr]_-{i\ob 1}\ar[rru]^-{i1}\ar@/^/[rruuu]^-{D}&&R\ot R\ar[rr]_-{i\ot 1}\ar@{}[u]|*[@]{\cong}&&R\ar[rr]_-{i\ob 1}\ar[rru]^-{i1}\ar@/^/[rruuu]^-{D}&&R\ot R\ar[rr]_-{i\ot 1}\ar@{}[u]|*[@]{\cong}&&R\ar[rr]_-{D}&&N
}}}
\]
\[
=
\vcenter{\hbox{\xymatrix@!0@R=10mm@C=8mm{
&&N\ar[rrd]^-{u1}\ar@/^5mm/[rrrr]^-{1}&&&&N\ar[rrd]^-{u1}\ar@/^18mm/[rrrrdddd]^-{1}&&&&\ar@{}[llddd]|*=0[@]{\cong}\\
&&&&NN\ar[rru]^-{m}\ar@{}[u]|>*[@d]{\cong}\ar[rr]^-{u11}\xtwocell[dd]{}<>{^D^0DD}&&NNN\ar[rr]^-{1m}\ar@{}[u]|*[@d]{\cong}\xtwocell[dd]{}<>{^DD^2}&&NN\ar[rrddd]^-{m}&&\\
&&&&&&&&&&\\
{\xtwocell[rrrru]{}<>{^<-3>D^0D\quad}}&&RR\ar[rrd]|-{i^*1}\ar[rr]_>>>>{i11}\ar[rruu]^-{DD}&&RRR\ar[rr]_<<<<{1i^*1}\ar[rruu]|-{DDD}&&RR\ar[rrd]|-{i^*1}\ar[rruu]|-{DD}\xtwocell[rrrrd]{}<>{^<-4>D^2\ }&&&&\\
R\ar[rr]_-{i\ob 1}\ar[rru]^-{i1}\ar@/^/[rruuuu]^-{D}&&R\ot R\ar[rr]_-{i\ot 1}\ar@{}[u]|*[@]{\cong}&&R\ar[rr]_-{i\ob 1}\ar[rru]|-{i1}\ar@{}[u]|*[@]{\cong}&&R\ot R\ar[rr]_-{i\ot 1}\ar@{}[u]|*[@]{\cong}&&R\ar[rr]_-{D}&&N
}}}
\]
\[
\stackrel{\eqref{ax:OM1}}{=}
\vcenter{\hbox{\xymatrix@!0@R=10mm@C=8mm{
&&N\ar[rrd]^-{u1}\ar[rdd]_-{u1}\ar@/^5mm/[rrrr]^-{1}&&&&N\ar[rrd]|-{u1}\ar@/^18mm/[rrrrddddd]^-{1}&&&&\\
&&&&NN\ar[rru]^-{m}\ar@{}[u]|>*[@d]{\cong}\ar[rr]^-{u11}&&NNN\ar[rr]^-{1m}\ar[rd]|-{m1}\ar@{}[u]|*[@d]{\cong}&&NN\ar[rrdddd]^-{m}&&\ar@{}[lldd]|*=0[@]{\cong}\\
&&&NN\ar[rrru]|-{1u1}\ar@{}[ru]|>>>>*=0[@ld]{\cong}&&&&NN\ar[rrrddd]^-{m}\ar@{}[ru]|*=0[@ld]{\cong}&&&\\
&&RR\ar[ru]^-{DD}\ar[rrd]^-{1i1}\xtwocell[rrruu]{}<>{^<3>DD^0D\qquad}&&{\xtwocell[rrrru]{}<>{^<-1>D^2D\quad}}&{\xtwocell[rrrrrd]{}<>{^D^2\ }}&RR\ar[ru]|-{DD}\ar[rrdd]|-{i^*1}&&&&\\
{\xtwocell[rrrrru]{}<>{^<-10>D^0D\quad}}&&RR\ar[rrd]|-{i^*1}\ar[rr]_>>>>{i11}\ar@{}[u]|*[@]{\cong}&&RRR\ar[rr]_<<<<{1i^*1}\ar[rruuu]|<<<<<<<<<<<{DDD}\ar[rru]|-{i^*11}&&RR\ar[rrd]|-{i^*1}\ar@{}[u]|*[@]{\cong}&&&&\\
R\ar[rr]_-{i\ob 1}\ar[rru]|-{i1}\ar[rruu]^-{i1}\ar@/^/[rruuuuu]^-{D}&&R\ot R\ar[rr]_-{i\ot 1}\ar@{}[u]|*[@]{\cong}&&R\ar[rr]_-{i\ob 1}\ar[rru]|-{i1}\ar@{}[u]|*[@]{\cong}&&R\ot R\ar[rr]_-{i\ot 1}\ar@{}[u]|*[@]{\cong}&&R\ar[rr]_-{D}&&N
}}}
\]
\[
\stackrel{\eqref{ax:SKM3}}{\stackrel{\eqref{ax:OM2}}{=}}
\vcenter{\hbox{\xymatrix@!0@R=10mm@C=8mm{
&&N\ar[rrd]^-{u1}\ar[rdd]_-{u1}\ar@/^5mm/[rrrr]^-{1}&&&&N\ar@/^18mm/[rrrrddddd]^-{1}&&&&\\
&&&&NN\ar[]!<-3mm,0mm>;[rru]^-{m}\ar@{}[u]|>*[@d]{\cong}\ar[rr]_<<{u11}\ar@/^12mm/[]!<3mm,1mm>;[rrrd]!<2mm,0mm>^>>>>{1}&&NNN\ar[rd]|-{m1}\ar@{}[u]|*=0[@ld]{\cong}&&&&\\
&&&NN\ar[rrru]|-{1u1}\ar[rrrr]^-{1}\ar@{}[ru]|>>>>*=0[@ld]{\cong}&&\ar@{}[ru]|*=0[@ld]{\cong}&&NN\ar[rrrddd]^-{m}&&&\\
&&RR\ar[ru]^-{DD}\ar[rrd]|-{1i1}\ar[rr]|-{i^*1}\ar@/^6mm/[rrrr]^-{1}&{\xtwocell[rrr]{}<>{^<-2>\varepsilon 1}}&R\ar[rr]|-{i1}&{\xtwocell[rrrrrd]{}<>{^D^2\ }}&RR\ar[ru]|-{DD}\ar[rrdd]|-{i^*1}&&&&\\
{\xtwocell[rrrrru]{}<>{^<-10>D^0D\quad}}&&RR\ar[rrd]|-{i^*1}\ar[rr]_>>>>{i11}\ar@{}[u]|*[@]{\cong}&&RRR\ar[rr]_<<<<{1i^*1}\ar[rru]|-{i^*11}\ar@{}[u]|*[@]{\cong}&&RR\ar[rrd]|-{i^*1}\ar@{}[u]|*[@]{\cong}&&&&\\
R\ar[rr]_-{i\ob 1}\ar[rru]|-{i1}\ar[rruu]^-{i1}\ar@/^/[rruuuuu]^-{D}&&R\ot R\ar[rr]_-{i\ot 1}\ar@{}[u]|*[@]{\cong}&&R\ar[rr]_-{i\ob 1}\ar[rru]|-{i1}\ar@{}[u]|*[@]{\cong}&&R\ot R\ar[rr]_-{i\ot 1}\ar@{}[u]|*[@]{\cong}&&R\ar[rr]_-{D}&&N
}}}
\]
\[
\stackrel{\eqref{ax:SKM5}}{=}
\vcenter{\hbox{\xymatrix@!0@R=8mm@C=8mm{
&&N\ar[rrd]^-{u1}\ar[rdd]_-{u1}\ar@/^5mm/[rrrr]^-{1}&&&&N\ar@/^16mm/[rrrrddddd]^-{1}&&&&\\
&&&&NN\ar[]!<-3mm,0mm>;[rru]^-{m}\ar@{}[u]|>*[@d]{\cong}\ar@/^5mm/[]!<3mm,1mm>;[rrrd]!<2mm,0mm>^-{1}&&&&&&\\
&&&NN\ar[rrrr]^-{1}&&&&NN\ar[rrrddd]^-{m}&&&\\
&&RR\ar[ru]^-{DD}\ar@/^/[rrrr]^-{1}&&&{\xtwocell[rrrrrd]{}<>{^D^2\ }}&RR\ar[ru]_-{DD}\ar[rrdd]|-{i^*1}&&&&\\
{\xtwocell[rrrrru]{}<>{^<-9>D^0D\quad}}&&&&&&&&&&\\
R\ar[rr]_-{i\ob 1}\ar[rruu]^-{i1}\ar@/^/[rruuuuu]^-{D}&&R\ot R\ar[rr]_-{i\ot 1}\ar@/^10mm/[rrrr]^-{1}\xtwocell[rrrr]{}<>{^<-3>\varepsilon\ob 1\quad}&&R\ar[rr]_-{i\ob 1}&&R\ot R\ar[rr]_-{i\ot 1}&&R\ar[rr]_-{D}&&N
}}}
\]
\[
=
\vcenter{\hbox{\xymatrix@!0@R=8mm@C=8mm{
&&N\ar[rrrrdd]_-{u1}\ar@/^17mm/[rrrrrrrrddddd]^-{1}&&&&&\ar@{}[ddl]|*[@]{\cong}&&&\\
&&&&&&&&&&\\
&&&&&&NN\ar[rrrrddd]^-{m}&&&&\\
{\xtwocell[rrrrrru]{}<>{^<-3>D^0D\quad}}&&&&RR\ar[rru]^-{DD}\ar[rrrrdd]|-{i^*1}\xtwocell[rrrrrd]{}<>{^D^2\ }&&&&&&\\
&&&&\ar@{}[u]|*[@]{\cong}&&&&&&\\
R\ar[rr]_-{i\ob 1}\ar[rrrruu]^-{i1}\ar@/^/[rruuuuu]^-{D}&&R\ot R\ar[rr]_-{i\ot 1}\ar@/^9mm/[rrrr]|-{1}\xtwocell[rrrr]{}<>{^<-3>\varepsilon\ob 1\quad}&&R\ar[rr]_-{i\ob 1}&&R\ot R\ar[rr]_-{i\ot 1}&&R\ar[rr]_-{D}&&N
}}}
\]
The axiom \eqref{ax:X1} holds.
\[
\vcenter{\hbox{\xymatrix@!0@R=10mm@C=8mm{
I\ar[rrr]^-{u}\ar[rddd]_-{i}\xtwocell[rr]{}<>{^<4>\quad\qquad D^0}&&&N\ar[rrd]^-{u1}\ar@/^17mm/[rrrrddd]^-{1}&&&&\ar@{}[lld]|*=0[@]{\cong}\\
&&&&&NN\ar[rrdd]^-{m}&&\\
&{\xtwocell[rrrru]{}<>{^<-2>D^0D\quad}}&&RR\ar[rrd]|-{i^*1}\ar[rru]^-{DD}\xtwocell[rrrrd]{}<>{^<-2>D^2\ }&&&&\\
&R\ar[rr]_-{i\ob 1}\ar[rru]^-{i1}\ar@/^/[rruuu]^-{D}&&R\ot R\ar[rr]_-{i\ot 1}\ar@{}[u]|*[@]{\cong}&&R\ar[rr]_-{D}&&N
}}}
=\!\!
\vcenter{\hbox{\xymatrix@!0@R=9mm@C=8mm{
I\ar[rrr]^-{u}\ar[rrrd]^-{u}\ar[rddd]_-{i}\ar[rdd]^-{i}\xtwocell[rrrd]{}<>{^<3>\quad\qquad D^0}&&&N\ar[rrd]^-{u1}\ar@/^17mm/[rrrrddd]^-{1}&&&&\ar@{}[lld]|*=0[@]{\cong}\\
&&&N\ar[rr]^-{1u}\ar@{}[u]|*[@d]{\cong}&&NN\ar[rrdd]^-{m}&&\\
&R\ar[rr]|-{1i}\ar[rru]^-{D}\xtwocell[rrrru]{}<>{^DD^0\quad}&&RR\ar[rrd]|-{i^*1}\ar[rru]|-{DD}\xtwocell[rrrrd]{}<>{^<-2>D^2\ }&&&&\\
&R\ar[rr]_-{i\ob 1}\ar[rru]|-{i1}\ar@{}[ruu]|<<<<<*[@]{\cong}&&R\ot R\ar[rr]_-{i\ot 1}\ar@{}[u]|*[@]{\cong}&&R\ar[rr]_-{D}&&N
}}}
\]
\[
\stackrel{\eqref{ax:SKM5}}{=}
\vcenter{\hbox{\xymatrix@!0@R=9mm@C=8mm{
I\ar[rrr]^-{u}\ar[rddd]_-{i}\ar[rrd]_-{i}\xtwocell[rrrr]{}<>{^<2>D^0\ }&&&N\ar[rrd]^-{1u}\ar@/^17mm/[rrrrddd]^-{1}&&&&\ar@{}[lld]|*=0[@]{\cong}\\
&&R\ar[rd]|-{1i}\ar[ru]_-{D}\xtwocell[rrr]{}<>{^DD^0\quad}&&&NN\ar[rrdd]^-{m}&&\\
&&&RR\ar[rrd]|-{i^*1}\ar[rru]|-{DD}\xtwocell[rrrrd]{}<>{^<-2>D^2\ }&&&&\\
&R\ar[rr]_-{i\ob 1}\ar[rru]|-{i1}\ar@{}[ruu]|*[@]{\cong}&&R\ot R\ar[rr]_-{i\ot 1}\ar@{}[u]|*[@]{\cong}&&R\ar[rr]_-{D}&&N
}}}
\]
\[
\stackrel{\eqref{ax:OM2}}{=}\!\!
\vcenter{\hbox{\xymatrix@!0@R=9mm@C=7mm{
I\ar[rrr]^-{u}\ar[rddd]_-{i}\ar[rrd]_-{i}\xtwocell[rrrr]{}<>{^<2>D^0\ }&&&N\ar@/^15mm/[rrrrddd]^-{1}&&&&\\
&&R\ar[rd]|-{1i}\ar[ru]_-{D}\ar[rrd]^-{i^*}\ar@/^12mm/[rrrdd]^-{1}\xtwocell[rrrdd]{}<>{^<-4>\varepsilon }&&&&&\\
&&&RR\ar[rrd]|-{i^*1}&I\ar[rd]^-{i}&&&\\
&R\ar[rr]_-{i\ob 1}\ar[rru]|-{i1}\ar@{}[ruu]|*[@]{\cong}&&R\ot R\ar[rr]_-{i\ot 1}\ar@{}[u]|*[@]{\cong}&&R\ar[rr]_-{D}&&N
}}}
=\!\!
\vcenter{\hbox{\xymatrix@!0@R=9mm@C=7mm{
I\ar@/^10mm/[rrrrrrrddd]^-{u}\ar[rddd]_-{i}\ar[rrrrrddd]^-{i}\ar[rrdd]_-{n}\xtwocell[rrrrrrrddd]{}<>{^<-2>D^0\ }&&&N&&&&\\
&&&&&&&\\
&&R\ot R\ar[rd]^-{1}\ar[ld]_-{i\ot 1}\xtwocell[rd]{}<>{^<2>\varepsilon\ob 1\ }&&&&&\\
&R\ar[rr]_-{i\ob 1}&&R\ot R\ar[rr]_-{i\ot 1}&&R\ar[rr]_-{D}&&N
}}}
\]
The axiom \eqref{ax:X2} holds.
\begin{gather*}
\vcenter{\hbox{\xymatrix@!0@R=7mm@C=6mm{
&&&&\ar@{}[lddd]|*=0[@]{\cong}&&NN\ar[rrrd]^-{m}&&&\\
&&&NN\ar[dd]_-{u11}\ar[rrru]^-{1}&&&&&&N\\
RR\ar[rrru]^-{DD}\ar[rdddd]_-{i\ob 11}\ar[rrddd]^-{i11}\xtwocell[rrrrd]{}<>{^D^0DD\qquad}&&&&&&&&&\\
&&&NNN\ar[rrruuu]|-{m1}&&&&&&\\
&&&&&&&&&\\
&\ar@{}[ru]|<*[@ru]{\cong}&RRR\ar[ddddd]|-{i^*11}\ar[ruu]|-{DDD}\xtwocell[rrruu]{}<>{^<1>D^2D\quad}&&&&&&&\\
&R\ot RR\ar[rdddd]_-{i\ot 11}&&&{\xtwocell[rrru]{}<>{^<2>D^2\quad}}&&&&&\\
&&&&&&&&&\\
&&&&&&&&&\\
&&&&&&&&&\\
&&RR\ar[rrrd]_-{i^*1}\ar[rrrruuuuuuuuuu]_-{DD}&&&&&&&\\
&&&&&R\ar[rrrruuuuuuuuuu]_-{D}&&&&
}}}
\!\!\!\!\!\!\!\!\!\!\!\!\!\stackrel{\eqref{ax:OM1}}{=}\;
\vcenter{\hbox{\xymatrix@!0@R=7mm@C=6mm{
&&&&\ar@{}[lddd]|*=0[@]{\cong}&&NN\ar[rrrd]^-{m}&&&\\
&&&NN\ar[dd]_-{u11}\ar[rrru]^-{1}&&&&&&N\\
RR\ar[rrru]^-{DD}\ar[rdddd]_-{i\ob 11}\ar[rrddd]^-{i11}\xtwocell[rrrrd]{}<>{^D^0DD\qquad}&&&&&&&&&\\
&&&NNN\ar[rrruuu]|-{m1}\ar[rrrd]^-{1m}&&&&&&\\
&&&&&&NN\ar[rrruuu]^-{m}\ar@{}[uuuu]|*[@d]{\cong}&&&\\
&\ar@{}[ru]|<*[@ru]{\cong}&RRR\ar[ddddd]|-{i^*11}\ar[ruu]|-{DDD}\ar[rrrd]_-{1i^*1}\xtwocell[rrrruu]{}<>{^<1>DD^2\quad\ }&&&&&&&\\
&R\ot RR\ar[rdddd]_-{i\ot 11}&&&&RR\ar[ddddd]_-{i^*1}\ar[ruu]^-{DD}\xtwocell[rrruu]{}<>{^<1>D^2\quad}&&&&\\
&&&&&&&&&\\
&&&&&&&&&\\
&&&&&&&&&\\
&&RR\ar[rrrd]_-{i^*1}\ar@{}[rrruuuu]|*[@]{\cong}&&&&&&&\\
&&&&&R\ar[rrrruuuuuuuuuu]_-{D}&&&&
}}}
\\
\\
\\
\stackrel{\eqref{ax:SKM3}}{=}\;
\vcenter{\hbox{\xymatrix@!0@R=7mm@C=6mm{
&&&&&&NN\ar[rrrd]^-{m}&&&\\
&&&NN\ar[dd]_-{u11}\ar[rrrd]^-{m}\ar[rrru]^-{1}&&&&\ar@{}[lddd]|*=0[@]{\cong}&&N\\
RR\ar[rrru]^-{DD}\ar[rdddd]_-{i\ob 11}\ar[rrddd]^-{i11}\xtwocell[rrrrd]{}<>{^D^0DD\qquad}&&&&&&N\ar[dd]_-{u1}\ar[rrru]^-{1}\ar@{}[llld]|*=0[@]{\cong}&&&\\
&&&NNN\ar[rrrd]^-{1m}&&&&&&\\
&&&&&&NN\ar[rrruuu]|-{m}&&&\\
&\ar@{}[ru]|<*[@ru]{\cong}&RRR\ar[ddddd]|-{i^*11}\ar[ruu]|-{DDD}\ar[rrrd]_-{1i^*1}\xtwocell[rrrruu]{}<>{^<1>DD^2\quad\ }&&&&&&&\\
&R\ot RR\ar[rdddd]_-{i\ot 11}&&&&RR\ar[ddddd]_-{i^*1}\ar[ruu]^-{DD}\xtwocell[rrruu]{}<>{^<1>D^2\quad\ }&&&&\\
&&&&&&&&&\\
&&&&&&&&&\\
&&&&&&&&&\\
&&RR\ar[rrrd]_-{i^*1}\ar@{}[rrruuuu]|*[@]{\cong}&&&&&&&\\
&&&&&R\ar[rrrruuuuuuuuuu]_-{D}&&&&
}}}
\\
\\
\\
=\!\!
\vcenter{\hbox{\xymatrix@!0@R=7mm@C=6mm{
&&&&&&NN\ar[rrrd]^-{m}&&&\\
&&&NN\ar[rrrd]^-{m}\ar[rrru]^-{1}&&&&\ar@{}[lddd]|*=0[@]{\cong}&&N\\
RR\ar[rrru]^-{DD}\ar[rdddd]_-{i\ob 11}\ar[rrddd]^-{i11}\ar[rrrd]^-{i^*1}\xtwocell[rrrrrr]{}<>{^D^2\ }&&&&&&N\ar[dd]_-{u1}\ar[rrru]^-{1}&&&\\
&&&R\ar[rrru]^-{D}\ar[rrddd]^-{i1}\xtwocell[rrrrd]{}<>{^D^0D\quad\ }&&&&&&\\
&&&&&&NN\ar[rrruuu]|-{m}&&&\\
&\ar@{}[ru]|<*[@ru]{\cong}&RRR\ar[ddddd]|-{i^*11}\ar[rrrd]_-{1i^*1}\ar@{}[ruu]|*[@]{\cong}&&&&&&&\\
&R\ot RR\ar[rdddd]_-{i\ot 11}&&&&RR\ar[ddddd]_-{i^*1}\ar[ruu]^-{DD}\xtwocell[rrruu]{}<>{^<1>D^2\quad}&&&&\\
&&&&&&&&&\\
&&&&&&&&&\\
&&&&&&&&&\\
&&RR\ar[rrrd]_-{i^*1}\ar@{}[rrruuuu]|*[@]{\cong}&&&&&&&\\
&&&&&R\ar[rrrruuuuuuuuuu]_-{D}&&&&
}}}
\!\!\!\!\!\!\!\!\!=\;\;
\vcenter{\hbox{\xymatrix@!0@R=7mm@C=6mm{
&&&&&&NN\ar[rrrd]^-{m}&&&\\
&&&NN\ar[rrrd]^-{m}\ar[rrru]^-{1}&&&&\ar@{}[lddd]|*=0[@]{\cong}&&N\\
RR\ar[rrru]^-{DD}\ar[rdddd]_-{i\ob 11}\ar[rrrd]^-{i^*1}\xtwocell[rrrrrr]{}<>{^D^2\ }&&&&&&N\ar[dd]_-{u1}\ar[rrru]^-{1}&&&\\
&&&R\ar[rrru]^-{D}\ar[rrddd]^-{i1}\ar[rdddd]_-{i\ob 1}\xtwocell[rrrrd]{}<>{^D^0D\quad\ }&&&&&&\\
&&&&&&NN\ar[rrruuu]|-{m}&&&\\
&&&&&&&&&\\
&R\ot RR\ar[rdddd]_-{i\ot 11}\ar[rrrd]_-{1i^*1}\ar@{}[rruuu]|*[@]{\cong}&&&\ar@{}[ru]|<*[@ru]{\cong}&RR\ar[ddddd]|-{i^*1}\ar[ruu]^-{DD}\xtwocell[rrruu]{}<>{^<1>D^2\quad}&&&&\\
&&&&R\ot R\ar[rdddd]_-{i\ot 1}&&&&&\\
&&&&&&&&&\\
&&&&&&&&&\\
&&RR\ar[rrrd]_-{i^*1}\ar@{}[rruuu]|*[@]{\cong}&&&&&&&\\
&&&&&R\ar[rrrruuuuuuuuuu]_-{D}&&&&
}}}
\end{gather*}
The axiom \eqref{ax:X3} holds.
\begin{gather*}
\vcenter{\hbox{\xymatrix@!0@R=7mm@C=6mm{
&&&&\ar@{}[lddd]|*=0[@]{\cong}&&NN\ar[rrrd]^-{m}&&&\\
&&&NN\ar[dd]_-{1u1}\ar[rrru]^-{1}&&&&&&N\\
RR\ar[rrru]^-{DD}\ar[rdddd]_-{1i\ob 1}\ar[rrddd]^-{1i1}\xtwocell[rrrrd]{}<>{^DD^0D\qquad}&&&&&&&&&\\
&&&NNN\ar[rrruuu]|-{1m}&&&&&&\\
&&&&&&&&&\\
&\ar@{}[ru]|<*[@ru]{\cong}&RRR\ar[ddddd]|-{1i^*1}\ar[ruu]|-{DDD}\xtwocell[rrruu]{}<>{^<1>DD^2\quad\ }&&&&&&&\\
&RR\ot R\ar[rdddd]_-{1i\ot 1}&&&{\xtwocell[rrru]{}<>{^<2>D^2\quad}}&&&&&\\
&&&&&&&&&\\
&&&&&&&&&\\
&&&&&&&&&\\
&&RR\ar[rrrd]_-{i^*1}\ar[rrrruuuuuuuuuu]_-{DD}&&&&&&&\\
&&&&&R\ar[rrrruuuuuuuuuu]_-{D}&&&&
}}}
\!\!\!\!\!\!\!\!\!\!\!\!\!\stackrel{\eqref{ax:OM1}}{=}\;
\vcenter{\hbox{\xymatrix@!0@R=7mm@C=6mm{
&&&&\ar@{}[lddd]|*=0[@]{\cong}&&NN\ar[rrrd]^-{m}&&&\\
&&&NN\ar[dd]_-{1u1}\ar[rrru]^-{1}&&&&&&N\\
RR\ar[rrru]^-{DD}\ar[rdddd]_-{1i\ob 1}\ar[rrddd]^-{1i1}\xtwocell[rrrrd]{}<>{^DD^0D\qquad}&&&&&&&&&\\
&&&NNN\ar[rrruuu]|-{1m}\ar[rrrd]^-{m1}&&&&&&\\
&&&&&&NN\ar[rrruuu]^-{m}\ar@{}[uuuu]|*[@d]{\cong}&&&\\
&\ar@{}[ru]|<*[@ru]{\cong}&RRR\ar[ddddd]|-{1i^*1}\ar[ruu]|-{DDD}\ar[rrrd]_-{i^*11}\xtwocell[rrrruu]{}<>{^<1>D^2D\quad\ }&&&&&&&\\
&RR\ot R\ar[rdddd]_-{1i\ot 1}&&&&RR\ar[ddddd]_-{i^*1}\ar[ruu]^-{DD}\xtwocell[rrruu]{}<>{^<1>D^2\quad}&&&&\\
&&&&&&&&&\\
&&&&&&&&&\\
&&&&&&&&&\\
&&RR\ar[rrrd]_-{i^*1}\ar@{}[rrruuuu]|*[@]{\cong}&&&&&&&\\
&&&&&R\ar[rrrruuuuuuuuuu]_-{D}&&&&
}}}
\\
\\
\\
\stackrel{\eqref{ax:SKM2}}{=}\;
\vcenter{\hbox{\xymatrix@!0@R=7mm@C=6mm{
&&&&&&NN\ar[rrrd]^-{m}&&&\\
&&&NN\ar[dd]_-{1u1}\ar[rrru]^-{1}\ar@/^4mm/[rrrddd]^-{1}&&&&&&N\\
RR\ar[rrru]^-{DD}\ar[rdddd]_-{1i\ob 1}\ar[rrddd]^-{1i1}\xtwocell[rrrrd]{}<>{^DD^0D\qquad}&&&&&&\ar@{}[llld]|*=0[@]{\cong}&&&\\
&&&NNN\ar[rrrd]^-{m1}&&&&&&\\
&&&&&&NN\ar[rrruuu]^-{m}&&&\\
&\ar@{}[ru]|<*[@ru]{\cong}&RRR\ar[ddddd]|-{1i^*1}\ar[ruu]|-{DDD}\ar[rrrd]_-{i^*11}\xtwocell[rrrruu]{}<>{^<1>D^2D\quad\ }&&&&&&&\\
&RR\ot R\ar[rdddd]_-{1i\ot 1}&&&&RR\ar[ddddd]_-{i^*1}\ar[ruu]^-{DD}\xtwocell[rrruu]{}<>{^<1>D^2\quad}&&&&\\
&&&&&&&&&\\
&&&&&&&&&\\
&&&&&&&&&\\
&&RR\ar[rrrd]_-{i^*1}\ar@{}[rrruuuu]|*[@]{\cong}&&&&&&&\\
&&&&&R\ar[rrrruuuuuuuuuu]_-{D}&&&&
}}}
\\
\\
\\
\stackrel{\eqref{ax:OM2}}{=}\!\!\!\!\!\!\!\!\!\!
\vcenter{\hbox{\xymatrix@!0@R=7mm@C=6mm{
&&&&&&NN\ar[rrrd]^-{m}&&&\\
&&&NN\ar[rrru]^-{1}\ar@/^4mm/[rrrddd]^-{1}&&&&&&N\\
RR\ar[rrru]^-{DD}\ar[rdddd]_-{1i\ob 1}\ar[rrddd]|-{1i1}\ar[rrdd]^-{i^*1}\ar@/^7mm/[rrrrrdddd]^-{1}\xtwocell[rrrrrdddd]{}<>{^<-2>\varepsilon 1}&&&&&&&&&\\
&&&&&&&&&\\
&&R\ar[rrrdd]^-{i1}&&&&NN\ar[rrruuu]^-{m}&&&\\
&\ar@{}[ru]|<*[@ru]{\cong}&RRR\ar[ddddd]|-{1i^*1}\ar[rrrd]_-{i^*11}&&&&&&&\\
&RR\ot R\ar[rdddd]_-{1i\ot 1}&&&&RR\ar[ddddd]_-{i^*1}\ar[ruu]^>>>>{DD}\xtwocell[rrruu]{}<>{^<1>D^2\quad}&&&&\\
&&&&&&&&&\\
&&&&&&&&&\\
&&&&&&&&&\\
&&RR\ar[rrrd]_-{i^*1}\ar@{}[rrruuuu]|*[@]{\cong}&&&&&&&\\
&&&&&R\ar[rrrruuuuuuuuuu]_-{D}&&&&
}}}
\!\!\!\!\!\!\!\!\!=\;
\vcenter{\hbox{\xymatrix@!0@R=7mm@C=6mm{
&&&&&&NN\ar[rrrd]^-{m}&&&\\
&&&NN\ar[rrru]^-{1}\ar@/^4mm/[rrrddd]^-{1}&&&&&&N\\
RR\ar[rrru]^-{DD}\ar[rdddd]_-{1i\ob 1}\ar@/^7mm/[rrrrrdddd]^-{1}&&&&&&&&&\\
&&&&&&&&&\\
&&&&&&NN\ar[rrruuu]^-{m}&&&\\
&&&&&&&&&\\
&RR\ot R\ar[rdddd]_-{1i\ot 1}\ar@/^/[rrdd]!<4mm,0mm>^-{1}&&&&RR\ar[ddddd]|-{i^*1}\ar[ruu]^>>>>{DD}\xtwocell[rrruu]{}<>{^<1>D^2\quad}&&&&\\
&{\xtwocell[rrrd]{!<2mm,1mm>}<>{^<1>1\varepsilon\ob 1\ \ }}&&&&&&&&\\
&&&**{!<-4mm,0mm>}RR\ot R\ar[]!<4mm,0mm>;[rrddd]^-{e1}&&&&&&\\
&&&&&&&&&\\
&&RR\ar[rrrd]_-{i^*1}\ar[ruu]!<4mm,0mm>|-{1i\ob 1}&&&&&&&\\
&&&\ar@{}[ruuu]|*=0[@]{\cong}&&R\ar[rrrruuuuuuuuuu]_-{D}&&&&
}}}
\end{gather*}
Hence $F$ is surjective on objects. Now, these actions defined purely in terms of the opmonoidal constraints turn every opmonoidal cell $\xymatrix@1@C=5mm{\gamma:D\ar[r]&D'}$ into an arrow in $\mathcal{X}(R,N)$, because as one can see below the axiom \eqref{ax:X4} holds.
\begin{gather*}
\vcenter{\hbox{\xymatrix@!0@R=10mm@C=8mm{
&&N\ar[rrd]^-{u1}\ar@/^18mm/[rrrrddd]^-{1}&&&&\ar@{}[lld]|*=0[@]{\cong}\\
&&&&NN\ar[rrdd]^-{m}&&\\
{\xtwocell[rrrrru]{}<>{^<-3>D^0D\quad\ }}&&RR\ar[rrd]|-{i^*1}\ar[rru]^-{DD}\xtwocell[rrrrd]{}<>{^<-2>D^2\ \ }&&&&\\
R\ar[rr]_-{i\ob 1}\ar[rru]^-{i1}\ar[rruuu]_-{D}\ar@/^10mm/[rruuu]^-{D'}\xtwocell[rruuu]{}<>{^<-3>\gamma}&&R\ot R\ar[rr]_-{i\ot 1}\ar@{}[u]|*[@]{\cong}&&R\ar[rr]_-{D}&&N
}}}
\stackrel{\eqref{ax:OM5}}{=}
\vcenter{\hbox{\xymatrix@!0@R=10mm@C=8mm{
&&N\ar[rrd]^-{u1}\ar@/^18mm/[rrrrddd]^-{1}&&&&\ar@{}[lld]|*=0[@]{\cong}\\
&&&&NN\ar[rrdd]^-{m}&&\\
{\xtwocell[rrru]{}<>{^<-2>D'^0D'\qquad}}&&RR\ar[rrd]|-{i^*1}\ar[rru]_-{DD}\ar@/^7mm/[rru]^-{D'D'}\xtwocell[rru]{}<>{^<-2>\gamma\gamma\ }\xtwocell[rrrrd]{}<>{^<-2>D^2\ \ }&&&&\\
R\ar[rr]_-{i\ob 1}\ar[rru]^-{i1}\ar@/^10mm/[rruuu]^-{D'}&&R\ot R\ar[rr]_-{i\ot 1}\ar@{}[u]|*[@]{\cong}&&R\ar[rr]_-{D}&&N
}}}
\\
\stackrel{\eqref{ax:OM4}}{=}\;\;
\vcenter{\hbox{\xymatrix@!0@R=10mm@C=8mm{
&&N\ar[rrd]^-{u1}\ar@/^18mm/[rrrrddd]^-{1}&&&&\ar@{}[lld]|*=0[@]{\cong}\\
&&&&NN\ar[rrdd]^-{m}&&\\
{\xtwocell[rrru]{}<>{^<-2>D'^0D'\qquad}}&&RR\ar[rrd]|-{i^*1}\ar@/^7mm/[rru]^-{D'D'}\xtwocell[rrrd]{}<>{^<-4>D'^2\ \ }&&&&\\
R\ar[rr]_-{i\ob 1}\ar[rru]^-{i1}\ar@/^10mm/[rruuu]^-{D'}&&R\ot R\ar[rr]_-{i\ot 1}\ar@{}[u]|*[@]{\cong}&&R\ar[rr]_-{D}\ar@/^6mm/[rr]!<-1mm,-1mm>^<<<<{D'}\xtwocell[rr]{}<>{^<-1.5>\gamma}&&N
}}}
\end{gather*}
Thus $F$ is full and therefore invertible.
\end{proof}

\begin{teo}\label{teo:OpmonadicityOpmonoidal}
Let $\mathcal{M}$ be an opmonadic-friendly monoidal bicategory in the sense of Definition~\ref{def:OpmonadicFriendly}. For every monoidale $(N,m,u)$, every biduality $R\dashv R\ot$, and every opmonadic adjunction $i\ob\dashv i\ot$ in $\mathcal{M}$
\[
\vcenter{\hbox{\xymatrix{
R\ot\dtwocell_{i\ob}^{i\ot}{'\dashv}\\
I
}}}
\]
precomposition with $\xymatrix@1@C=5mm{i\ob 1:R\ar[r]&R\ot R}$ defines an equivalence of categories.
\[
\Opmon(R\ot R,N)\simeq\SkOpmon(R,N)
\]
\end{teo}
\begin{proof}

By Proposition~\ref{prop:OneOpmonoidalArrow} $i\ob 1$ is an opmonoidal arrow, thus precomposition along this arrow in $\SkOpmon(\mathcal{M})$ is a well defined functor.
\[
\xymatrix@1@C=5mm{\SkOpmon(i\ob 1,N):\Opmon(R\ot R,N)\ar[r]&\SkOpmon(R,N)}
\]
Let $G$ be the composite of $\SkOpmon(i\ob 1,N)$ followed by the inverse of the isomorphism $F$ in Lemma~\ref{lem:OpmonadicityOpmonoidal} which equips an opmonoidal arrow with its canonical module structure \eqref{eq:redundant_action1} for the monad induced by $i\ob 1\dashv i\ot 1$. We shall now see that the functor $G$ is an equivalence of categories.
\[
\vcenter{\hbox{\xymatrix@!0@=55mm{
G:\Opmon(R\ot R,N)\ar[r]^-{\SkOpmon(i\ob 1,N)}&\SkOpmon(R,N)
}
\!\!\!\!
\xymatrix{
\ar[r]^-{F^{-1}}_-{\cong}&\mathcal{X}(R,N)
}}}
\]
Faithfulness of $G$ follows easily because precomposing with the opmonadic arrow $i\ot 1$ is faithful in $\mathcal{M}$, and since the forgetful functor
\[
\xymatrix{\SkOpmon(R\ot R,N)\ar[r]&\mathcal{M}(R\ot R,N)}
\]
is faithful so is precomposing with $i\ot 1$ in $\SkOpmon(\mathcal{M})$. Now, the functor $G$ is essentially surjective on objects and full, mainly due to the opmonadicity of $i\ob\dashv i\ot$. Remember that opmonadicity in $\mathcal{M}$ is preserved by tensoring with objects, so $i\ob 1\dashv i\ot 1$ is opmonadic, and so for an object $(D,\varphi)$ in $\mathcal{X}(R,N)$, the action cell $\varphi$ induces an arrow $\xymatrix@1@C=5mm{C:R\ot R\ar[r]&N}$ and an isomorphism
\begin{equation}\label{iso:opmonadicity1}
\vcenter{\hbox{\xymatrix@!0@C=15mm{
R\ar[rd]_-{i\ob 1}\ar[rr]^-{D}&&N\\
&R\ot R\ar[ru]_-{C}\ar@{}[u]|-*[@]{\cong}&
}}}
\end{equation}
such that the following equation holds.
\begin{equation}\label{iso:opmonadicity1:module}
\vcenter{\hbox{\xymatrix@!0@=11mm{
R\ar[rd]_-{i\ob 1}\ar[rrrr]<1mm>^D&&&&N\\
&R\ot R\ar[rd]_-{i\ot 1}\ar[rr]^-{1}\xtwocell[rr]{}<>{^<2>\varepsilon\ob 1\quad}&\ar@{}[u]|*[@]{\cong}&R\ot R\ar[ru]_-{C}&\\
&&R\ar[ru]_-{i\ob 1}&&
}}}
\quad=\quad
\vcenter{\hbox{\xymatrix@!0@=11mm{
R\ar[rd]_-{i\ob 1}\ar[rrrr]<1mm>^D\xtwocell[rrrd]{}<>{^\varphi}&&&&N\\
&R\ot R\ar[rd]_-{i\ot 1}&&R\ot R\ar[ru]_-{C}\ar@{}[ul]|<<*=0[@]{\cong}&\\
&&R\ar[ru]_-{i\ob 1}\ar@/^8mm/[rruu]^-{D}&&
}}}
\end{equation}

Now, since $i\ob 1\dashv i\ot 1$ is an opmonadic adjunction, its counit $\varepsilon\ob 1$ is a coequaliser, and then, by hypothesis, the cell
\begin{equation}\label{cell:coequaliser}
\vcenter{\hbox{\xymatrix@!0@C=11mm@R=6mm{
&I\ar[ld]_-{n}&\\
R\ot R\ar[dd]_-{i\ot 1}\ar@/^/[rddd]^-{1}\xtwocell[rddd]{}<>{^<1>\varepsilon\ob 1}&&\\
&&\\
R\ar[rd]_-{i\ob 1}&&N\\
&R\ot R\ar[ru]_-{C}&
}}}
\end{equation}
is the coequaliser of the parallel pair of cells below.
\begin{equation}\label{cell:cofork}
\vcenter{\hbox{\xymatrix@!0@=10mm{
&R\ot R\ar[ld]_-{i\ot 1}\ar@/^5mm/[ldd]^-{1}\xtwocell[ldd]{}<>{^\varepsilon\ob 1}&I\ar[l]_-{n}&\\
R\ar[d]_-{i\ob 1}&&&\\
R\ot R\ar[rd]_-{i\ot 1}&&&N\\
&R\ar[r]_-{i\ob 1}&R\ot R\ar[ru]_-{C}&
}}}
\qquad\qquad\qquad
\vcenter{\hbox{\xymatrix@!0@=10mm{
&R\ot R\ar[ld]_-{i\ot 1}&I\ar[l]_-{n}&\\
R\ar[d]_-{i\ob 1}&&&\\
R\ot R\ar[rd]_-{i\ot 1}\ar@/^5mm/[rrd]^-{1}\xtwocell[rrd]{}<>{^\varepsilon\ob 1\ }&&&N\\
&R\ar[r]_-{i\ob 1}&R\ot R\ar[ru]_-{C}&
}}}
\end{equation}
Taking this into account, one may read axiom \eqref{ax:X1} for $(D,\varphi)$ as saying that precomposing the cell below with each of the parallel cells \eqref{cell:cofork} gives the same result.
\[
\vcenter{\hbox{\xymatrix@!0@C=15mm{
&I\ar[ld]_-{n}\ar[rddd]^-u\ar[dddl]^-i\xtwocell[rddd]{}<>{^<5>D^0}&\\
R\ot R\ar[dd]_-{i\ot 1}\ar@{}[rd]|<<<*[@]{\cong}&&\\
&&\\
R\ar[rd]_-{i\ob 1}\ar[rr]^-{D}&&N\\
&R\ot R\ar[ru]_-{C}\ar@{}[u]|*[@]{\cong}&
}}}
\]
Ergo, by universality of the coequaliser \eqref{cell:coequaliser} there exists a cell
\[
\vcenter{\hbox{\xymatrix@!0@R=15mm@C=9mm{&I\ar[dl]_n\ar[rd]^-u\xtwocell[rd]{}<>{^<2>C^0}&\\
R\ot R\ar[rr]_C&&N
}}}
\]
such that the following equation holds.
\begin{equation}\label{iso:opmonadicity1:unit}
\vcenter{\hbox{\xymatrix@!0@C=10mm@R=8mm{
&&I\ar[ddddll]_-i\ar[dd]^-n\ar[rrdddd]^-u&&\\
&\ar@{}[rdd]|<<<<<<<*[@]{\cong}&&&\\
&&R\ot R\ar[ddll]|-{i\ot 1}\ar[dd]^-{1}\xtwocell[rdd]{}<>{^<6>\varepsilon\ob 1}\xtwocell[rdd]{}<>{^<-1>C^0}&&\\
&&&&\\
R\ar[rr]_{i\ob 1}&&R\ot R\ar[rr]_C&&N
}}}
\quad=\quad
\vcenter{\hbox{\xymatrix@!0@C=10mm@R=8mm{
&&I\ar[ddddll]_i\ar[rrdddd]^-u&&\\
&{\xtwocell[rrdd]{}<>{^D^0}}&&&\\
&&&&\\
&&\ar@{}[d]|<<*[@]{\cong}&&\\
R\ar[rr]_{i\ob 1}\ar@/^10mm/[rrrr]^-{D}&&R\ot R\ar[rr]_C&&N
}}}
\end{equation}

The axioms \eqref{ax:X2} and \eqref{ax:X3} say precisely that $D^2$ is a morphism of modules for the monads induced by the opmonadic adjunctions
\[
\vcenter{\hbox{\xymatrix{
R\ot RR\dtwocell_{i\ob 11\quad}^{\quad i\ot 11}{'\dashv}\\
RR
}}}
\qquad\qquad
\vcenter{\hbox{\xymatrix{
RR\ot R\dtwocell_{1i\ob 1\quad}^{\quad 1i\ot 1}{'\dashv}\\
RR
}}}
\]
with the obvious actions on the target of $D^2$ for each of the monads, and the following actions on the source of $D^2$.
\[
\vcenter{\hbox{\xymatrix@!0@C=13mm{
&R\ar[rd]^-{i\ob 1}\ar@/^8mm/[rrrdd]^-{D}\ar@{}[dd]|*[@]{\cong}\xtwocell[rrrdd]{}<>{^<-2>\varphi}&&&\\
&&R\ot R\ar[rd]^-{i\ot 1}\ar@{}[d]|*[@]{\cong}&&\\
RR\ar[r]_-{i\ob 11}\ar[ruu]^-{i^*1}&R\ot RR\ar[r]_-{i\ot 11}\ar[ru]^-{1i^*1}&RR\ar[r]_-{i^*1}&R\ar[r]_-{D}&N}
\qquad
\xymatrix@!0@C=6.5mm{&&&&&&&&\\
&&&&&**[l]RR\ot R\ar[]!<-3mm,0mm>;[rd]^-{e1}&&&\\
RR\ar[rr]_-{1i\ob 1}\ar@/^20mm/[rrrrrr]!<3mm,0mm>^-{i^*1}&&RR\ot R\ar[rr]_-{1i\ot 1}\ar@/^5mm/[]!<-4mm,0mm>;[rrru]^-{1}\xtwocell[rrru]{}<>{^<-1>1\varepsilon\ob 1\quad}&\ar@{}[uu]|>>*[@]{\cong}&RR\ar[rr]_-{i^*1}\ar[]!<-3mm,0mm>;[ru]!<-2mm,0mm>|{1i\ot 1}&\ar@{}[u]|<<<*[@]{\cong}&R\ar[rr]_-{D}&&N
}}}
\]
Hence, one may read the axioms \eqref{ax:X2} and \eqref{ax:X3} for $(D,\varphi)$ as saying that $D^2$ is a morphism of modules for the monad induced by the adjunction $i\ob 1i\ob1\dashv i\ot 1i\ot 1$ (below left). This adjunction is opmonadic by hypothesis, and so it induces a cell $C^2$ (as shown on the right)
\[
\vcenter{\hbox{\xymatrix{
R\ot RR\ot R\dtwocell_{i\ob 1i\ob 1\qquad}^{\qquad i\ot 1i\ot 1}{'\dashv}\\
RR
}}}
\qquad\qquad\qquad
\vcenter{\hbox{\xymatrix@!0@=16mm{
R\ot RR\ot R\ar[r]^-{CC}\ar[d]_{1e1}\xtwocell[rd]{}<>{^C^2}&NN\ar[d]^-{m}\\
R\ot R\ar[r]_C&N
}}}
\]
such that the following equation holds.
\begin{equation}\label{iso:opmonadicity1:prod}
\vcenter{\hbox{\xymatrix@!0@C=20mm{
RR\ar[rd]|-{i\ob 1i\ob 1}\ar[dd]_{i^*1}\ar[rr]^-{DD}&&NN\ar[dd]^-{m}\\
&R\ot RR\ot R\ar[ru]_-{CC}\ar[dd]_{1e1}\ar@{}[u]|*[@]{\cong}\xtwocell[rd]{}<>{^C^2}&\\
R\ar[rd]_-{i\ot 1}\ar@{}[ru]|-*[@]{\cong}&&N\\
&R\ot R\ar[ru]_-{C}&
}}}
\quad=\quad
\vcenter{\hbox{\xymatrix@!0@C=20mm{
RR\ar[dd]_{i^*1}\ar[rr]^-{DD}\xtwocell[rrdd]{}<>{^D^2\ }&&NN\ar[dd]^-{m}\\
&&\\
R\ar[rd]_-{i\ot 1}\ar[rr]^-{DD}&&N\\
&R\ot R\ar[ru]_-{C}\ar@{}[u]|*[@u]{\cong}&
}}}
\end{equation}

To deduce that the data $(C,C^0,C^2)$ constitute an opmonoidal arrow, first take axiom \eqref{ax:OM1} for the data $(C,C^0,C^2)$ and apply the faithful functor given by precomposition with the opmonadic arrow
\[
\vcenter{\hbox{\xymatrix{
i\ob 1i\ob 1i\ob 1:RRR\ar[r]&R\ot RR\ot RR\ot R
}}}
\]
to both sides of the axiom; this produces the two sides of axiom \eqref{ax:OM1} for $D$, which are equal. Then, precompose both sides of axiom \eqref{ax:OM2} for the data $(C,C^0,C^2)$ with the opmonadic arrow $\xymatrix@1@C=5mm{i\ob 1:R\ar[r]&R\ot R}$, and also with the epimorphic cell
\[
\vcenter{\hbox{\xymatrix@!0@=12mm{
&R\ot RR\ot R\ar[rrd]^-{1}\ar[ld]_-{11i\ot 1}\xtwocell[rd]{}<>{^<2>11\varepsilon\ob 1\quad}&&\\
R\ot RR\ar[rrr]_-{11i\ob 1}&&&R\ot RR\ot R
}}}
\]
at $\id_{R\ot RR\ot R}$; this produces the two sides of axiom \eqref{ax:OM2} for $D$, which are equal. And finally, precompose the two sides of axiom \eqref{ax:OM3} for the data $(C,C^0,C^2)$ with the opmonadic arrow $\xymatrix@1@C=5mm{i\ob 1:R\ar[r]&R\ot R}$, and substitute the cell below for the identity on $i\ob 1$; this produces both sides of axiom \eqref{ax:OM3} for $D$, which are equal.
\[
\vcenter{\hbox{\xymatrix@!0@=20mm{
R\ar[r]^-{i\ob 1}\ar[rd]_-{1}\xtwocell[rd]{}<>{^<-3>\eta\ob 1\ }&R\ot R\ar[d]|-{i\ot 1}\ar[rd]^-{1}\xtwocell[rd]{}<>{^<3>\varepsilon\ob 1\ }&\\
&R\ar[r]_-{i\ob 1}&R\ot R
}}}
=\quad
\vcenter{\hbox{\xymatrix@!0@R=10mm@C=12mm{
R\ar[rr]^-{i\ob 1}\ar[rdd]_-{i1}\ar@/_15mm/[rrdddd]_-{1}\xtwocell[rrdddd]{}<>{^<4>\eta 1}&&R\ot R\ar[dd]_-{i11}\ar[rd]^-{n11}\ar@/^20mm/[rrdddd]^-{1}&&\\
&&{\xtwocell[rddd]{}<>{^<-4>\varepsilon\ob 111}}&R\ot RR\ot R\ar[ld]|-{i\ot 111}\ar[rdd]^-{1}&\\
&**[l]RR\ar[r]^-{1i\ob 1}\ar[rdd]_-{i^*1}&RR\ot R\ar[dd]|-{e1}\ar[rrd]_-{i\ob 111}&&\\
&&&&**[l]R\ot RR\ot R\ar[d]_-{1e1}\\
&&R\ar[rr]_-{i\ob 1}&&R\ot R
}}}
\]
We have now built an opmonoidal arrow $\xymatrix@1@C=5mm{C:R\ot R\ar[r]&N}$ for every $D$ and the isomorphism \eqref{iso:opmonadicity1} reads as $G(C)\cong D$; moreover, this isomorphism is in $\mathcal{X}(R,N)$ by \eqref{iso:opmonadicity1:module}, \eqref{iso:opmonadicity1:unit}, and \eqref{iso:opmonadicity1:prod}, therefore $G$ is essentially surjective. Now, let $\xymatrix@1@C=5mm{\gamma:D\ar[r]&D'}$ be an arrow in $\mathcal{X}(R,N)$; by the opmonadicity of $i\ob 1$ axiom $\eqref{ax:X4}$ implies the existence of a cell $\xymatrix@1@C=5mm{\xi:C\ar[r]&C'}$ in $\mathcal{M}$ such that the following equation holds.
\[
\vcenter{\hbox{\xymatrix@!0@C=15mm{
R\ar[rd]_-{i\ob 1}\ar[rr]^-{D}\ar@/^10mm/[rr]<1mm>^-{D'}\xtwocell[rr]{}<>{^<-4>\gamma}&&N\\&R\ot R\ar[ru]_-{C}\ar@{}[u]|-*[@]{\cong}&
}}}
\quad=\quad
\vcenter{\hbox{\xymatrix@!0@C=15mm{
R\ar[rd]_-{i\ob 1}\ar@/^10mm/[rr]<1mm>^-{D'}\ar@{}[r]|>>>>>*[@u]{\cong}&&N\\&R\ot R\ar[ru]_-{C}\ar@/^8mm/[ru]^-{C'}\xtwocell[ru]{}<>{^<-2>\xi}&
}}}
\]
To prove that $\xi$ is an opmonoidal arrow, first precompose the two sides of axiom \eqref{ax:OM4} for $\xi$ with the opmonadic arrow $\xymatrix@1@C=5mm{i\ob 1i\ob 1:RR\ar[r]&R\ot RR\ot R}$; this produces each side of axiom $\eqref{ax:OM4}$ for $\gamma$, which holds true. And then, precompose the two sides of axiom \eqref{ax:OM5} for $\xi$ with the epimorphic cell $\varepsilon\ob 1$ at $R\ot R$; this produces each side of axiom $\eqref{ax:OM5}$ for $\gamma$, which holds true. Therefore $G$ is full, and consequently is an equivalence.
\end{proof}

\begin{rem}
Suppose that $\mathcal{M}$ is a right autonomous opmonadic-friendly monoidal bicategory and that for every object there is an opmonadic adjunction $i\ob\dashv i\ot$, as is the case in the examples $\mathcal{M}=\Mod_k$ and $\mathcal{M}=\Span^{\mathrm{co}}$. Then the equivalence in Theorem~\ref{teo:OpmonadicityOpmonoidal} suggests that the assignation $\xymatrix@1@C=5mm{R\:\ar@{|->}[r]&R\ot R}$ behaves as a partial left adjoint to the forgetful functor, or as a strictification of the skew monoidal structure on $R$ into the monoidal structure of $R\ot R$.
\[
\vcenter{\hbox{\xymatrix{
R\ot R\\
R\ar@{|->}[u]
}}}
\vcenter{\hbox{\xymatrix{
&\Opmon\mathcal{M}\ar@/^3mm/[d]^{\text{forget}}\ar@{}[d]|\dashv\\
&\SkOpmon\mathcal{M}\ar@/^3mm/@{..>}[u]
}}}
\]
\end{rem}
\section{Oplax Actions}\label{sec:OplaxActionsOpmonadicity}
In this section we introduce a new concept that plays a central role in this paper, we call it ``oplax action''. Monoids and actions with respect to a monoid may be defined in a context as general as a monoidal category, but here we work one dimension higher: in the context of a monoidal bicategory. Because of the extra dimension, there are various ways to generalise the concept of a monoid in a monoidal bicategory. Examples of such generalisations are monoidales and skew monoidales which we have used earlier. In these examples the associative and unit laws do not hold strictly as they do with monoids, but only up to a cell satisfying some coherence axioms; for monoidales these cells must be isomorphisms, and for skew monoidales there is no such restriction. Oplax actions are defined with respect to a fixed skew monoidale, they generalise actions with respect to a monoid in a similar way as skew monoidales generalise monoids. That is, the associative and unit laws hold up to a not necessarily invertible cell satisfying some coherence axioms. Syntactically there is no distinction between ``actions'' and ``modules'' whatever the context, but their spirit is slightly different, the former focuses on the arrow bit while the latter focuses on the object bit. During this research having the arrow perspective proves to be useful. Oplax actions arise in the following way: we know that a bialgebroid corresponds to an opmonoidal monad on $R$-$\Mod$-$R$ which in turn corresponds to a skew monoidal structure on $\Mod$-$R$. We also know that bialgebroids are coalgebroids with some additional structure. One is led to ask what happens when we focus only on the coalgebroid bit of a bialgebroid in the correspondences above. It is known that a coalgebroid corresponds to an opmonoidal functor $\xymatrix@1@C=5mm{R\text{-}\Mod\text{-}R\ar[r]&S\text{-}\Mod\text{-}S}$, and in this section we prove that these opmonoidal functors also correspond to oplax action structures on $\Mod$-$S$ with respect to a particular skew monoidal structure on $\Mod$-$R$.

\begin{defi}\label{def:OplaxAction}
Let $(M,u,m,\alpha,\lambda,\rho)$ be a right skew monoidale, and let $A$ be an object in $\mathcal{M}$. An \emph{oplax right $M!$-action on $A$} consists of an arrow $\xymatrix@1@C=5mm{a:AM\ar[r]&A}$, an associator cell $a^2$, and a right unitor cell $a^0$ in $\mathcal{M}$
\[
\vcenter{\hbox{\xymatrix@!0@=15mm{
AMM\ar[r]^-{a1}\ar[d]_-{1m}\xtwocell[rd]{}<>{^a^2}&AM\ar[d]^-{a}\\
AM\ar[r]_-{a}&A
}
\qquad
\xymatrix@!0@=15mm{
AM\ar[d]_-{a}&A\ar[l]_-{1u}\ar[dl]^-{1}\xtwocell[ld]{}<>{^<2>\ \ a^0}\\
A
}}}
\]
satisfying the following three axioms.
\begin{align}
\tag{OLA1}\label{ax:OLA1}
\vcenter{\hbox{\xymatrix@!0@C=15mm{
&AMM\ar[rd]^-{a1}&\\
AMMM\ar[ru]^-{a11}\ar[dd]_-{11m}\ar[rd]_-{1m1}\xtwocell[rddd]{}<>{^\alpha}\xtwocell[rr]{}<>{^a^21\ \ }&&AM\ar[dd]^-{a}\\
&AMM\ar[dd]^-{1m}\ar[ru]^-{a1}\xtwocell[rd]{}<>{^a^2}&\\
AMM\ar[rd]_-{1m}&&A\\
&AM\ar[ru]_-{a}&
}}}
\quad&=\quad
\vcenter{\hbox{\xymatrix@!0@C=15mm{
&AMM\ar[rd]^-{a1}\ar[dd]_-{1m}\xtwocell[rddd]{}<>{^a^2}&\\
AMMM\ar[ru]^-{a11}\ar[dd]_-{11m}&&AM\ar[dd]^-{a}\\
&AMM\ar[rd]^-{a}&\\
AMM\ar[rd]_-{1m}\ar[ru]_-{a1}\ar@{}[ruuu]|*[@]{\cong}\xtwocell[rr]{}<>{^a^2\ }&&A\\
&AM\ar[ru]_-{a}&
}}}
\\
\tag{OLA2}\label{ax:OLA2}
\vcenter{\hbox{\xymatrix@!0@C=15mm{
&A\ar@/^7mm/[rddd]^-{1}&\\
AM\ar[ru]^-{a}\ar[dd]_-{11u}\xtwocell[rddd]{}<>{^1\rho\ }\ar@/^7mm/[rddd]^-{1}&&\\
&&\\
AMM\ar[rd]_-{1m}&&A\\
&AM\ar[ru]_-{a}&
}}}
\quad&=\quad
\vcenter{\hbox{\xymatrix@!0@C=15mm{
&A\ar[dd]_-{1u}\xtwocell[rddd]{}<>{^t^0}\ar@/^7mm/[rddd]^-{1}&\\
AM\ar[ru]^-{a}\ar[dd]_-{11u}&&\\
&AM\ar[rd]^-{a}&\\
AMM\ar[rd]_-{1m}\ar[ru]_-{a1}\ar@{}[ruuu]|*[@]{\cong}\xtwocell[rr]{}<>{^a^2\ }&&A\\
&AM\ar[ru]_-{a}&
}}}
\\
\tag{OLA3}\label{ax:OLA3}
\vcenter{\hbox{\xymatrix@!0@C=15mm{
&&\\
AM\ar@/^7mm/[rr]^-{1}\ar[rd]_-{1u1}\ar@/_7mm/[rddd]_-{1}\xtwocell[rddd]{}<>{^1\lambda\ }\xtwocell[rr]{}<>{^a^01\ \ }&&AM\ar[dd]^-{a}\\
&AMM\ar[ru]^-{a1}\ar[dd]^-{1m}\xtwocell[rd]{}<>{^a^2}&\\
&&A\\
&AM\ar[ru]_-{a}&
}}}
\quad&=\quad
\vcenter{\hbox{\xymatrix@!0@C=15mm{
&&\\
AM\ar@/^7mm/[rrdd]^-{a}\ar@/_7mm/[rrdd]_-{a}&&\\
&&\\
&&A\\
&&
}}}
\end{align}
\end{defi}
\begin{rem}\label{rem:SkewInducesAction}
One can similarly define \emph{oplax left actions}, \emph{lax right actions}, and \emph{lax left actions} with respect to a right skew monoidale or with respect to a left skew monoidale. If the associator and the left unitor are isomorphisms we speak of \emph{pseudo right actions}.

For every right skew monoidale $(M,u,m,\alpha,\lambda,\rho)$ there is a \emph{regular} oplax right $M$-action on $M$ given by its product arrow $\xymatrix@1@C=5mm{m:MM\ar[r]&M}$, associator cell $\alpha$, and right unitor cell $\rho$; axioms~\eqref{ax:OLA1}, \eqref{ax:OLA2}, and \eqref{ax:OLA3} are respectively axioms~\eqref{ax:SKM1}, \eqref{ax:SKM4}, and \eqref{ax:SKM5} for $M$. In particular, adjunctions and bidualities induce oplax right actions because, as explained in Lemma~\ref{lem:OneRightSkewMonoidale}, an adjunction
\[
\vcenter{\hbox{\xymatrix{
R\dtwocell_{i}^{i^*}{'\dashv}\\
I
}}}
\]
induces a right skew monoidal structure on $R$; and as explained in Remark~\ref{rem:PantsMonoidale}, a biduality $R\dashv R\ot$ induces the enveloping monoidale $R\ot R$. On the other hand, a biduality induces another pseudo action with respect to the enveloping monoidale but on the object $R$, it is given by the arrow $\xymatrix@1@C=5mm{e1:RR\ot R\ar[r]&R}$. When we think of the enveloping monoidale as an internal endo-hom monoidale the pseudo action $e1$ is the internal evaluation arrow.
\end{rem}

\begin{eje}
In $\Cat$ right oplax actions with respect to a skew monoidal category may be called \emph{oplax right actegories}. These are related to right skew monoidal bicategories as defined in \cite[Section 3]{Lack2014b} in the following way: a right skew monoidal bicategory $\mathcal{B}$ consists of a set of objects $X$, $Y$, and so on; for each object $X$ a right skew monoidal category $\mathcal{B}(X,X)$, and for each pair of objects $X$ and $Y$ a left-$\mathcal{B}(X,X)$ right-$\mathcal{B}(Y,Y)$ oplax actegory $\mathcal{B}(X,Y)$.
\end{eje}

\begin{rem}
Motivated by the previous example one may have chosen to name oplax actions as skew actions. But one needs to specify if it is a right action or left action, and also if it is right skew or left skew depending on the direction of the cells $a^2$ and $a^0$. Thus the full name for right oplax actions with this perspective would be right skew right actions which seems inconveniently long. Furthermore, a monoidale $M$ in $\mathcal{M}$ defines a pseudomonad by tensoring on the right $\xymatrix@1@C=5mm{\_\otimes M:\mathcal{M}\ar[r]&\mathcal{M}}$ whose oplax algebras are our oplax $M$-actions.
\end{rem}

\begin{defi} Let $a$ and $a'$ be oplax right $M$-actions on $A$. A \emph{cell of oplax right $M$-actions on $A$} from $a$ to $a'$ consists of a cell $\varphi$ in $\mathcal{M}$
\[
\vcenter{\hbox{\xymatrix{
**[l]AM\ar@/^3mm/[r]^-{a'}\ar@/_3mm/[r]_-{a}\xtwocell[r]{}<>{^\varphi}&A
}}}
\]
satisfying the following two conditions
\begin{align}
\tag{OLA4}\label{ax:OLA4}
\vcenter{\hbox{\xymatrix@!0@=19mm{
AMM\ar@/^3mm/[r]^-{a'1}\ar@/_3mm/[d]_-{1m}\xtwocell[rd]{}<>{^<-1>a'^2}&AM\ar@/^3mm/[d]^-{a'}\\
AM\ar@/^3mm/[r]^-{a'}\ar@/_3mm/[r]_-{a}\xtwocell[r]{}<>{^\varphi}&A
}}}
\quad&=\quad
\vcenter{\hbox{\xymatrix@!0@=19mm{
AMM\ar@/_3mm/[d]_-{1m}\xtwocell[rd]{}<>{^<1>a^2}\ar@/^3mm/[r]^-{a'1}\ar@/_3mm/[r]_-{a1}\xtwocell[r]{!<3mm,0mm>}<>{^\varphi 1\ }&AM\ar@/^3mm/[d]^-{a'}\ar@/_3mm/[d]_-{a}\xtwocell[d]{}<>{^\varphi}\\
AM\ar@/_3mm/[r]_-{a}&A
}}}
\\
\tag{OLA5}\label{ax:OLA5}\vcenter{
\hbox{\xymatrix@!0@=19mm{
AM\ar@/^3mm/[d]^-{a'}\ar@/_3mm/[d]_-{a}\xtwocell[d]{}<>{^\varphi}&A\ar[l]_-{1u}\ar[dl]^-{1}\xtwocell[dl]{}<>{^<3>\ \ a'^0}\\
A
}}}
\quad&=\quad
\vcenter{\hbox{\xymatrix@!0@=19mm{
AM\ar@/_3mm/[d]_-{a}&A\ar[l]_-{1u}\ar[dl]^-{1}\xtwocell[dl]{}<>{^<3>\ \ a^0}\\
A
}}}
\end{align}
\end{defi}

Oplax actions and their cells form a category $\OplaxAct(M;A)$, composition and identities in this category are calculated at the level of their underlying counterparts in $\mathcal{M}(AM,A)$, which means that there is a forgetful functor.
\[
\vcenter{\hbox{\xymatrix{
\OplaxAct(M;A)\ar[r]&\mathcal{M}(AM,A)
}}}
\]

\begin{rem}\label{rem:ComonadsAreIActions}
A glance at the axioms reveals that oplax $I$-actions on an object $A$ are nothing but comonads on $A$. Another point of view of this phenomenon is that comonads are oplax algebras of the identity pseudomonad, this is considered in \cite[Section 9]{Lack2014c}. In the classical case actions and representations always come hand in hand, and there is no exception with oplax actions. One may as well say that an \emph{oplax representation} of a skew monoidale $M$ with respect to an object $A$ is an opmonoidal arrow to the internal endo-hom monoidale of $A$.
\[
\vcenter{\hbox{\xymatrix{
M\ar[r]&[A,A]
}}}
\]
This means that we require the existence of an object $[A,A]$ in $\mathcal{M}$ with the universal property $\mathcal{M}(AX,A)\simeq\mathcal{M}(X,[A,A])$. One way it might exist is if $A$ has a right bidual, in which case we take the enveloping monoidale induced by the biduality. Another way is if the monoidal bicategory has a right closed monoidal structure as defined in \cite[Definition~5. and Example~2.]{Day1997}. In any case, right oplax actions and oplax representations are in correspondence by the usual means of transposition. A particular case of this situation was mentioned in Remark~\ref{rem:EndoHomMonoidale} where it is said that comonads on $R$ are oplax $R$-representations of the unit object $I$. We make all this very precise in the following theorem which we prove in full detail since with little effort its proof may be adapted to other results: see Corollary~\ref{cor:CompOpmonWithOplax}.
\end{rem}

\begin{teo}\label{teo:BidualityOpmonoidalOplaxAction}
For every right skew monoidale $M$ and every biduality $S\dashv S\ot$, there is an equivalence of categories given by transposition along the biduality.
\[
\SkOpmon(M,S\ot S)\simeq\OplaxAct(M;S)
\]
\end{teo}
\begin{proof}

The biduality $S\dashv S\ot$ induces the following equivalences of categories
\begin{align}
\label{eq:Opmon-OplaxAct1}
\mathcal{M}(M,S\ot S)&\simeq\mathcal{M}(SM,S)\\
\label{eq:Opmon-OplaxAct2}
\mathcal{M}(MM,S\ot S)&\simeq\mathcal{M}(SMM,S)\\
\label{eq:Opmon-OplaxAct3}
\mathcal{M}(I,S\ot S)&\simeq\mathcal{M}(S,S)
\end{align}
The data of an opmonoidal arrow consists of items in the left hand side of these equivalences: an object in \eqref{eq:Opmon-OplaxAct1}, an arrow in \eqref{eq:Opmon-OplaxAct2}, and an arrow in \eqref{eq:Opmon-OplaxAct3}. These ``opmonoidal data'' correspond under the equivalences above to the data for an oplax action:
\[
\begin{array}{c}
\xymatrix{
M\ar[r]^-{C}&S\ot S
}
\\
\hline
\xymatrix{
SM\ar[r]_-{s}&S}
\end{array}
\simeq\qquad
\begin{array}{c}
\vcenter{\hbox{\xymatrix@!0@=15mm{
MM\ar[r]^-{CC}\ar[d]_-{m}\xtwocell[rd]{}<>{^C^2}&S\ot SS\ot S\ar[d]^-{m}\\
M\ar[r]_-{C}&S\ot S
}}}
\\
\hline
\vcenter{\hbox{\xymatrix@!0@=15mm{
SMM\ar[r]^-{s1}\ar[d]_-{1m}\xtwocell[rd]{}<>{^s^2}&SM\ar[d]^-{s}\\
SM\ar[r]_-{s}&S
}}}
\end{array}
\cong\qquad
\begin{array}{c}
\vcenter{\hbox{\xymatrix@!0@R=15mm@C=10mm{
&I\ar[dl]_-{u}\ar[dr]^-{u}\xtwocell[rd]{}<>{^<3>C^0}&\\
M\ar[rr]_-{C}&&S\ot S
}}}
\\
\hline
\vcenter{\hbox{\xymatrix@!0@=15mm{
SM\ar[d]_-{s}&S\ar[l]_-{1u}\ar[dl]^-{1}\xtwocell[ld]{}<>{^<2>\ s^0}\\
S
}}}
\end{array}
\cong
\]
and the data for opmonoidal cells and oplax action cells is in a bijective correspondence.
\[
\begin{array}{c}
\xymatrix{
M\ar@/^3mm/[r]^-{C'}\ar@/_3mm/[r]_-{C}\xtwocell[r]{}<>{^\xi}&**[r]S\ot S
}
\\
\hline
\xymatrix{
**[l]SM\ar@/^3mm/[r]^-{s'}\ar@/_3mm/[r]_-{s}\xtwocell[r]{}<>{^\sigma}&S
}
\end{array}\cong
\]

If under this equivalence the property of being opmonoidal corresponds to the property of being an oplax action, then the theorem follows. We prove it by direct calculation: Let $\xymatrix@1@C=5mm{C:M\ar[r]&S\ot S}$ be an opmonoidal arrow, then the axioms \eqref{ax:OM1} and \eqref{ax:OLA1} are equations that lie in each of the sides of $\mathcal{M}(MMM,S\ot S)\simeq\mathcal{M}(SMMM,S)$, and the calculation below shows that \eqref{ax:OM1} for $C$ corresponds to \eqref{ax:OLA1} under the equivalence.
{\scriptsize
\[
\vcenter{\hbox{\xymatrix@!0@C=12mm@R=7mm{
&&SMM\ar[rd]^-{1C1}&&\\
&SS\ot SMM\ar[ru]^-{e111}\ar[rd]^-{111C1}&&SS\ot SM\ar[rd]^-{e11}&\\
SMMM\ar[ru]^-{1C11}\ar[dddd]|-{11m}\ar[rrdd]_-{1m1}\ar[rr]_-{1CC1}\xtwocell[rrdddddd]{}<>{^1\alpha\ }\xtwocell[rrrrd]{}<>{^<1>1C^21\quad\ }&&SS\ot SS\ot SM\ar[ru]^-{e1111}\ar[rd]^-{11e11}\ar@{}[uu]|*[@]{\cong}&&SM\ar[dd]^-{1C}\\
&\ar@{}[uu]|>>>>*[@]{\cong}&&SS\ot SM\ar[ru]^-{e11}\ar[dd]^-{111C}\ar@{}[uu]|*[@]{\cong}&\\
&&SMM\ar[dddd]^-{1m}\ar[ru]^-{1C1}\ar[rd]_-{1CC}\xtwocell[rdddd]{}<>{^<-1>1C^2}&&SS\ot S\ar[dd]^-{e1}\\
&&&SS\ot SS\ot S\ar[ru]_-{e111}\ar[dd]^-{11e1}\ar@{}[ruuu]|*[@]{\cong}&\\
SMM\ar[rrdd]_-{1m}&&\ar@{}[ruuu]|>>>>>>>*[@]{\cong}&&S\\
&&&SS\ot S\ar[ru]_-{e1}\ar@{}[ruuu]|*[@]{\cong}&\\
&&SM\ar[ru]_-{1C}&&
}}}
\!\!\!\stackrel{\eqref{ax:OM1}}{=}\!\!\!\!
\vcenter{\hbox{\xymatrix@!0@C=12mm@R=7mm{
&&SMM\ar[rd]^-{1C1}&&\\
&SS\ot SMM\ar[ru]^-{e111}\ar[rd]^-{111C1}\ar[dddd]_-{111m}\ar[rddd]|-{111CC}\xtwocell[rddddd]{}<>{^<1>111C^2}&&SS\ot SM\ar[rd]^-{e11}&\\
SMMM\ar[ru]^-{1C11}\ar[dddd]|-{11m}&&SS\ot SS\ot SM\ar[ru]^-{e1111}\ar[rd]^-{11e11}\ar[dd]|-{11111C}\ar@{}[uu]|*[@]{\cong}&&SM\ar[dd]^-{1C}\\
&&&SS\ot SM\ar[ru]^-{e11}\ar[dd]^-{111C}\ar@{}[uu]|*[@]{\cong}&\\
&&SS\ot SS\ot SS\ot S\ar[dd]|-{1111e1}\ar[rd]^-{11e111}\ar@{}[ru]|*[@]{\cong}&&SS\ot S\ar[dd]^-{e1}\\
&SS\ot SM\ar[rd]^-{111C}&&SS\ot SS\ot S\ar[ru]_-{e111}\ar[dd]^-{11e1}\ar@{}[ruuu]|*[@]{\cong}&\\
SMM\ar[rrdd]_-{1m}\ar[ru]^-{1C1}\ar[rr]_-{1CC}\ar@{}[ruuuuu]|*[@]{\cong}\xtwocell[rrrrd]{}<>{^<1>1C^2\quad}&&SS\ot SS\ot S\ar[rd]^-{11e1}\ar@{}[ru]|*[@]{\cong}&&S\\
&\ar@{}[uu]|>>>>*[@]{\cong}&&SS\ot S\ar[ru]_-{e1}\ar@{}[ruuu]|*[@]{\cong}&\\
&&SM\ar[ru]_-{1C}&&
}}}
\]
\[
=
\vcenter{\hbox{\xymatrix@!0@C=12mm@R=7mm{
&&SMM\ar[rd]^-{1C1}&&\\
&SS\ot SMM\ar[ru]^-{e111}\ar[rd]^-{111C1}\ar[dddd]_-{111m}\ar[rddd]|-{111CC}\xtwocell[rddddd]{}<>{^<1>111C^2}&&SS\ot SM\ar[rd]^-{e11}\ar[dd]^-{111C}&\\
SMMM\ar[ru]^-{1C11}\ar[dddd]|-{11m}&&SS\ot SS\ot SM\ar[ru]^-{e1111}\ar[dd]|-{11111C}\ar@{}[uu]|*[@]{\cong}\ar@{}[rd]|*[@]{\cong}&&SM\ar[dd]^-{1C}\\
&&&SS\ot SS\ot S\ar[rd]^-{e111}\ar@{}[ru]|*[@]{\cong}&\\
&&SS\ot SS\ot SS\ot S\ar[dd]|-{1111e1}\ar[rd]^-{11e111}\ar[ru]_-{e11111}&&SS\ot S\ar[dd]^-{e1}\\
&SS\ot SM\ar[rd]^-{111C}&&SS\ot SS\ot S\ar[ru]_-{e111}\ar[dd]^-{11e1}\ar@{}[uu]|*[@]{\cong}&\\
SMM\ar[rrdd]_-{1m}\ar[ru]^-{1C1}\ar[rr]_-{1CC}\ar@{}[ruuuuu]|*[@]{\cong}\xtwocell[rrrrd]{}<>{^<1>1C^2\quad}&&SS\ot SS\ot S\ar[rd]^-{11e1}\ar@{}[ru]|*[@]{\cong}&&S\\
&\ar@{}[uu]|>>>>*[@]{\cong}&&SS\ot S\ar[ru]_-{e1}\ar@{}[ruuu]|*[@]{\cong}&\\
&&SM\ar[ru]_-{1C}&&
}}}
\]
\[
=\!\!
\vcenter{\hbox{\xymatrix@!0@C=12mm@R=7mm{
&&SMM\ar[rd]^-{1C1}\ar[rddd]_-{1CC}&&\\
&SS\ot SMM\ar[ru]^-{e111}\ar[dddd]_-{111m}\ar[rddd]|-{111CC}\xtwocell[rddddd]{}<>{^<1>111C^2}&&SS\ot SM\ar[rd]^-{e11}\ar[dd]^-{111C}&\\
SMMM\ar[ru]^-{1C11}\ar[dddd]|-{11m}&&\ar@{}[ru]|>>*[@]{\cong}&&SM\ar[dd]^-{1C}\\
&&&SS\ot SS\ot S\ar[rd]^-{e111}\ar[dd]^-{11e1}\ar@{}[ru]|*[@]{\cong}&\\
&&SS\ot SS\ot SS\ot S\ar[dd]|-{1111e1}\ar[ru]_-{e11111}\ar@{}[rd]|*[@]{\cong}\ar@{}[uuuu]|*[@]{\cong}&&SS\ot S\ar[dd]^-{e1}\\
&SS\ot SM\ar[rd]^-{111C}&&SS\ot S\ar[rd]_-{e1}\ar@{}[ur]|*[@]{\cong}&\\
SMM\ar[rrdd]_-{1m}\ar[ru]^-{1C1}\ar[rr]_-{1CC}\ar@{}[ruuuuu]|*[@]{\cong}\xtwocell[rrrrd]{}<>{^<1>1C^2\quad}&&SS\ot SS\ot S\ar[rd]^-{11e1}\ar[ru]^-{e111}&&S\\
&\ar@{}[uu]|>>>>*[@]{\cong}&&SS\ot S\ar[ru]_-{e1}\ar@{}[uu]|*[@]{\cong}&\\
&&SM\ar[ru]_-{1C}&&
}}}
=\!\!
\vcenter{\hbox{\xymatrix@!0@C=12mm@R=7mm{
&&SMM\ar[rd]^-{1C1}\ar[rddd]_-{1CC}\ar[dddd]_-{1m}\xtwocell[rddddd]{}<>{^<1>1C^2}&&\\
&SS\ot SMM\ar[ru]^-{e111}\ar[dddd]_-{111m}&&SS\ot SM\ar[rd]^-{e11}\ar[dd]^-{111C}&\\
SMMM\ar[ru]^-{1C11}\ar[dddd]|-{11m}&&\ar@{}[ru]|>>*[@]{\cong}&&SM\ar[dd]^-{1C}\\
&&&SS\ot SS\ot S\ar[rd]^-{e111}\ar[dd]^-{11e1}\ar@{}[ru]|*[@]{\cong}&\\
&&SM\ar[rd]^-{1C}&&SS\ot S\ar[dd]^-{e1}\\
&SS\ot SM\ar[rd]^-{111C}\ar[ru]^-{e11}\ar@{}[ruuuuu]|*[@]{\cong}&&SS\ot S\ar[rd]_-{e1}\ar@{}[ur]|*[@]{\cong}&\\
SMM\ar[rrdd]_-{1m}\ar[ru]^-{1C1}\ar[rr]_-{1CC}\ar@{}[ruuuuu]|*[@]{\cong}\xtwocell[rrrrd]{}<>{^<1>1C^2\quad}&&SS\ot SS\ot S\ar[rd]^-{11e1}\ar[ru]^-{e111}\ar@{}[uu]|*[@]{\cong}&&S\\
&\ar@{}[uu]|>>>>*[@]{\cong}&&SS\ot S\ar[ru]_-{e1}\ar@{}[uu]|*[@]{\cong}&\\
&&SM\ar[ru]_-{1C}&&
}}}
\]
}
The other axioms \eqref{ax:OM2}, \eqref{ax:OM3}, \eqref{ax:OLA2}, and \eqref{ax:OLA3} lie in each of the sides of \eqref{eq:Opmon-OplaxAct1}. The calculation below shows that \eqref{ax:OM2} for $C$ corresponds to \eqref{ax:OLA2} under the equivalence \eqref{eq:Opmon-OplaxAct1};
{\scriptsize
\[
\qquad\vcenter{\hbox{\xymatrix@!0@C=12mm@R=7mm{
&&S\ar@/^9mm/[ddddddrr]^-{1}&&\\
&SS\ot S\ar[ru]^-{e1}\ar@/^9mm/[ddddddrr]^-{1}&&&\\
SM\ar[ru]^-{1C}\ar[dddd]_-{11u}\xtwocell[rrdddddd]{}<\omit>{^1\rho}\ar@/^9mm/[ddddddrr]^-{1}&&&&\\
&&&&\\
&&&&\\
&&&&\\
SMM\ar[rrdd]_-{1m}&&&&S\\
&&&SS\ot S\ar[ru]_-{e1}&\\
&&SM\ar[ru]_-{1C}&&
}}}
\!\!\!\stackrel{\eqref{ax:OM2}}{=}\!\!\!
\vcenter{\hbox{\xymatrix@!0@C=12mm@R=7mm{
&&S\ar@/^9mm/[ddddddrr]^-{1}&&\\
&SS\ot S\ar[ru]^-{e1}\ar[dddd]_-{111u}\ar@/^3mm/[rddddd]^-{111n}\ar@/^9mm/[ddddddrr]^-{1}\xtwocell[rddddd]{}<>{^111C^0\quad}&&&\\
SM\ar[ru]^-{1C}\ar[dddd]_-{11u}&&&&\\
&&&&\\
&&&&\\
&SS\ot SM\ar[rd]^-{111C}&&&\\
SMM\ar[rrdd]_-{1m}\ar[ru]^-{1C1}\ar[rr]_-{1CC}\ar@{}[ruuuuu]|*[@]{\cong}\xtwocell[rrrrd]{}<>{^<1>1C^2\quad}&\ar@{}[u]|*=0[@]{\cong}&SS\ot SS\ot S\ar[rd]^-{11e1}&&S\\
&&&SS\ot S\ar[ru]_-{e1}&\\
&&SM\ar[ru]_-{1C}&&
}}}
\]
\[
=
\vcenter{\hbox{\xymatrix@!0@C=12mm@R=7mm{
&&S\ar@/^9mm/[ddddddrr]^-{1}\ar@/^3mm/[rddddd]^-{1n}\ar[dddd]_-{1u}\xtwocell[rddddd]{}<>{^1C^0\quad}&&\\
&SS\ot S\ar[ru]^-{e1}\ar[dddd]_-{111u}&&&\\
SM\ar[ru]^-{1C}\ar[dddd]_-{11u}&&&&\\
&&&&\\
&&SM\ar[rd]^-{1C}&&\\
&SS\ot SM\ar[rd]^-{111C}\ar[ru]_-{e11}\ar@{}[ruuuuu]|*[@]{\cong}&&SS\ot S\ar[rd]^-{e1}\ar@{}[ruuu]|*[@]{\cong}&\\
SMM\ar[rrdd]_-{1m}\ar[ru]^-{1C1}\ar[rr]_-{1CC}\ar@{}[ruuuuu]|*[@]{\cong}\xtwocell[rrrrd]{}<>{^<1>1C^2\quad}&\ar@{}[u]|*=0[@]{\cong}&SS\ot SS\ot S\ar[rd]^-{11e1}\ar[ru]_-{e111}\ar@{}[uu]|*[@]{\cong}&&S\\
&&&SS\ot S\ar[ru]_-{e1}\ar@{}[uu]|*[@]{\cong}&\\
&&SM\ar[ru]_-{1C}&&
}}}
\]
}
and the calculation below shows that \eqref{ax:OM3} for $C$ corresponds to \eqref{ax:OLA3} under the equivalence \eqref{eq:Opmon-OplaxAct1}.
{\scriptsize
\[
\vcenter{\hbox{\xymatrix@!0@C=12mm@R=7mm{
&&&&\\
&&&&\\
SM\ar@/^9mm/[rrrr]^-{1}\ar@/^5mm/[rrrd]^-{1n1}\ar[rrdd]_-{1u1}\ar@/_9mm/[ddddddrr]_-{1}\xtwocell[rrdddddd]{}<>{^1\lambda}\xtwocell[rrrd]{}<>{^1C^01\quad }&&&&SM\ar[dd]^-{1C}\\
&&&SS\ot SM\ar[ru]^-{e11}\ar[dd]^-{111C}&\\
&&SMM\ar[dddd]_-{1m}\ar[ru]^-{1C1}\ar[rd]_-{1CC}\xtwocell[rdddd]{}<>{^<-1>1C^2}&&SS\ot S\ar[dd]^-{e1}\\
&&&SS\ot SS\ot S\ar[ru]_-{e111}\ar[dd]^-{11e1}\ar@{}[ruuu]|*[@]{\cong}&\\
&&\ar@{}[ruuu]|>>>>>>>*[@]{\cong}&&S\\
&&&SS\ot S\ar[ru]_-{e1}\ar@{}[ruuu]|*[@]{\cong}&\\
&&SM\ar[ru]_-{1C}&&
}}}
\stackrel{\eqref{ax:OM3}}{=}
\vcenter{\hbox{\xymatrix@!0@C=12mm@R=7mm{
&&&&\\
&&&&\\
SM\ar@/^9mm/[rrrr]^-{1}\ar@/^5mm/[rrrd]^-{1n1}\ar[rrdd]_-{1C}\ar@/_9mm/[ddddddrr]_-{1}&&&&SM\ar[dd]^-{1C}\\
&&&SS\ot SM\ar[ru]^-{e11}\ar[dd]^-{111C}\ar@{}[uu]|*[@]{\cong}&\\
&\ar@{}[rruu]|*=0[@]{\cong}&SS\ot S\ar[rd]^-{1n11}\ar@/_7mm/[rddd]_-{1}&&SS\ot S\ar[dd]^-{e1}\\
&&&SS\ot SS\ot S\ar[ru]_-{e111}\ar[dd]^-{11e1}\ar@{}[ruuu]|*[@]{\cong}&\\
&&\ar@{}[ru]|*[@]{\cong}&&S\\
&&&SS\ot S\ar[ru]_-{e1}\ar@{}[ruuu]|*[@]{\cong}&\\
&&SM\ar[ru]_-{1C}&&
}}}
\]
\[
=
\vcenter{\hbox{\xymatrix@!0@C=12mm@R=7mm{
&&&&\\
&&&&\\
SM\ar@/^9mm/[rrrr]^-{1}\ar[rrdd]_-{1C}\ar@/_9mm/[ddddddrr]_-{1}&&&&SM\ar[dd]^-{1C}\\
&&&&\\
&&SS\ot S\ar[rd]^-{1n11}\ar@/_7mm/[rddd]_-{1}\ar@/^7mm/[rr]^-{1}&&SS\ot S\ar[dd]^-{e1}\\
&&&SS\ot SS\ot S\ar[ru]_-{e111}\ar[dd]^-{11e1}\ar@{}[uu]|*[@]{\cong}&\\
&&\ar@{}[ru]|*[@]{\cong}&&S\\
&&&SS\ot S\ar[ru]_-{e1}\ar@{}[ruuu]|*[@]{\cong}&\\
&&SM\ar[ru]_-{1C}&&
}}}
=
\vcenter{\hbox{\xymatrix@!0@C=12mm@R=7mm{
&&&&\\
&&&&\\
SM\ar@/^9mm/[ddrrrr]^-{1C}\ar@/_9mm/[dddddrrr]_-{1C}&&&&\\
&&&&\\
&&&&SS\ot S\ar[dd]^-{e1}\\
&&&&\\
&&&&S\\
&&&SS\ot S\ar[ru]_-{e1}&\\
&&&&
}}}
\]
}
Therefore $\xymatrix@1@C=5mm{SM\ar[r]^-{1C}&SS\ot S\ar[r]^-{e1}&S}$ is an oplax $S$-action. Now let $\xymatrix@1@C=5mm{\xi:C\ar[r]&C'}$ be an opmonoidal cell, then \eqref{ax:OLA4} and \eqref{ax:OLA5} hold by the calculations below.
\[
\vcenter{\hbox{\xymatrix@!0@R=18mm@C=20mm{
SMM\ar@/^3mm/[r]^-{1C'1}\ar[dd]_-{1m}\ar[rd]|-{1C'C'\quad}\xtwocell[rdd]{!<0mm,-6mm>}<>{^1C^2\ \ }&SS\ot SM\ar[r]^-{e1}\ar[d]|-{111C'}&SM\ar@/^3mm/[d]^-{1C'}\\
\ar@{}[ru]|>>>>>>*=0[@]{\cong}&SS\ot SS\ot S\ar[r]^-{e111}\ar[d]|-{11e1}\ar@{}[ru]|*=0[@]{\cong}&SS\ot S\ar[d]^-{e1}\\
SM\ar@/^3mm/[r]^-{1C'}\ar@/_3mm/[r]_-{1C}\xtwocell[r]{!<2mm,0mm>}<>{^1\xi\ }&SS\ot S\ar[r]_-{e1}\ar@{}[ru]|*=0[@]{\cong}&S
}}}
\stackrel{\eqref{ax:OM4}}{=}\;\;
\vcenter{\hbox{\xymatrix@!0@R=18mm@C=20mm{
SMM\ar@/^3mm/[r]^-{1C'1}\ar[dd]_-{1m}\ar@/_3mm/[rd]_>>>>{1CC}\ar@/^3mm/[rd]^<<<<<<{1C'C'}\xtwocell[rd]{}<>{^1\xi\xi\ \ }&SS\ot SM\ar[r]^-{e1}\ar[d]|-{\quad 111C'}&SM\ar@/^3mm/[d]^-{1C'}\\
{\xtwocell[rd]{!<0mm,4mm>}<>{^1C^2\ \ }}&SS\ot SS\ot S\ar[r]^-{e111}\ar[d]|-{11e1}\ar@{}[ru]|*=0[@]{\cong}&SS\ot S\ar[d]^-{e1}\\
SM\ar@/_3mm/[r]_-{1C}&SS\ot S\ar[r]_-{e1}\ar@{}[ru]|*=0[@]{\cong}&S
}}}
\]
\[
=\quad
\vcenter{\hbox{\xymatrix@!0@R=18mm@C=20mm{
SMM\ar[dd]_-{1m}\ar[rd]_-{1CC}\ar@/^3mm/[r]^-{1C'1}\ar@/_3mm/[r]_-{1C1}\xtwocell[r]{!<4mm,0mm>}<>{^1\xi 1\ \ }&SS\ot SM\ar[r]^-{e1}\ar[d]|-{\ \ 111C}&SM\ar@/_3mm/[d]_-{1C'}\ar@/^3mm/[d]^-{1C'}\xtwocell[d]{}<>{^1\xi}\\
{\xtwocell[rd]{!<0mm,4mm>}<>{^1C^2\ \ }}&SS\ot SS\ot S\ar[r]^-{e111}\ar[d]|-{11e1}\ar@{}[ru]!<-4mm,0mm>|*=0[@ru]{\cong}&SS\ot S\ar[d]^-{e1}\\
SM\ar@/_3mm/[r]_-{1C}\ar@{}[ruu]|>>>>>>>*=0[@ru]{\cong}&SS\ot S\ar[r]_-{e1}\ar@{}[ru]|*=0[@]{\cong}&S
}}}
\]
\[
\vcenter{\hbox{\xymatrix@!0@R=13mm@C=26mm{
SM\ar@/_3mm/[d]_-{C'}\ar@/^3mm/[d]^-{C'}\xtwocell[d]{}<>{^\xi}\ar@{}[ddr]|<<<<<<<<<<<<<*[@]{\cong}&S\ar[l]_-{1u}\ar[ld]|-{1n}\ar[ldd]^-{1}\xtwocell[ld]{}<>{^<2>\ \ \ C'^0}\\
SS\ot S\ar[d]_-{e1}&\\
S&
}}}
\quad=\quad
\vcenter{\hbox{\xymatrix@!0@R=13mm@C=26mm{
SM\ar@/_3mm/[d]_-{C}\ar@{}[ddr]|<<<<<<<<<<<<<*[@]{\cong}&S\ar[l]_-{1u}\ar[ld]|-{1n}\ar[ddl]^-{1}\xtwocell[ddl]{}<>{^<6>\ \ C^0}\\
SS\ot S\ar[d]_-{e1}&\\
S&
}}}
\]
Therefore the cell $\xymatrix@1@=10mm{**[l]SM\ar@/^3mm/[r]^-{1C'}\ar@/_3mm/[r]_-{1C}\xtwocell[r]{}{^1\xi\ }&**[r]SS\ot S\ar[r]^-{e1}&S}$ is a cell of oplax right $S$-actions.
\end{proof}

\begin{cor}\label{cor:CompOpmonWithOplax}
For every object $A$ in a monoidal bicategory $\mathcal{M}$ there is a pseudofunctor
\[
\vcenter{\hbox{\xymatrix{
\OplaxAct(\_;A):\SkOpmon^{\mathrm{op}}(\mathcal{M})\ar[r]&\Cat.
}}}
\]
\end{cor}
\begin{proof}

For objects $S$ which are part of a biduality $S\dashv S\ot$ the previous theorem provides the pseudonatural equivalence given by,
\[
\SkOpmon(\_,S\ot S)\simeq\OplaxAct(\_;S)
\]
In general, for an object $A$ in $\mathcal{M}$, if $\xymatrix@1@C=5mm{C:M\ar[r]&N}$ is an opmonoidal arrow and $\xymatrix@1@C=5mm{a:AN\ar[r]&A}$ an oplax right $N$-action, then the proof that the composite
\[
\vcenter{\hbox{\xymatrix{
AM\ar[r]^-{1C}&AN\ar[r]^-{a}&A
}}}
\]
is an oplax right $M$-action is analogous to the big diagram calculation of Theorem~\ref{teo:BidualityOpmonoidalOplaxAction} but replacing each instance of $\xymatrix@1@C=5mm{e1:SS\ot S\ar[r]&S}$ with $a$ where appropriate. The same argument goes for morphisms of opmonoidal arrows and cells of oplax actions.
\end{proof}

\begin{rem}\label{rem:InducedComonad}
Oplax representations with respect to an object $S$ induce comonads on $S$. Indeed, let $\xymatrix@1@C=5mm{C:M\ar[r]&S\ot S}$ be an oplax representation of a right skew monoidale $M$; precomposition of $C$ with the unit of $M$ is an opmonoidal arrow
\[
\vcenter{\hbox{\xymatrix{
I\ar[r]^-{u}&M\ar[r]^-{C}&S\ot S
}}}
\]
since the unit $\xymatrix@1@C=5mm{u:I\ar[r]&M}$ is an opmonoidal arrow by Lemma~\ref{lem:OpmonoidalUnit}. Then the transposition along the enveloping monoidale $S\ot S$ as in Theorem~\ref{teo:BidualityOpmonoidalOplaxAction} is an oplax $I$-action, in other words a comonad on $S$, see Remark~\ref{rem:ComonadsAreIActions}.

But we may get another perspective on this as a simple application of the previous corollary: any oplax right action $\xymatrix@1@C=5mm{a:AM\ar[r]&A}$ induces a comonad on $A$ in exactly the same way. This time we may precompose $a$ with the unit of $M$ using Corollary~\ref{cor:CompOpmonWithOplax} to get an oplax right $I$-action on $A$.
\[
\xymatrix{AI\ar[r]^-{1u}&AM\ar[r]^-{a}&A}
\]
Again, this is a comonad and it has comultiplication and counit as shown below.
\[
\vcenter{\hbox{\xymatrix@!0@R=10mm@C=10mm{
&&A\ar[rd]^-{1u}&&\\
&AM\ar[rd]|-{11u}\ar[ru]^-{a}&&AM\ar[rd]^-{a}&\\
A\ar[rd]_-{1u}\ar[ru]^-{1u}&&AMM\ar[rd]|-{1m}\ar[ru]|-{a1}\ar@{}[uu]|*[@]{\cong}\xtwocell[rr]{!<3mm,0mm>}<>{^a^2\ }&&A\\
&AM\ar[ru]|-{1u1}\ar@/_7mm/[rr]_-{1}\ar@{}[uu]|*[@]{\cong}\xtwocell[rr]{}<>{^1\lambda\ }&&AM\ar[ru]_-{a}&\\
}}}
\qquad
\vcenter{\hbox{\xymatrix{
A\ar[r]_-{1u}\ar@/^10mm/[rr]^-{1}\xtwocell[rr]{}<>{^<-3>a^0\ }&AM\ar[r]_-{a}&A
}}}
\]
A nice case to consider is the oplax right $R$-action $\xymatrix@1@C=5mm{i^*1:RR\ar[r]&R}$ induced by an adjunction $i\dashv i^*$ in $\mathcal{M}$,
\[
\vcenter{\hbox{\xymatrix{
R\dtwocell_{i}^{i^*}{'\dashv}\\
I
}}}
\]
then the process above recovers the comonad on $R$ associated to the adjunction.

\end{rem}

The next theorem asserts that the functor
\[
\vcenter{\hbox{\xymatrix{
\OplaxAct(i\ob 1,A):\OplaxAct(R\ot R;A)\ar[r]&\OplaxAct(R,A)
}}}
\]
is an equivalence of categories, assuming that the opmonoidal left adjoint $i\ob 1$ in Proposition~\ref{prop:OneOpmonoidalArrow} is opmonadic and that we are in an opmonadic-friendly monoidal bicategory. Its proof is entirely analogous to the one of Theorem~\ref{teo:OpmonadicityOpmonoidal} so we present only a sketch. Its statement may also be informally interpreted as an ``opmonoidal $\dashv$ monoidal opmonadicity'', again in analogy with Theorem~\ref{teo:OpmonadicityOpmonoidal}, but instead of taking the hom functor $\Opmon(\_,N)$, it takes the functor $\OplaxAct(\_;A)$ of Corollary~\ref{cor:CompOpmonWithOplax}. In the case that $A$ has a right bidual it is possible to make this analogy into a formal statement; it takes the shape of the commutative square of equivalences in Corollary~\ref{cor:Opmon_is_OplaxAct_Local}.

\begin{teo}\label{teo:OpmonadicityOplaxActions}
Let $\mathcal{M}$ be an opmonadic-friendly monoidal bicategory. For every object $A$, every biduality $R\dashv R\ot$, and every opmonadic adjunction $i\ob\dashv i\ot$ in $\mathcal{M}$,
\[
\vcenter{\hbox{\xymatrix{
R\ot\dtwocell_{i\ob}^{i\ot}{'\dashv}\\
I
}}}
\]
there is an equivalence of categories given by precomposition along $\xymatrix@1@C=5mm{1i_o1:AR\ar[r]&AR\ot R}$.
\[
\OplaxAct(R\ot R;A)\simeq\OplaxAct(R;A)
\]
\end{teo}
\begin{proof}[Sketch]

The strategy is to consider the category $\mathcal{Y}(R;A)$ of oplax right actions $\xymatrix@1@C=5mm{a:AR\ar[r]&A}$ together with an action $\psi$ (two axioms) for the monad induced by $1i\ob 1\dashv 1i\ot 1$ which is compatible in the appropriate way (three axioms: (Y1), (Y2) and (Y3)) with the oplax action constraints $a^2$ and $a^0$,
\[
\vcenter{\hbox{\xymatrix@!0@C=16mm{
AR\ar[r]_-{1i\ob 1}\ar@/^9mm/[rrr]^-{a}&AR\ot R\ar[r]_-{1i\ot 1}\xtwocell[r]{}<>{^<-3>\psi}&AR\ar[r]_-{a}&A
}}}
\]
and then prove the existence of an equivalence and an isomorphism as shown.
\[
\OplaxAct(R\ot R;A)\simeq\mathcal{Y}(R;A)\cong\OplaxAct(R;A)
\]

For an object $(a,\psi)$ in $\mathcal{Y}(R;A)$ the action $\psi$ is redundant as it may be written in terms of the oplax action constraints $a^0$ and $a^2$,
\[
\vcenter{\hbox{\xymatrix@!0@C=16mm{
AR\ar[r]_-{1i\ob 1}\ar@/^9mm/[rrr]^-{a}&AR\ot R\ar[r]_-{1i\ot 1}\xtwocell[r]{}<>{^<-3>\psi}&AR\ar[r]_-{a}&A
}}}
\quad=\quad
\vcenter{\hbox{\xymatrix@!0@R=10mm@C=8mm{
&&&&AR\ar[rrdd]^-{a}&&\\
&&ARR\ar[rrd]|-{1i^*1}\ar[rru]^-{a1}&{\xtwocell[rrd]{}<>{^<-2>a^2\ }}&&\\
AR\ar[rr]_-{1i\ob 1}\ar@/^12mm/[rrrruu]^-{1}\ar[rru]^-{1i1}\xtwocell[rrrruu]{}<>{^<-4>a^01\quad}&&AR\ot R\ar[rr]_-{1i\ot 1}\ar@{}[u]|*[@]{\cong}&&AR\ar[rr]_-{a}&&A
}}}
\]
and for an arbitrary oplax action $a$ the cell $\psi$ written in terms of $a^0$ and $a^2$ as above provides $a$ with the structure of an object in $\mathcal{Y}(R;A)$, hence the functor that forgets this structure is an isomorphism $\mathcal{Y}(R;A)\cong\OplaxAct(R;A)$.

Yet with the five axioms that hold for the objects $(a,\psi)$ of $\mathcal{Y}(R;A)$ and the opmonadicity of $i\ob\dashv i\ot$, we get the data for an oplax $R\ot R$-action on $A$: The first two axioms say that $a$ is a module for the monad induced by $1i\ob 1\dashv 1i\ot 1$, guaranteeing the existence of an arrow $\xymatrix@1@C=5mm{AR\ot R\ar[r]&A}$; axiom~(Y1) confirms the existence of the unitor cell; and axioms~(Y2) and (Y3) ensure the existence of the associator cell. Finally, this induced data constitute an oplax $R\ot R$-action with the property that precomposing with $\xymatrix@1@C=5mm{1i\ob 1:AR\ar[r]&AR\ot R}$ gives back the original oplax $R$-action one started with, up to isomorphism. This gives the behaviour on objects of an equivalence of categories $\OplaxAct(R\ot R;A)\simeq\mathcal{Y}(R;A)$.
\end{proof}

\begin{cor}\label{cor:Opmon_is_OplaxAct_Local}
Let $\mathcal{M}$ be an opmonadic-friendly monoidal bicategory. For every two bidualities $R\dashv R\ot$ and $S\dashv S\ot$, and every opmonadic adjunction $i\ob\dashv i\ot$ in $\mathcal{M}$,
\[
\vcenter{\hbox{\xymatrix{R\ot\dtwocell_{i\ob}^{i\ot}{'\dashv}\\
I
}}}
\]
there is an equivalence of categories as shown,
\[
\Opmon(R\ot R,S\ot S)\simeq\OplaxAct(R;S)
\]
where $R$ has the skew monoidal structure induced by the adjunction $i\dashv i^*$ opposite to $i\ob\dashv i\ot$ as in Lemma~\ref{lem:OneRightSkewMonoidale}. Moreover, the following square of equivalences commutes up to isomorphism,
\begin{equation}
\label{eq:OpmonOplaxactSquare}
\vcenter{\hbox{\xymatrix@!0@C=60mm@R=18mm{
\Opmon(R\ot R,S\ot S)\ar[r]_-{\simeq}^{\SkOpmon(i\ob 1,S\ot S)}\ar[d]^-{\simeq}&\SkOpmon(R,S\ot S)\ar[d]_-{\simeq}\\
\OplaxAct(R\ot R;S)\ar[r]_-{\OplaxAct(i\ob 1;S)}^-{\simeq}\ar@{}[ru]|*[@]{\cong}&\OplaxAct(R;S)
}}}
\end{equation}
where the vertical functors in the square are given by transposition along $S\dashv S\ot$ as in Theorem~\ref{teo:BidualityOpmonoidalOplaxAction}, the functor on the top is an instance of the equivalence in Theorem~\ref{teo:OpmonadicityOpmonoidal}, and the functor on the bottom is an instance of the equivalence in Theorem~\ref{teo:OpmonadicityOplaxActions}.
\end{cor}
\begin{proof}

The equivalence in the statement follows from either of the two composites in the square \eqref{eq:OpmonOplaxactSquare}. The commutativity of the square follows strictly in the case that $\mathcal{M}$ is a strict monoidal 2-category, because an opmonoidal arrow $C$ in $\Opmon(R\ot R,S\ot S)$ gets sent to the unambiguous composite below,
\[
\vcenter{\hbox{\xymatrix{SR\ar[r]^-{1i\ob 1}&SR\ot R\ar[r]^-{1C}&SS\ot S\ar[r]^-{e1}&S}}}
\]
which, in the case of an arbitrary monoidal bicategory $\mathcal{M}$, depends on how the parenthesis are placed. The two different ways to do it corresponding to the top path and the bottom path in the square of the statement differ by a coherent isomorphism which consists of instances of the associativity of the composition and instances of the interchanger between the composition and the tensor.
\end{proof}

\begin{rem}\label{rem:SimpleArrows}
In an opmonadic-friendly monoidal bicategory $\mathcal{M}$, for a duality $R\dashv R\ot$ and an opmonadic adjunction $i\ob\dashv i\ot$ with opposite adjunction $i\dashv i^*$,
\[
\vcenter{\hbox{\xymatrix{
R\dtwocell_{i\ob}^{i\ot}{'\dashv}\\
I
}}}
\]
one may take the identity opmonoidal arrow on $R\ot R$ through all the equivalences in the square~\eqref{eq:OpmonOplaxactSquare} which gives the following interesting items.
\[
\begin{array}{c@{}c@{}c}
\vcenter{\hbox{\xymatrix{
(R\ot R\ar[r]^-{1}&R\ot R)
}}}
&
\vcenter{\hbox{\xymatrix@C=5mm{
\ar@{|->}[r]&
}}}&
\vcenter{\hbox{\xymatrix{
(R\ar[r]^-{i\ob 1}&R\ot R)
}}}
\\
\vcenter{\hbox{\xymatrix@R=5mm{
\ar@{|->}[d]\\
\\
}}}
&&
\vcenter{\hbox{\xymatrix@R=5mm{
\ar@{|->}[d]\\
\\
}}}
\\
\vcenter{\hbox{\xymatrix{
(RR\ot R\ar[r]_-{e1}&R)
}}}
&
\vcenter{\hbox{\xymatrix@C=5mm{
\ar@{|->}[r]&
}}}
&
\vcenter{\hbox{\xymatrix{
(RR\ar[r]_-{i^*1}&R)
}}}
\end{array}
\]
On the top-right, $i\ob 1$ has the opmonoidal structure defined in Proposition~\ref{prop:OneOpmonoidalArrow}. On the bottom-left, $e1$ is the evaluation arrow of the internal hom $R\ot R$ and is canonically a pseudo right $R\ot R$-action structure on $R$. And on the bottom-right, $i^*1$ has the regular oplax $R$-action structure of the skew monoidale structure on $R$ (as in Remark~\ref{rem:SkewInducesAction}), which is induced by the adjunction $i\dashv i^*$ as in Lemma~\ref{lem:OneRightSkewMonoidale}.

\end{rem}

We close this section by going back to the example of $R|S$-coalgebroids given at the end of Section~\ref{sec:SkewMonoidalesBidualitiesAdjunctions}, which by means of Corollary~\ref{cor:Opmon_is_OplaxAct_Local} may now be described with less effort.

\begin{eje}\label{exa:CoalgebroidsAsOplaxActions}
Let $R$ and $S$ be $k$-algebras for a commutative ring $k$. In Lemma~\ref{lem:CoalgebroidsAsQuantum} we showed that $R|S$-coalgebroids are equivalent to opmonoidal arrows between enveloping monoidales in $\Mod_k$. We know that $\Mod_k$ is an opmonadic-friendly autonomous monoidal bicategory and every adjunction is monadic and opmonadic. Moreover, the unit of $R$ which might be seen as a ring morphism from $k$ to $R$ induces an adjunction $i\dashv i^*$ and its dual $i\ob\dashv i\ot$. Hence, we may use the equivalence in Corollary~\ref{cor:Opmon_is_OplaxAct_Local} to express an $R|S$-coalgebroid via oplax right $R$-actions on $S$ in $\Mod_k$ where $R$ has the right skew monoidal structure induced by the unit $\xymatrix@1@C=5mm{i:k\ar[r]&R}$. This definition involves considerably less information than the one in \ref{def:coalgebroid}.

An \emph{$R|S$-coalgebroid via oplax right $R$-actions} is an $S$-coring $(C,\varepsilon,\delta)$ together with a left $R$-module structure on $C$ compatible with both of its $S$-module structures and such that $\delta(rc)=\sum c_{(1)}\otimes rc_{(2)}$. More explicitly, one has a module $C$ in $SR$-$\Mod$-$S$, a morphism $\xymatrix@1@C=5mm{\varepsilon:C\ar[r]&S}$ in $S$-$\Mod$-$S$, and a morphism $\xymatrix@1@C=5mm{\delta:C\ar[r]&C\otimes_S C}$ in $SR$-$\Mod$-$S$ where the left $R$-module structure of $C\otimes_SC$ is given by $r.(c\otimes c')=c\otimes rc'$. And together, these constitute a comonoid $(C,\varepsilon,\delta)$ in the monoidal category $S$-$\Mod$-$S$.
\end{eje}
\section{Comodules for Opmonoidal Arrows}\label{sec:Comodules}

In this last section we define comodules with respect to an opmonoidal arrow. We saw in Lemma~\ref{lem:CoalgebroidsAsQuantum} that $R|S$-coalgebroids are opmonoidal arrows between enveloping monoidales in the bicategory $\Mod_k$. Comodules for $R|S$-coalgebroids are classically defined as the comodules with respect to their underlying comonoid in $S$-$\Mod$-$S$. There is no problem in expressing this definition purely in terms of a monoidal bicategory $\mathcal{M}$. Now, by using the same techniques as in the two previous sections, we show in Corollary~\ref{cor:ComodulesAreComodules} that both definitions of a comodule coincide modulo an equivalence of categories. Moreover, we exhibit a monoidal structure on the category of comodules for the underlying opmonoidal arrow of an opmonoidal monad, and this monoidal structure is such that the forgetful functor down to the underlying arrows of the comodules is strong monoidal. This generalises the classical case for $R|S$-coalgebroids in \cite{Hai2008}.
\begin{defi}
Let $M$ and $N$ be two right skew monoidales, $A$ and $B$ two objects, and $\xymatrix@1@C=5mm{a:AM\ar[r]&A}$ and $\xymatrix@1@C=5mm{b:BN\ar[r]&B}$ two oplax right actions in $\mathcal{M}$. A \emph{right comodule} $(Y,C,y)$ from $a$ to $b$ consists of an arrow $\xymatrix@1@C=5mm{Y:A\ar[r]&B}$, an opmonoidal arrow $\xymatrix@1@C=5mm{C:M\ar[r]&N}$, and a cell $y$ in $\mathcal{M}$
\[
\vcenter{\hbox{\xymatrix@!0@=15mm{
AM\ar[r]^-{YC}\ar[d]_-{a}\xtwocell[rd]{}<>{^y}&BN\ar[d]^-{b}\\
A\ar[r]_-{Y}&B
}}}
\]
called the $C$-coaction, satisfying the coassociative and counit laws.
\begin{align}
\tag{COM1}\label{ax:COM1}
\vcenter{\hbox{\xymatrix@!0@C=15mm{
&BNN\ar[rd]^-{b1}&\\
AMM\ar[ru]^-{YCC}\ar[dd]_-{1m}\ar[rd]_-{a1}\xtwocell[rddd]{}<>{^a^2}\xtwocell[rr]{}<>{^yC\quad}&&BN\ar[dd]^-{b}\\
&AM\ar[dd]^-{a}\ar[ru]^-{YC}\xtwocell[rd]{}<>{^y}&\\
AM\ar[rd]_-{a}&&B\\
&A\ar[ru]_-{Y}&
}}}
\quad&=\quad
\vcenter{\hbox{\xymatrix@!0@C=15mm{
&BNN\ar[rd]^-{b1}\ar[dd]_-{1m}\xtwocell[rddd]{}<>{^b^2}&\\
AMM\ar[ru]^-{YCC}\ar[dd]_-{1m}\xtwocell[rd]{}<>{^YC^2\quad}&&BN\ar[dd]^-{b}\\
&BN\ar[rd]^-{b}&\\
AM\ar[rd]_-{a}\ar[ru]_-{YC}\xtwocell[rr]{}<>{^y}&&B\\
&A\ar[ru]_-{Y}&
}}}
\\
\tag{COM2}\label{ax:COM2}
\vcenter{\hbox{\xymatrix@!0@C=15mm{
&B\ar@/^7mm/[dddr]^-{1}&\\
A\ar[ru]^-{Y}\ar[dd]_-{1u}\ar@/^7mm/[dddr]^-{1}\xtwocell[rddd]{}<>{^a^0}&&\\
&&\\
AM\ar[rd]_-{a}&&B\\
&A\ar[ru]_-{Y}&
}}}
\quad&=\quad
\vcenter{\hbox{\xymatrix@!0@C=15mm{
&B\ar[dd]_-{1u}\ar@/^7mm/[dddr]^-{1}\xtwocell[rddd]{}<>{^b^0}&\\
A\ar[ru]^-{Y}\ar[dd]_-{1u}\xtwocell[rd]{}<>{^YC^0\quad}&&\\
&BN\ar[rd]^-{b}&\\
AM\ar[rd]_-{a}\ar[ru]_-{YC}\xtwocell[rr]{}<>{^y}&&B\\
&A\ar[ru]_-{Y}&
}}}
\end{align}
\end{defi}

\begin{rem}
For a fixed opmonoidal arrow $\xymatrix@1@C=5mm{C:M\ar[r]&N}$ right comodules $(Y,C,y)$ from $a$ to $b$ are also called \emph{right $C$-comodules from $a$ to $b$} or \emph{right comodules over $C$ from $a$ to $b$}, and shall be denoted by $(Y,y)$. Note that the opmonoidal arrow $C$ plays a similar role as a comonoid in the definition of comodules for comonoids. One may similarly define \emph{right modules over} monoidal arrows (instead of opmonoidal ones) between right skew monoidales by changing the direction of $y$ and by modifying the axioms accordingly.
\end{rem}

\begin{defi}
A morphism $\xymatrix@1@C=5mm{(\gamma,\xi):(Y,C,y)\ar[r]&(Y',C',y')}$ of right comodules from $a$ to $b$ consists of a cell $\gamma$ and an opmonoidal cell $\xi$ in $\mathcal{M}$,
\[
\vcenter{\hbox{\xymatrix{
A\xtwocell[r]{}^{Y'}_Y{^\gamma}&B
}}}
\qquad
\vcenter{\hbox{\xymatrix{
M\xtwocell[r]{}^{C'}_C{^\xi}&N
}}}
\]
satisfying the following equation.
\begin{equation}
\tag{COM3}\label{ax:COM3}
\vcenter{\hbox{\xymatrix@!0@R=15mm@C=19mm{
AM\ar@/^3mm/[r]^-{Y'C'}\ar[d]_-a\xtwocell[rd]{!<-1mm,2mm>}<>{^<-1>y'}&BN\ar[d]^-b\\
A\rtwocell^{Y'}_Y{^\gamma}&B
}}}
\quad=\quad
\vcenter{\hbox{\xymatrix@!0@R=15mm@C=19mm{
AM\rtwocell^{Y'C'}_{YC}{^\gamma\xi\ }\ar[d]_-a\xtwocell[rd]{!<3mm,-4mm>}<>{^<1>y}&BN\ar[d]^-b\\
A\ar@/_3mm/[r]_-Y&B
}}}
\end{equation}
\end{defi}

Right comodules from $a$ to $b$ and their morphisms constitute a category that we denote by $\rComod((A,M,a),(B,N,b))$, its composition and identities are taken as the ones in $\mathcal{M}(A,B)\times\SkOpmon(M,N)$, hence the forgetful functor below.
\[
\vcenter{\hbox{\xymatrix@R=1mm{
\rComod((A,M,a),(B,N,b))\ar[r]&\mathcal{M}(A,B)\times\SkOpmon(M,N)\\
(Y,C,y)\ar@{|->}[r]&(Y,C)
}}}
\]

\begin{rem}\label{rem:HorizontalCompComodules}
There is a horizontal composition functor of right comodules given in the following way,
\begin{multline*}
\rComod((A',M',a'),(A'',M'',a''))\times\rComod((A,M,a),(A',M',a'))\\
\xymatrix{\ar[r]&\rComod((A,M,a),(A'',M'',a''))}
\end{multline*}
\[
((Z,D,z),(Y,C,y))
\vcenter{\hbox{\xymatrix@!0@=15mm{
\ar@{|->}[r]&
}}}
\vcenter{\hbox{\xymatrix@!0@=15mm{
AM\ar[r]^-{YC}\ar[d]_-{a}\xtwocell[rd]{}<>{^y}&A'N\ar[d]_-{a'}\ar[r]^-{ZD}\xtwocell[rd]{}<>{^z}&A''L\ar[d]^-{a''}\\
A\ar[r]_-Y&A'\ar[r]_-{Z}&A''
}}}
\]
as well as an identity in $\rComod((A,M,a),(A,M,a))$ as shown below.
\[
\vcenter{\hbox{\xymatrix@!0@=15mm{
AM\ar[r]^-{1}\ar[d]_-{a}&AM\ar[d]^-{a}\\
A\ar[r]_-{1}&A
}}}
\]
Together these constitute a bicategory $\rComod(\mathcal{M})$ which comes equipped with a strict functor $\xymatrix@1@C=5mm{\rComod(\mathcal{M})\ar[r]&\mathcal{M}\times\SkOpmon(\mathcal{M})}$. This is the bicategory of oplax right actions in $\mathcal{M}$, right comodules between them, and morphisms of right comodules. The reader should not confuse $\rComod(\mathcal{M})$ with the bicategory $\Comod(\mathcal{V})$ of comonoids, two sided comodules between them, and their morphisms in suitable a monoidal category $\mathcal{V}$. There is no way to compare, for example, the objects of these bicategories: the data for a comonoid in $\mathcal{V}$ consist of an object $V$ and two arrows $\xymatrix@1@C=5mm{V\ar[r]&VV}$ and $\xymatrix@1@C=5mm{V\ar[r]&I}$ in $\mathcal{V}$; and the data for an object in $\rComod(\mathcal{M})$ consist of a right skew monoidale $M$, an object $A$, and an oplax right $M$-action $\xymatrix@1@C=5mm{AM\ar[r]&A}$ in $\mathcal{M}$.

There are other reasonable names for the objects, arrows, and cells of $\rComod(\mathcal{M})$ which one might be tempted to give. In an action-oriented approach one might say: oplax right actions, oplax morphisms of oplax right actions, and transformations of oplax right actions between them. Although one may feel inclined to reserve the name of \emph{oplax morphisms of oplax actions} for the case of $\id_M$-comodules,
\[
\vcenter{\hbox{\xymatrix@!0@=15mm{
AM\ar[r]^-{Y1}\ar[d]_-{a}\xtwocell[rd]{}<>{^y}&BM\ar[d]^-{b}\\
A\ar[r]_-{Y}&B
}}}
\]
which is certainly the case when we mention them in the introduction. When $\mathcal{M}$ is a locally discrete monoidal bicategory, i.e. it is obtained by adding identity cells to a monoidal category, one has the usual notion of morphism between two actions.

Perhaps for a more module-oriented approach to $\rComod(\mathcal{M})$ one may give the names: oplax right modules, oplax morphisms of oplax right modules, and transformations of oplax right modules between them. We opted for the ones that are conveniently shorter because of how they fit in the forthcoming theorems, particularly in Corollary~\ref{cor:ComodulesAreComodules}; where right comodules for an opmonoidal arrow are comodules for a comonad in $\mathcal{M}$.
\end{rem}

\begin{rem}
For two oplax $M$-actions $a$ and $\xymatrix@1@C=5mm{a':AM\ar[r]&A}$, right comodules $(\id_{A},\id_{M},y)$ from $a$ to $a'$ are nothing but cells of oplax actions $y$ from $a$ to $a'$: axiom~\eqref{ax:COM1} for $(\id_{A},\id_{M},y)$ is \eqref{ax:OLA4} for $y$, and axiom~\eqref{ax:COM2} for $(\id_{A},\id_{M},y)$ is \eqref{ax:OLA5} for $y$. Hence, for a right skew monoidale $M$ and an object $A$ in $\mathcal{M}$, we recover the categories $\OplaxAct(M;A)$ from $\rComod(\mathcal{M})$ by taking the pullback below,
\[
\vcenter{\hbox{\xymatrix@!0@R=15mm@C=25mm{
**{!<5mm>}\OplaxAct(M;A)\ar[r]\ar[d]\pb{rd}&\rComod(\mathcal{M})\ar[d]\\
\mathbbm{1}\ar[r]_-{(\id_A,\id_M)}&**{!<-10mm>}\mathcal{M}\times\SkOpmon(\mathcal{M})
}}}
\]
which picks those comodules in $\rComod(\mathcal{M})$ of the form $(\id_{A},\id_{M},y)$ between oplax $M$-actions on $A$.

For a fixed opmonoidal arrow $\xymatrix@1@C=5mm{C:M\ar[r]&N}$, right $C$-comodules from $a$ to $b$ also constitute a category which we denote by $\rComod_C((A,a),(B,b))$. This category may be described by a pullback along the forgetful functor from $\rComod((A,M,a),(B,N,b))$ down to $\SkOpmon(M,N)$ as shown below.
\[
\vcenter{\hbox{\xymatrix@!0@R=15mm@C=35mm{
**{!<8mm>}\rComod_C((A,a),(B,b))\ar[r]\ar[d]\pb{rd}&**{!<-8mm>}\rComod((A,M,a),(B,N,b))\ar[d]\\
\mathbbm{1}\ar[r]_-{C}&\SkOpmon(M,N)
}}}
\]
\end{rem}
\begin{eje}\label{exa:ARightComodule}
Now for a biduality $R\dashv R\ot$ and an adjunction $i\dashv i^*$ in an opmonadic-friendly monoidal bicategory the items of Remark~\ref{rem:SimpleArrows} show an even closer relationship. The identity arrow on $R$ comes equipped with a $i\ob 1$-comodule structure from $i^*1$ to $e1$ given by the square below.
\[
\vcenter{\hbox{\xymatrix@!0@=15mm{
RR\ar[r]^{1i\ob 1}\ar[d]_{i^*1}&RR\ot R\ar[d]^{e1}\\
R\ar[r]_{1}\ar@{}[ru]|*[@]{\cong}&R
}}}
\]
\end{eje}
\begin{eje}\label{exa:ComodulesForCoalgebroids}
Let $R$ and $S$ be two $k$-algebras and $C$ an $R|S$-coalgebroid. In Lemma~\ref{lem:CoalgebroidsAsQuantum} we saw that $C$ is an opmonoidal arrow between enveloping monoidales in $\Mod_k$. In a similar fashion, the objects of $\rComod_C((R,e1),(S,e1))$ may be described in the language of classical ring and module theory. A comodule over the opmonoidal arrow $C$ consists of a module $Y$ in $R$-$\Mod$-$S$ together with a coaction morphism $\xymatrix@1@C=5mm{\varrho:Y\ar[r]&Y\otimes_SC}$ in $R$-$\Mod$-$S$ in which $Y\otimes_SC$ has the module structure from $R$ to $S$ given by $r(a\otimes c)s=a\otimes rcs$, that is
\begin{itemize}
\item $\varrho(as)=\sum a_{(1)}\otimes a_{(2)}s$
\item $\varrho(ra)=\sum a_{(1)}\otimes ra_{(2)}$
\end{itemize}
subject to the following axioms.
\begin{enumerate}
\item $\sum ra_{(1)}\otimes a_{(2)}=\sum a_{(1)}\otimes a_{(2)}r$
\item $(Y,\varrho)$ forms a $C$-comodule in the category $R$-$\Mod$-$S$
\end{enumerate}
Note that using the two sided $R$-module structure on $Y\otimes_SC$ given by $r\cdot(a\otimes c)\cdot r'=ra\otimes cr'$ item (i) may be rewritten as follows.
\begin{itemize}
\item[(i').] The image of the coaction $\varrho$ is in the $R$-centralizer of $Y\otimes_SC$, that is $r\cdot\varrho(a)=\varrho(a)\cdot r$.
\end{itemize}
\end{eje}

For the rest of this section our goal is to simplify the definition of comodules for opmonoidal arrows between enveloping monoidales in an opmonadic-friendly monoidal bicategory $\mathcal{M}$. And, in the two theorems that follow we apply the same technique used in Corollary~\ref{cor:Opmon_is_OplaxAct_Local} to simplify the definition of coalgebroids in terms of opmonoidal arrows to coalgebroids in terms of oplax actions. This technique consists of two steps: that of Theorem~\ref{teo:BidualityOpmonoidalOplaxAction}, which is basically the transposition along a biduality; and that of Theorems~\ref{teo:OpmonadicityOpmonoidal} or \ref{teo:OpmonadicityOplaxActions}, where the main tool is the universal property of an opmonadic adjunction. Hence the two theorems below: Theorem~\ref{teo:BidualityComodule} is the first step which corresponds to the use of transposition along bidualities, although, in this case it is considerably simpler; and Theorem~\ref{teo:OpmonadicityComodules} below is the second step which is analogous to the one that relies on the opmonadicity of an adjunction. We combine these two results in Corollary~\ref{cor:Comodules_Local} to obtain an equivalence between comodules for opmonoidal arrows between enveloping monoidales and certain oplax morphisms of oplax actions.

\begin{teo}\label{teo:BidualityComodule}
Let $\mathcal{M}$ be a monoidal bicategory. For every right skew monoidale $M$, every biduality $S\dashv S\ot$, every oplax right $M$-action $\xymatrix@1@C=5mm{a:AM\ar[r]&A}$, and every opmonoidal arrow $\xymatrix@1@C=5mm{C:M\ar[r]&S\ot S}$ in $\mathcal{M}$ there is an isomorphism between the categories,
\[
\rComod_C((A,a),(S,e1))\cong\rComod_{\id_{M}}((A,a),(S,s))
\]
\[
\vcenter{\hbox{\xymatrix@!0@=15mm{
AM\ar[r]^-{YC}\ar[d]_-{a}\xtwocell[rd]{}<>{^}&SS\ot S\ar[d]^-{e1}\\
A\ar[r]_-{Y}&S
}}}
\quad
\vcenter{\hbox{\xymatrix@!0@=15mm{
\ar@{<~>}[r]&
}}}
\quad
\vcenter{\hbox{\xymatrix@!0@=15mm{
AM\ar[r]^-{Y1}\ar[d]_-{a}\xtwocell[rd]{}<>{^}&SM\ar[d]^-{s}\\
A\ar[r]_-{Y}&S
}}}
\]
where $\xymatrix@1@C=5mm{s:SM\ar[r]&S}$ is the oplax right $M$-action which corresponds to $C$ under Theorem~\ref{teo:BidualityOpmonoidalOplaxAction}.
\end{teo}
\begin{proof}

The objects of these two categories differ only by the isomorphism
\[
\vcenter{\hbox{\xymatrix@!0@C=15mm{
SM\ar[rr]^{s}\ar[rd]_{1C}&&S\\
&SS\ot S\ar[ru]_{e1}\ar@{}[u]|*[@]{\cong}&
}}}
\]
induced by the equivalence of Theorem~\ref{teo:BidualityOpmonoidalOplaxAction} between the opmonoidal arrow $C$ and the oplax $M$-action $s$.
\end{proof}

Now, in the following theorem the comodule in Example~\ref{exa:ARightComodule} plays the role of the opmonoidal left adjoint for the ``opmonoidal $\dashv$ monoidal opmonadicity'' of Theorem~\ref{teo:OpmonadicityOpmonoidal}, but in the bicategory $\rComod(\mathcal{M})$ instead of $\Opmon(\mathcal{M})$.

\begin{teo}\label{teo:OpmonadicityComodules}
Let $\mathcal{M}$ be an opmonadic-friendly monoidal bicategory. For every biduality $R\dashv R\ot$, every opmonadic adjunction $i\ob\dashv i\ot$ in $\mathcal{M}$,
\[
\vcenter{\hbox{\xymatrix{
R\ot\dtwocell_{i\ob}^{i\ot}{'\dashv}\\
I
}}}
\]
every monoidale $N$, every pseudo right action $\xymatrix@1@C=5mm{b:BN\ar[r]&B}$, and every pair of opmonoidal arrows $\xymatrix@1@C=5mm{C:R\ot R\ar[r]&N}$ and $\xymatrix@1@C=5mm{D:R\ar[r]&N}$ which correspond to each other under the equivalence of Theorem~\ref{teo:OpmonadicityOpmonoidal}, there is an isomorphism of categories
\[
\rComod_C((R,e1),(B,b))\cong\rComod_D((R,i^*1),(B,b))
\]
\[
\vcenter{\hbox{\xymatrix@!0@=15mm{
RR\ot R\ar[r]^-{YC}\ar[d]_-{e1}\xtwocell[rd]{}<>{^\bar{y}}&BN\ar[d]^-{b}\\
R\ar[r]_-{Y}&B
}}}
\quad
\vcenter{\hbox{\xymatrix@!0@=15mm{
\ar@{<~>}[r]&
}}}
\quad
\vcenter{\hbox{\xymatrix@!0@=15mm{
RR\ar[r]^-{YD}\ar[d]_-{i^*\!1}\xtwocell[rd]{}<>{^y}&BN\ar[d]^-{b}\\
R\ar[r]_-{Y}&B
}}}
\]
given by precomposition with the $i\ob 1$-comodule in Example~\ref{exa:ARightComodule}.
\end{teo}
\begin{proof}

Let $H$ be the functor in the statement; its action on objects is given below.
\[
\vcenter{\hbox{\xymatrix{
H:\rComod_C((R,e1),(B,b))\ar[r]&\rComod_D((R,i^*1),(B,b))
}}}
\]
\[
\left(
\vcenter{\hbox{\xymatrix{
R\ar[d]^-{Y}\\
B
}}}
\quad,
\vcenter{\hbox{\xymatrix@!0@=15mm{
RR\ot R\ar[r]^-{YC}\ar[d]_-{e1}\xtwocell[rd]{}<>{^\bar{y}}&BN\ar[d]^-{b}\\
R\ar[r]_-{Y}&B
}}}
\right)
\vcenter{\hbox{\xymatrix@!0{\ar@{|->}[r]&
}}}
\left(
\vcenter{\hbox{\xymatrix{
R\ar[d]^-{Y}
\\
B
}}}
\quad,
\vcenter{\hbox{\xymatrix@!0@=15mm{
RR\ar[rd]|-{1i\ob 1}\ar@/_3mm/[rdd]_-{i^*1}\ar@/^3mm/[rrd]^-{YD}&&\\
\ar@{}[rru]|*[@]{\cong}&RR\ot R\ar[r]^-{YC}\ar[d]_-{e1}\xtwocell[rd]{}<>{^\bar{y}}&BN\ar[d]^-{b}\\
\ar@{}[ruu]|*[@]{\cong}&R\ar[r]_-{Y}&B
}}}
\right)
\]
This functor is faithful because it is equal to the identity on the underlying cells $\xymatrix@1@C=5mm{\gamma:Y\ar[r]&Y'}$ in $\mathcal{M}$. It is essentially surjective on objects since for every $D$-comodule $(\xymatrix@1@C=5mm{Y:R\ar[r]&B},y)$ from $e1$ to $b$ in the codomain of $H$, the source and target of $y$ have a structure of module for the monad induced by $1i\ob 1\dashv 1i\ot 1$,
\[
\vcenter{\hbox{\xymatrix@!0@C=6.5mm@R=11mm{
&&&&&RR\ot R\ar[rd]^-{e1}&&&\\
RR\ar[rr]_-{1i\ob 1}&&RR\ot R\ar[rr]_-{1i\ot 1}\ar@/^3mm/[]!<-3mm,2mm>;[rrru]^-{1}\xtwocell[rrru]{}<>{^1\varepsilon\ob 1\quad}&&RR\ar[rr]_-{i^*1}\ar[]!<-1mm,1mm>;[ru]|-{\ 1i\ob 1}&&R\ar[rr]_-{Y}&&B
}
\quad\xymatrix@!0@C=13mm@R=11mm{
&&&&\\
RR\ar[r]_-{1i\ob 1}\ar@/^9mm/[rrr]^-{1}\xtwocell[rrr]{}<>{^<-3>1\varphi\ }&RR\ot R\ar[r]_-{1i\ot 1}&RR\ar[r]_-{YD}&BN\ar[r]_-{b}&B
}}}
\]
where $\varphi$ is the action~\eqref{eq:redundant_action1}. Furthermore, the axioms for a $D$-comodule together with the fact that $\varphi$ is defined in terms of $D^2$ and $D^0$ imply that $y$ is a morphism of modules for the monad induced by the adjunction $1i\ob 1\dashv 1i\ot 1$.
\begin{align*}
&
\vcenter{\hbox{\xymatrix@!0@R=10mm@C=10mm{
&&BN\ar[rrd]^-{1u1}\ar@/^14mm/[rrrrdd]^-{1}&&&\ar@{}[ld]|<<<<<*=0[@]{\cong}_-{1\lambda}&&&\\
&&&&BNN\ar[rrd]^-{1m}&&&&\\
{\xtwocell[rrrrru]{}<>{^<-2>YD^0D\qquad}}&&RRR\ar[rrd]|-{1i^*1}\ar[rru]|-{YDD}\xtwocell[rrrr]{}<>{^YD^2\quad}&&&&BN\ar[rrd]^-{b}&&\\
RR\ar[rr]_-{1i\ob 1}\ar[rru]^-{1i1}\ar@/^/[rruuu]^-{YD}&&RR\ot R\ar[rr]_-{1i\ot 1}\ar@{}[u]|*[@]{\cong}&&RR\ar[rru]|-{YD}\ar[rrd]_-{i^*1}\xtwocell[rrrr]{}<>{^y}&&&&B\\
&&&&&&R\ar[rru]_-{Y}&&
}}}
\\
\stackrel{\eqref{ax:COM1}}{=}&
\vcenter{\hbox{\xymatrix@!0@R=5mm@C=10mm{
&&BN\ar[rrdd]^-{1u1}\ar@/^10mm/[rrrrrddd]^-{1}&&&\ar@{}[ldd]|<<<<<*=0[@]{\cong}_-{1\lambda}&&&\\
&&&&&&&&\\
&&&&BNN\ar[rrdd]_-{b1}\ar[rrrd]^-{1m}&&&&\\
&&&&&&&BN\ar[rddd]^-{b}\ar@{}[ld]|*=0[@]{\cong}^-{b^2}&\\
{\xtwocell[rrrrruu]{}<>{^<-2>YD^0D\qquad}}&&RRR\ar[rrdd]|-{1i^*1}\ar[rruu]|-{YDD}\ar[rrrd]|-{i^*11}\xtwocell[rrrr]{}<>{^<-1>yD\quad}&&&&BN\ar[rrdd]^-{b}&&\\
&&&&&RR\ar[rddd]|-{i^*1}\ar[ru]^-{YD}\xtwocell[rrrd]{}<>{^<1>y}&&&\\
RR\ar[rr]_-{1i\ob 1}\ar[rruu]^-{1i1}\ar@/^/[rruuuuuu]^-{YD}&&RR\ot R\ar[rr]_-{1i\ot 1}\ar@{}[uu]|*[@]{\cong}&&RR\ar[rrdd]_-{i^*1}\ar@{}[ru]|*=0[@]{\cong}&&&&B\\
&&&&&&&&\\
&&&&&&R\ar[rruu]_-{Y}&&
}}}
\\
\stackrel{\eqref{ax:COM2}}{=}&
\vcenter{\hbox{\xymatrix@!0@R=5mm@C=10mm{
&&BN\ar[rrd]^-{1u1}\ar@/_3mm/[rrrrdddd]|-{1}\ar@/^10mm/[rrrrrddd]^-{1}&&&\ar@{}[ldd]|<<<*=0[@]{\cong}_<<<<{1\lambda}&&&\\
&&&&BNN\ar[rrddd]|-{b1}\ar[rrrdd]^-{1m}\ar@{}[ldd]^<<<*=0[@]{\cong}^<{b^0}&&&&\\
&&&&&&&&\\
&&R\ar[rrrdd]^-{i1}&&&&&BN\ar[rddd]^-{b}\ar@{}[ld]|*=0[@]{\cong}^-{b^2}&\\
{\xtwocell[rrrrrru]{}<>{^<-2>\varepsilon 1\ }}&&RRR\ar[rrdd]|-{1i^*1}\ar[rrrd]|-{i^*11}&&&&BN\ar[rrdd]^-{b}&&\\
&&&&&RR\ar[rddd]|-{i^*1}\ar[ru]^-{YD}\xtwocell[rrrd]{}<>{^<1>y}&&&\\
RR\ar[rr]_-{1i\ob 1}\ar[rruu]|-{1i1}\ar[rruuu]^>>>>>{i^*1}\ar@/^18mm/[rrrrru]^-{1}\ar@/^/[rruuuuuu]^-{YD}&&RR\ot R\ar[rr]_-{1i\ot 1}\ar@{}[uu]|*[@]{\cong}&&RR\ar[rrdd]_-{i^*1}\ar@{}[ru]|*=0[@]{\cong}&&&&B\\
&&&&&&&&\\
&&&&&&R\ar[rruu]_-{Y}&&
}}}
\\
\stackrel{\eqref{ax:SKM2}}{=}&
\vcenter{\hbox{\xymatrix@!0@R=5mm@C=10mm{
&&BN\ar@/^10mm/[rrrrrddd]^-{1}&&&&&&\\
&&&&&&&&\\
&&&&&&&&\\
&&&&&&&BN\ar[rddd]^-{b}&\\
&&&&**[l]RR\ot R\ar[rrdddd]|-{e1}&&&&\\
&&&&&RR\ar[rddd]^-{i^*1}\ar[rruu]^-{YD}\xtwocell[rrrd]{}<>{^y}&&&\\
RR\ar[rr]_-{1i\ob 1}\ar@/^18mm/[rrrrru]^-{1}\ar@/^/[rruuuuuu]^-{YD}&&RR\ot R\ar[rr]_-{1i\ot 1}\ar@/^5mm/[]!<-2mm,1mm>;[rruu]^-{1}\xtwocell[rruu]{}<>{^<0>1\varepsilon\ob 1\quad}&&RR\ar[rrdd]_-{i^*1}\ar[uu]^-{1i\ob 1}&&&&B\\
&&&&&&&&\\
&&&&&\ar@{}[luuuuu]|*=0[@]{\cong}&R\ar[rruu]_-{Y}&&
}}}
\end{align*}
Therefore by opmonadicity of $1i\ob 1\dashv 1i\ot 1$ there exists a cell $\bar{y}$ as below,
\[
\vcenter{\hbox{\xymatrix@!0@=15mm{
RR\ot R\ar[r]^-{YC}\ar[d]_-{e1}\xtwocell[rd]{}<>{^\bar{y}}&BN\ar[d]^-{b}\\
R\ar[r]_-{Y}&B
}}}
\]
that composed with $1i\ob 1$ is equal to $y$. The cell $\bar{y}$ provides the arrow $Y$ with a $C$-comodule structure; one proves axiom~\eqref{ax:COM1} for $\bar{y}$ by precomposing both sides with the opmonadic left adjoint $\xymatrix@1@C=5mm{1i\ob 1i\ob 1:RRR\ar[r]&RR\ot RR\ot R}$, to get each side of axiom~\eqref{ax:COM1} for $y$, which are equal. And to prove axiom~\eqref{ax:COM2} for $\bar{y}$ precompose both sides of the axiom with the epimorphic cell
\[
\vcenter{\hbox{\xymatrix@!0@=12mm{
&RR\ot R\ar[rrd]^-{1}\ar[ld]_-{1i\ot 1}\xtwocell[rd]{}<>{^<2>1\varepsilon\ob 1\quad}&&\\
RR\ar[rrr]_-{1i\ob 1}&&&RR\ot R
}}}
\]
at $RR\ot R$ to get each side of axiom~\eqref{ax:COM2} for $y$, which are equal. Therefore $H(Y,\bar{y})=(Y,y)$, which means $H$ is surjective on objects. The proof that $H$ is full consists of a similar calculation for axiom~\eqref{ax:COM3}, ergo $H$ is an isomorphism.
\end{proof}

\begin{rem}
In view of Theorem~\ref{teo:OpmonadicityOpmonoidal}, by varying the opmonoidal arrows $C$ and $D$ we may lift the isomorphisms in the previous theorem to an equivalence
\[
\rComod((R,R\ot R,e1),(B,N,b))\simeq\rComod((R,R,i^*1),(B,N,b))
\]
between the hom categories of $\rComod(\mathcal{M})$.
\end{rem}

Together, the two previous theorems imply the following.
\begin{cor}\label{cor:Comodules_Local}
Let $\mathcal{M}$ be an opmonadic-friendly monoidal bicategory. For every pair of bidualities $R\dashv R\ot$ and $S\dashv S\ot$, every opmonoidal arrow $\xymatrix@1@C=5mm{C:R\ot R\ar[r]&S\ot S}$, and every opmonadic adjunction $i\ob\dashv i\ot$ in $\mathcal{M}$,
\[
\vcenter{\hbox{\xymatrix{
R\ot\xtwocell[d]{}_{i\ob}^{i\ot}{'\dashv}\\
I
}}}
\]
there is an isomorphism of categories,
\[
\rComod_C((R,e1),(S,e1))\cong\rComod_{\id_{R}}((R,i^*1),(S,s))
\]
where $\xymatrix@1@C=5mm{s:SM\ar[r]&S}$ is the oplax right $M$-action which corresponds to $C$ under Corollary~\ref{cor:Opmon_is_OplaxAct_Local}. Moreover, the pentagon below commutes strictly,
\begin{equation}
\label{eq:ComodulePentagon}
\vcenter{\hbox{\xymatrix@!0@R=20mm@C=11mm{
&**[l]\rComod_C((R,e1),(S,e1))\ar[rr]_-{\cong}\ar[ld]^-{\cong}&&**[r]\rComod_D((R,i^*1),(S,e1))\ar[rd]_-{\cong}&\\
**{!<5mm>}\rComod_{\id_{R\ot R}}((R,e1),(S,\widehat{s}))\ar[rrd]^-{\cong}&&&&**{!<-6mm>}\rComod_{\id_{R}}((R,i^*1),(S,s))\\
&&\rComod_{i\ob 1}((R,i^*1),(S,\widehat{s}))\ar[rru]^-{\cong}&&
}}}
\end{equation}
where $\xymatrix@1@C=5mm{\widehat{s}:SR\ot R\ar[r]&S}$ is the oplax right action that corresponds to $C$ under the equivalence in Theorem~\ref{teo:BidualityOpmonoidalOplaxAction}, $\xymatrix@1@C=5mm{D:R\ar[r]&S\ot S}$ is the opmonoidal arrow that corresponds to $C$ under the equivalence in Theorem~\ref{teo:OpmonadicityOpmonoidal}, and the edges of the pentagon are instances of the isomorphisms in Theorems~\ref{teo:OpmonadicityComodules} and \ref{teo:BidualityComodule}.\qed
\end{cor}

\begin{rem}\label{rem:SimpleActions}
Given a duality $R\dashv R\ot$ and an opmonadic adjunction $i\ob\dashv i\ot$ with opposite adjunction $i\dashv i^*$ in an opmonadic-friendly monoidal bicategory $\mathcal{M}$,
\[
\vcenter{\hbox{\xymatrix{
R\dtwocell_{i\ob}^{i\ot}{'\dashv}\\
I
}}}
\]
we can take comodules between the actions in Remark~\ref{rem:SimpleArrows} around the pentagon~\eqref{eq:ComodulePentagon} above. Start with the identity comodule on $\xymatrix@1@C=5mm{e1:RR\ot R\ar[r]&R}$ that lives in the source category of the pentagon, taking it down the first equivalence does not change it, taking it down-right the second equivalence, as well as to take it from the source to the right, gives the comodule from Example~\ref{exa:ARightComodule}, and to take it all the way to the target gives the identity comodule on $\xymatrix@1@C=5mm{i^*1:RR\ar[r]&R}$.
\[
\vcenter{\hbox{\xymatrix@!0@=15mm{
RR\ot R\ar[r]^-{1}\ar[d]_-{e1}&RR\ot R\ar[d]^-{e1}\\
R\ar[r]_-{1}&R
}}}
\xymatrix@C=5mm{\ar@{|->}[r]&}
\vcenter{\hbox{\xymatrix@!0@=15mm{
RR\ar[r]^-{1i\ob 1}\ar[d]_-{i^*1}&RR\ot R\ar[d]^-{e1}\\
R\ar[r]_-{1}\ar@{}[ru]|*[@]{\cong}&R
}}}
\xymatrix@C=5mm{\ar@{|->}[r]&}
\vcenter{\hbox{\xymatrix@!0@=15mm{
RR\ar[r]^-{1}\ar[d]_-{i^*1}&RR\ar[d]^-{i^*1}\\
R\ar[r]_-{1}&R
}}}
\]
\end{rem}

\begin{eje}\label{exa:ComoduleForOplaxAction}
Comodules for opmonoidal arrows in $\Mod_k$ that live in the source of the pentagon of Corollary~\ref{cor:Comodules_Local} were described in Example~\ref{exa:ComodulesForCoalgebroids}. Let us describe what is a comodule in the target of the pentagon, that is, an $\id_R$-comodule from $i^*1$ to $s$ in $\Mod_k$. Here $s$ is an $R|S$-coalgebroid via oplax actions which is denoted by a module $C$ in $SR$-$\Mod$-$S$ as in Example~\ref{exa:CoalgebroidsAsOplaxActions}. An object in $\rComod_{\id_{R}}((R,i^*1),(S,s))$ consists of a module $Y$ in $R$-$\Mod$-$S$ together with a module morphism $\xymatrix@1@C=5mm{\widetilde{\varrho}:Y\ar[r]&Y\otimes_S C}$ in $\Mod$-$S$ where $Y\otimes_S C$ takes the right $S$ module structure given by $(y\otimes c).s=y\otimes cs$, hence the condition below,
\begin{itemize}
\item $\widetilde{\varrho}(as)=\sum a_{(1)}\otimes a_{(2)}s$
\end{itemize}
and it is subject to the following axiom.
\begin{enumerate}
\item $(Y,\widetilde{\varrho})$ forms a $C$-comodule in the category $\Mod$-$S$
\end{enumerate}

What changed from Example~\ref{exa:ComodulesForCoalgebroids} is that the coaction $\widetilde{\varrho}$ is not necessarily a left $R$-morphism, condition \ref{exa:ComodulesForCoalgebroids}.(i) vanishes, and $(Y,\widetilde{\varrho})$ is a $C$-comodule in $\Mod$-$S$ rather than $R$-$\Mod$-$S$.

\end{eje}

At this point we pause to recall the main results of this and the previous sections. These may be arranged in the chart below using the bar notation for equivalences. So, let $R\dashv R\ot$ and $S\dashv S\ot$ be two bidualities and let $i\dashv i^*$ be an adjunction whose opposite is opmonadic all of which are in an opmonadic-friendly monoidal bicategory $\mathcal{M}$. On the left column we have the equivalence of  Corollary~\ref{cor:Opmon_is_OplaxAct_Local}, and on the right column we have the equivalence of Corollary~\ref{cor:Comodules_Local} for a fixed pair of items in the left column.
\begin{center}
\begin{tabular}{|c|c|}
\hline
\begin{tabular}{c}
Opmonoidal arrow\\
$\xymatrix@C=5mm{
R\ot R\ar[r]^-{C}&S\ot S
}$
\\
\hline
$\xymatrix@C=5mm{
SR\ar[r]_-{s}&R
}$
\\
Oplax action
\end{tabular}
$\simeq$&
\begin{tabular}{c}
$C$-comodule\\
$\vcenter{\hbox{\xymatrix@!0@=15mm{
RR\ot R\ar[r]^-{YC}\ar[d]_-{e1}\xtwocell[rd]{}<>{^\bar{y}}&S\ot SS\ar[d]^-{e1}\\
R\ar[r]_-{Y}&S
}}}$
\\
\hline
$\vcenter{\hbox{\xymatrix@!0@=15mm{
RR\ar[r]^-{Y1}\ar[d]_-{i^*1}\xtwocell[rd]{}<>{^y}&SR\ar[d]^-{s}\\
R\ar[r]_-{Y}&S
}}}$
\\
$\id_{R}$-comodule to $s$
\end{tabular}
$\cong$\\
\hline
\end{tabular}
\end{center}
But this table is still incomplete: there is one more equivalence of categories to add at the bottom of the right column and that is precisely what the next theorem is about. This new equivalence is another application of opmonadicity, but this time from the adjunction $i\dashv i^*$. The target category is the category $\mathcal{M}(I,S)_{\mathcal{M}(I,c)}$ of comodules for a comonad $\xymatrix@1@C=5mm{c:S\ar[r]&S}$ based at $I$. What completes the chart is our concluding corollary below, in which the interesting case is when $c$ is the comonad induced by an opmonoidal arrow or by an oplax action as discussed in Remark~\ref{rem:InducedComonad}.

\begin{teo}\label{teo:SkewActions_are_Coactions}
For every biduality $R\dashv R\ot$, every opmonadic adjunction
\[
\xymatrix{R\dtwocell_{i}^{i^*}{'\dashv}\\
I
}
\]
and every oplax right $R$-action $\xymatrix@1@C=5mm{b:BR\ar[r]&B}$ with respect to the right skew monoidal structure induced by $i\dashv i^*$ as in Lemma~\ref{lem:OneRightSkewMonoidale}, there is an equivalence of categories
\[
\rComod_{\id_{R}}((R,i^*1),(B,b))\simeq\mathcal{M}(I,B)_{\mathcal{M}(I,c)}
\]
where $\xymatrix@1@C=5mm{c:B\ar[r]&B}$ is the comonad induced by $b$ as in Remark~\ref{rem:InducedComonad}; and the equivalence is given as follows.
\[
\left(
\vcenter{\hbox{\xymatrix{
R\ar[d]^-{Y}\\
B
}}}
\ ,\ 
\vcenter{\hbox{\xymatrix@!0@=15mm{
RR\ar[r]^-{Y1}\ar[d]_-{i^*\!1}\xtwocell[rd]{}<>{^y}&BR\ar[d]^-{b}\\
R\ar[r]_-{Y}&B
}}}
\right)
\vcenter{\hbox{\xymatrix@!0{\ar@{|->}[r]&}}}
\left(
\vcenter{\hbox{\xymatrix{
I\ar[d]^-{i}\\
R\ar[d]^-{Y}\\
B}}}
\ ,
\vcenter{\hbox{\xymatrix@!0@R=3.75mm@C=5mm{
&&&&B\ar[rrdddd]^-{1i}\ar@/^8mm/[rrrrdddddddd]^-{c}&&&&\\
&&&&&&&&\\
&&&&&&&&\\
&&&&&&&&\\
&&R\ar[rrd]^-{1i}\ar[rruuuu]^-{Y}&&&&BR\ar[rrdddd]^-{b}&&\\
&&&&RR\ar[rddd]^-{i^*\!1}\ar[rru]^-{Y1}\ar@{}[uuuu]|*=0[@]{\cong}\xtwocell[rrrrddd]{}<>{^<-.5>y}&&&&\\
&&&&&&&&\\
&&&&&&&&\\
I\ar[rrr]_-{i}\ar[rruuuu]^-{i}&&&R\ar[rr]_-{1}\ar[ruuu]^-{i1}\ar@{}[luuuu]|*=0[@]{\cong}\xtwocell[rr]{!<2mm,0mm>}<>{^<-2>\eta 1 }&&R\ar[rrr]_-{Y}&&&B
}}}
\right)
\]
\end{teo}
\begin{proof}[Sketch]

The action of the proposed functor in the statement on the structure cells $y$ may be factorised by first taking the mate of $y$ with respect to the adjunction $i\dashv i^*$ and then by precomposing with $i$.

For an arrow $\xymatrix@1@C=5mm{Y:R\ar[r]&B}$, cells $y$ as in the statement are in bijection with their mates with respect to the adjunction $i\dashv i^*$.
\[
\vcenter{\hbox{\xymatrix@!0@=15mm{
RR\ar[r]^-{Y1}\ar[d]_-{i^*1}\xtwocell[rd]{}<>{^y}&BR\ar[d]^-{b}\\
R\ar[r]_-{Y}&B
}}}
\vcenter{\hbox{\xymatrix{
\ar@{<~>}[r]&
}}}
\vcenter{\hbox{\xymatrix@!0@=15mm{
RR\ar[r]^-{Y1}&BR\ar[d]^-{b}\\
R\ar[r]_-{Y}\ar[u]^-{i1}\xtwocell[r]{}<>{^<-4.5>\tilde{y}}&B
}}}
\]
A cell $y$ satisfies the axioms \eqref{ax:COM1} and \eqref{ax:COM2} that turn $(Y,y)$ into an $\id_R$-comodule from $i^*1$ to $b$ if and only if its mate $\tilde{y}$ satisfies two other axioms \eqref{ax:COM1}\textsuperscript{mate} and \eqref{ax:COM2}\textsuperscript{mate} obtained by taking mates of each side of the original ones for $y$. Call $\rComod_{\id_{R}}^{\text{mate}}((R,i^*1),(B,b))$ the category whose objects consist of an arrow $\xymatrix@1@C=5mm{Y:R\ar[r]&B}$ together with a cell $\tilde{y}$ as above, satisfying axioms \eqref{ax:COM1}\textsuperscript{mate} and \eqref{ax:COM2}\textsuperscript{mate}. The arrows of $\rComod_{\id_{R}}^{\text{mate}}((R,i^*1),(B,b))$ are cells $\xymatrix@1@C=5mm{\gamma:Y\ar[r]&Y'}$ obtained in a similar fashion. Hence the isomorphism of categories below.
\[
\rComod_{\id_{R}}((R,i^*1),(B,b))\cong\rComod_{\id_{R}}^{\text{mate}}((R,i^*1),(B,b))
\]

Now define $\mathcal{Z}(I,B)$ as the category whose objects are triples $(X,x,\zeta)$,
\[
\vcenter{\hbox{\xymatrix{
I\ar[r]^-{X}&B
}}}
\qquad
\vcenter{\hbox{\xymatrix@!0@=15mm{
B\ar[r]^-{1i}&BR\ar[d]^-{b}\\
I\ar[u]^-{X}\ar[r]_-{X}\xtwocell[r]{}<>{^<-4.5>x}&B
}}}
\qquad
\vcenter{\hbox{\xymatrix{
I\ar[r]_-{i}\ar@/^9mm/[rrr]^-{X}\xtwocell[rrr]{}<>{^<-3>\zeta}&R\ar[r]_-{i^*}&I\ar[r]_-{X}&B
}}}
\]
such that $(X,x)$ is a comodule for the comonad $c$, and the cell $\zeta$ is an action on $X$ with respect to the monad induced by $i\dashv i^*$, satisfying the following compatibility condition,
\begin{equation}
\tag{Z1}\label{ax:Z1}
\vcenter{\hbox{\xymatrix@!0@=15mm{
B\ar@/^9mm/[rrr]^-{1i}\xtwocell[rrr]{}<>{^<-1>x}&&&BR\ar[d]^-{b}\\
I\ar[r]_-{i}\ar[u]^-{X}\ar@/^9mm/[rrr]^-{X}\xtwocell[rrr]{}<>{^<-3>\zeta}&R\ar[r]_-{i^*}&I\ar[r]_-{X}&B
}}}
=
\vcenter{\hbox{\xymatrix@!0@=15mm{
B\ar[r]^-{1i}&BR\ar[r]^-{1i^*}\ar@/^9mm/[rr]^-{1}\xtwocell[rr]{}<>{^<-3>1\varepsilon\ }&B\ar[r]^-{1i}&BR\ar[d]^-{b}\\
I\ar[r]_-{i}\ar[u]^-{X}\ar@{}[ru]|*[@d]{\cong}&R\ar[u]|-{X1}\ar[r]_-{i^*}\ar@{}[ru]|*[@d]{\cong}&I\ar[u]|-{X}\ar[r]_-{X}\xtwocell[r]{}<>{^<-4.5>x}&B
}}}
\end{equation}
An arrow $\xymatrix@1@C=5mm{\chi:(X,x,\zeta)\ar[r]&(X',x',\zeta')}$ in $\mathcal{Z}(I,B)$ is a cell $\xymatrix@1@C=5mm{\chi:X\ar[r]&X'}$ which is simultaneously a morphism of $c$-comodules and a morphism of modules for the monad associated to $i\dashv i^*$.

There is an isomorphism of categories $\mathcal{Z}(I,B)\cong\mathcal{M}(I,B)_{\mathcal{M}(I,c)}$ which is deduced from the redundancy of the action $\zeta$ in the objects of $\mathcal{Z}(I,B)$. Indeed, for every object $(X,x,\zeta)$ in $\mathcal{Z}(I,B)$ one may express $\zeta$ in terms of $x$ and $b^0$ as follows.
\[
\vcenter{\hbox{\xymatrix{
I\ar[r]_-{i}\ar@/^10mm/[rrr]^-{X}&R\ar[r]_-{i^*}\xtwocell[r]{}<>{^<-3>\zeta}&I\ar[r]_-{X}&B
}}}
=
\vcenter{\hbox{\xymatrix@!0@=15mm{
&&B\ar@/^3mm/[rd]^-{1i}\ar@/^18mm/[rdd]^-{1}\xtwocell[rdd]{}<>{^<-7>b^0}&\\
B\ar[r]^-{1i}\ar@/^5mm/[rru]^-{1}&BR\ar[r]^-{1i^*}\ar@/^9mm/[rr]^-{1}\xtwocell[rr]{}<>{^<-3>1\varepsilon\ }&B\ar[r]^-{1i}&BR\ar[d]^-{b}\\
I\ar[r]_-{i}\ar[u]^-{X}\ar@{}[ru]|*[@d]{\cong}&R\ar[u]|-{X1}\ar[r]_-{i^*}\ar@{}[ru]|*[@d]{\cong}&I\ar[u]|-{X}\ar[r]_-{X}\xtwocell[r]{}<>{^<-4.5>x}&B
}}}
\]
And if for an arbitrary $c$-comodule $(X,x)$ one defines $\zeta$ by the equation above, then $(X,x,\zeta)$ becomes an object of $\mathcal{Z}(I,B)$, therefore the functor $\xymatrix@1@C=5mm{\mathcal{Z}(I,B)\ar[r]&\mathcal{M}(I,B)_{\mathcal{M}(I,c)}}$ which forgets the action $\zeta$ is an isomorphism of categories.

Now, the functor $K$ below induced by precomposition with the opmonadic arrow $\xymatrix@1@C=5mm{i:I\ar[r]&R}$ is an equivalence of categories $\rComod_{\id_{R}}^{\text{mate}}((R,i^*1),(B,b))\simeq\mathcal{Z}(I,B)$ because of the opmonadicity of $i\dashv i^*$.
\[
\vcenter{\hbox{\xymatrix{K:\rComod_{\id_{R}}^{\text{mate}}((R,i^*1),(B,b))\ar[rr]&&\mathcal{M}(I,B)_{\mathcal{M}(I,c)}\cong\mathcal{Z}(I,B)
}}}
\]
\[
\left(
\vcenter{\hbox{\xymatrix{
R\ar[d]^-{Y}\\
B
}}}
\quad,
\vcenter{\hbox{\xymatrix@!0@=15mm{
RR\ar[r]^-{Y1}&BR\ar[d]^-{b}\\
R\ar[r]_-{Y}\ar[u]^-{i1}\xtwocell[r]{}<>{^<-4.5>\tilde{y}}&B
}}}
\right)
\vcenter{\hbox{\xymatrix@!0{\ar@{|->}[r]&}}}
\left(
\vcenter{\hbox{\xymatrix{
I\ar[d]^-{i}\\
R\ar[d]^-{Y}\\
B}}}
\quad,
\vcenter{\hbox{\xymatrix@!0@R=7.5mm@C=10mm{
&&B\ar[rdd]^-{1i}&&\\
&&&&\\
&R\ar[rd]_-{1i}\ar[ruu]^-{Y}&&BR\ar[rdd]^-{b}&\\
&&RR\ar[ru]^-{Y1}\ar@{}[uu]|*[@]{\cong}&&\\
I\ar[rr]_-{i}\ar[ruu]^-{i}\ar@{}[rru]|*[@]{\cong}&&R\ar[rr]_-{Y}\ar[u]^-{i1}\xtwocell[rru]{}<>{^<-1>\tilde{y}}&&B
}}}
\right)
\]
This assignation $K$ is a well defined functor because axioms \eqref{ax:COM1}\textsuperscript{mate} and \eqref{ax:COM2}\textsuperscript{mate} translate precisely into the axioms for a $c$-comodule from $I$ to $B$. In fact, the unit axiom is literally the same, and the associative axiom follows from the calculation below.
\[
\vcenter{\hbox{\xymatrix@!0@R=7mm@C=13mm{
&&B\ar[rdd]^-{1i}&&\\
&&&&\\
&R\ar[rd]^-{1i}\ar[ruu]^-{Y}&&BR\ar[rdd]^-{b}&\\
&&RR\ar[ru]^-{Y1}\ar@{}[uu]|*[@]{\cong}&&\\
I\ar[rr]|-{i}\ar[ruu]^-{i}\ar[rdd]_-{i}\ar@{}[rru]|*[@]{\cong}\ar@{}[rrd]|*[@]{\cong}&&R\ar[rr]|-{Y}\ar[u]^-{i1}\ar[d]_-{1i}\xtwocell[rru]{}<>{^<-1>\tilde{y}}&&B\ar[ldd]^-{1i}\\
&&RR\ar[rd]^-{Y1}\ar@{}[rru]|*[@]{\cong}&&\\
&R\ar[ru]_-{i1}\ar[rdd]_-{Y}\xtwocell[rrd]{}<>{^\tilde{y}}&&BR\ar[ldd]^-{b}&\\
&&&&\\
&&B&&\\
}}}
\quad=\quad
\vcenter{\hbox{\xymatrix@!0@R=7mm@C=6.5mm{
&&&&B\ar[rrdd]^-{1i}&&&&\\
&&&&&&&&\\
&&R\ar[rrd]^-{1i}\ar[rruu]^-{Y}&&&&BR\ar[rrdd]^-{b}\ar[dd]|-{11i}&&\\
&&&&RR\ar[rru]^-{Y1}\ar[d]_>>{11i}\ar@{}[uu]|*[@]{\cong}&&&&\\
I\ar[rr]|-{i}\ar[rruu]^-{i}\ar[rrdd]_-{i}\ar@{}[rrrruu]|*[@]{\cong}\ar@{}[rrrrdd]|*[@]{\cong}&&R\ar[rru]^-{i1}\ar[rrd]_-{1i}&&RRR\ar[rr]^-{Y11}\ar@{}[rruu]|*[@]{\cong}&&BRR\ar[dd]^-{b1}\ar@{}[rr]|*[@]{\cong}&&B\ar[lldd]^-{1i}\\
&&&&RR\ar[rrd]_-{Y1}\ar[u]^>>{i11}\xtwocell[rrd]{}<>{^<-3>\tilde{y}1}&&&&\\
&&R\ar[rru]_-{i1}\ar[rrdd]_-{Y}\xtwocell[rrrrd]{}<>{^\tilde{y}}&&&&BR\ar[lldd]^-{b}&&\\
&&&&&&&&\\
&&&&B&&&&\\
}}}
\]
\[
=\quad
\vcenter{\hbox{\xymatrix@!0@R=7mm@C=6.5mm{
&&&&B\ar[rrdd]^-{1i}&&&&\\
&&&&&&&&\\
&&R\ar[rrd]^-{1i}\ar[dd]_-{1i}\ar[rruu]^-{Y}&&&&BR\ar[rrdd]^-{b}\ar[dd]|-{11i}&&\\
&&&&RR\ar[rru]^-{Y1}\ar[d]_-{11i}\ar@{}[uu]|*[@]{\cong}&&&&\\
I\ar[rruu]^-{i}\ar[rrdd]_-{i}\ar@{}[rr]|*[@]{\cong}&&R\ar[rr]^-{1i1}\ar@{}[rruu]|*[@]{\cong}\ar@{}[rrdd]|*[@]{\cong}&&RRR\ar[rr]^-{Y11}\ar@{}[rruu]|*[@]{\cong}&&BRR\ar[dd]^-{b1}\ar@{}[rr]|*[@]{\cong}&&B\ar[lldd]^-{1i}\\
&&&&RR\ar[rrd]_-{Y1}\ar[u]^-{i11}\xtwocell[rrd]{}<>{^<-3>\tilde{y}1}&&&&\\
&&R\ar[rru]_-{i1}\ar[uu]^-{i1}\ar[rrdd]_-{Y}\xtwocell[rrrrd]{}<>{^\tilde{y}}&&&&BR\ar[lldd]^-{b}&&\\
&&&&&&&&\\
&&&&B&&&&\\
}}}
\]
\[
\stackrel{\eqref{ax:COM1}}{=}\quad
\vcenter{\hbox{\xymatrix@!0@R=7mm@C=6.5mm{
&&&&B\ar[rrdd]^-{1i}&&&&\\
&&&&&&&&\\
&&R\ar[rrd]^-{1i}\ar[dd]_-{1i}\ar[rruu]^-{Y}&&&&BR\ar[rrdd]^-{b}\ar[dd]|-{11i}&&\\
&&&&RR\ar[rru]^-{Y1}\ar[d]_-{11i}\ar@{}[uu]|*[@]{\cong}&&&&\\
I\ar[rruu]^-{i}\ar[rrdd]_-{i}\ar@{}[rr]|*[@]{\cong}&&RR\ar[rr]^-{1i1}\ar[rrdd]|-{Y1}\ar@{}[rruu]|*[@]{\cong}\xtwocell[rrrrr]{}<>{^<3>Y\eta 1\quad}&&RRR\ar[rr]^-{Y11}\ar@{}[rruu]|*[@]{\cong}&{\xtwocell[dddd]{}<>{^b^2}}&BRR\ar[dd]^-{b1}\ar[lldd]|-{1i^*1}\ar@{}[rr]|*[@]{\cong}&&B\ar[lldd]^-{1i}\\
&&&&&&&&\\
&&R\ar[uu]^-{i1}\ar[rrdd]_-{Y}&&BR\ar[dd]^-{b}&&BR\ar[lldd]^-{b}&&\\
&{\xtwocell[rrd]{}<>{^<-7>\tilde{y}}}&&&&&&&\\
&&&&B&&&&\\
}}}
\quad=\quad
\vcenter{\hbox{\xymatrix@!0@R=7mm@C=6.5mm{
&&&&B\ar[dd]_-{1i}\ar[rrdd]^-{1i}&&&&\\
&&&&&&&&\\
&&R\ar[rruu]^-{Y}\ar[dd]_-{1i}\ar@{}[rr]|*[@]{\cong}&&BR\ar[rrdd]|-{1i1}\ar[dddd]_-{1}\ar@{}[rr]|*[@]{\cong}\xtwocell[dddd]{}<>{^<-2>1\eta 1}&&BR\ar[dd]|-{11i}\ar[rrdd]^-{b}&&\\
&&&&&&&&\\
I\ar[rruu]^-{i}\ar[rrdd]_-{i}\ar@{}[rr]|*[@]{\cong}&&RR\ar[rruu]|-{Y1}\ar[rrdd]|-{Y1}\ar[dd]_-{i1}&&&{\xtwocell[dddd]{}<>{^b^2}}&BRR\ar[lldd]|-{1i^*1}\ar[dd]^-{b1}\ar@{}[rr]|*[@]{\cong}&&B\ar[lldd]^-{1i}\\
&&&&&&&&\\
&&R\ar[rrdd]_-{Y}&&BR\ar[dd]_-{b}&&BR\ar[lldd]^-{b}&&\\
&{\xtwocell[rrd]{}<>{^<-7>\tilde{y}}}&&&&&&&\\
&&&&B&&&&
}}}
\]
The functor $K$ is automatically faithful since precomposing with an opmonadic arrow is a faithful process. By the opmonadicity of $i\dashv i^*$ the functor $K$ is essentially surjective on objects and full. Indeed, let $(X,x,\zeta)$ be an object in $\mathcal{Z}(I,B)$, since $(X,\zeta)$ is a module for the monad induced by $i\dashv i^*$ there exists an arrow $\xymatrix@1@C=5mm{Y:R\ar[r]&B}$ and an isomorphism
\begin{equation}\label{iso:coactions}
\vcenter{\hbox{\xymatrix@!0@C=15mm{
I\ar[rr]^{X}\ar[rd]_{i}&&B\\
&R\ar[ru]_{Y}\ar@{}[u]|*[@]{\cong}&
}}}
\end{equation}
such that the following equation holds.
\begin{equation}
\label{ax:coactions:module}
\vcenter{\hbox{\xymatrix@!0@C=12mm{
I\ar[r]_-{i}\ar@/^9mm/[rrrr]^-{X}\xtwocell[rrrr]{}<>{^<-3>\zeta}&R\ar[r]_-{i^*}&I\ar[rr]^{X}\ar[rd]_{i}&&B\\
&&&R\ar[ru]_{Y}\ar@{}[u]|*[@]{\cong}&
}}}
\quad=\quad
\vcenter{\hbox{\xymatrix@!0@C=12mm{
I\ar[r]_-{i}\ar@/^9mm/[rrrr]^-{X}&R\ar[dr]_-{i^*}\ar@/^3mm/[rr]|-{1}\xtwocell[rr]{}<>{^<1>\varepsilon}&\ar@{}[]!<0mm,13mm>|*=0[@]{\cong}&R\ar[r]_{Y}&B\\
&&I\ar[ru]_{i}&&
}}}
\end{equation}
Furthermore, axiom~\eqref{ax:Z1} may be read as the fact that $x$ is a morphism of modules for the monad induced by $i\dashv i^*$, thus by opmonadicity there exists a cell $\tilde{y}$ such that the following equation holds.
\begin{equation}
\label{ax:coactions:comodule}
\vcenter{\hbox{\xymatrix@!0@R=7.5mm@C=10mm{
&&B\ar[rdd]^-{1i}&&\\
&&&&\\
&&&BR\ar[rdd]^-{b}&\\
&&&&\\
I\ar[rr]_-{i}\ar@/^9mm/[rruuuu]^-{X}\ar@/^9mm/[rrrr]^-{X}\xtwocell[rrr]{}<>{^<-10>x}&&R\ar[rr]_-{Y}\ar@{}[u]|>>*=0[@]{\cong}&&B
}}}
\quad=\qquad
\vcenter{\hbox{\xymatrix@!0@R=7.5mm@C=10mm{
&&B\ar[rdd]^-{1i}&&\\
&&&&\\
&R\ar[rd]_-{1i}\ar[ruu]^-{Y}\ar@{}[lu]|<<<*=0[@]{\cong}&&BR\ar[rdd]^-{b}&\\
&&RR\ar[ru]^-{Y1}\ar@{}[uu]|*[@]{\cong}&&\\
I\ar[rr]_-{i}\ar[ruu]^-{i}\ar@/^9mm/[rruuuu]^-{X}\ar@{}[rru]|*[@]{\cong}&&R\ar[rr]_-{Y}\ar[u]^-{i1}\xtwocell[rru]{}<>{^<-1>\tilde{y}}&&B
}}}
\end{equation}
The data $(Y,\tilde{y})$ constitute an object of $\rComod_{\id_{R}}^{\text{mate}}((R,i^*1),(B,b))$: axiom~\eqref{ax:COM1}\textsuperscript{mate} for $(Y,\tilde{y})$ follows by precomposing both sides with the opmonadic arrow $\xymatrix@1@C=5mm{i:I\ar[r]&R}$ to obtain each side of the coassociative axiom for the coaction $x$, which are equal; and axiom~\eqref{ax:COM2}\textsuperscript{mate} is equal to the counit axiom for the coaction $x$. Hence, in light of equations~\eqref{ax:coactions:module} and \eqref{ax:coactions:comodule} the isomorphism~\eqref{iso:coactions} is in $\mathcal{Z}(I,B)$ and reads as $K(Y,\tilde{y})\cong(X,x)$, so $K$ is essentially surjective on objects. Now, let $\xymatrix@1@C=5mm{\chi:(X,x,\zeta)\ar[r]&(X',x',\zeta')}$ be a morphism in $\mathcal{Z}(I,B)$, as $\chi$ is a morphism of modules for the monad associated to $i\dashv i^*$ there exists a cell $\xymatrix@1@C=5mm{\gamma:Y\ar[r]&Y'}$ in $\mathcal{M}$ such that the following equation holds.
\begin{equation}
\label{iso:opmonadicity:cells2}
\vcenter{\hbox{\xymatrix@!0@C=15mm{
I\ar[rd]_-{i}\ar[rr]^-{X}\ar@/^10mm/[rr]<1mm>^-{X'}\xtwocell[rr]{}<>{^<-4>\chi}&&B\\
&R\ar[ru]_-{Y}\ar@{}[u]|-*[@]{\cong}&
}}}
\quad=\quad
\vcenter{\hbox{\xymatrix@!0@C=15mm{
I\ar[rd]_-{i}\ar@/^10mm/[rr]<1mm>^-{X'}\ar@{}[r]|>>>>>*[@u]{\cong}&&N\\
&R\ar[ru]_-{Y}\ar@/^8mm/[ru]^-{Y'}\xtwocell[ru]{}<>{^<-2>\gamma}&
}}}
\end{equation}
To prove that $\gamma$ is a cell in $\rComod_{\id_{R}}^{\text{mate}}((R,i^*1),(B,b))$ precompose both sides of axiom \eqref{ax:COM3}\textsuperscript{mate} with the opmonadic arrow $i$; this produces the two sides of the axiom that makes $\chi$ into a morphism of $c$-comodules, which are equal. Thus $\gamma$ is in $\rComod_{\id_{R}}^{\text{mate}}((R,i^*1),(B,b))$ and the equation~\eqref{iso:opmonadicity:cells2} now reads as $K(\gamma)=\chi$, so $K$ is full. The theorem follows by the sequence of equivalences and isomorphisms below.
\[
\rComod_{\id_{R}}((R,i^*1),(B,b))\cong\rComod_{\id_{R}}^{\text{mate}}((R,i^*1),(B,b))\simeq\mathcal{Z}(I,B)\cong\mathcal{M}(I,B)_{\mathcal{M}(I,c)}
\]
\end{proof}

\begin{cor}\label{cor:ComodulesAreComodules}
Let $\mathcal{M}$ be an opmonadic-friendly monoidal bicategory. For every pair of bidualities $R\dashv R\ot$ and $S\dashv S\ot$, every opmonoidal arrow $\xymatrix@1@C=5mm{C:R\ot R\ar[r]&S\ot S}$, and every opmonadic adjunction $i\dashv i^*$ whose dual $i\ob\dashv i\ot$ is opmonadic too,
\[
\xymatrix{
R\dtwocell_{i}^{i^*}{'\dashv}\\
I
}
\qquad
\xymatrix{
R\ot\dtwocell_{i\ob}^{i\ot}{'\dashv}\\
I
}
\]
there is an equivalence of categories,
\[
\rComod_C((R,e1),(S,e1))\simeq\mathcal{M}(I,S)_{\mathcal{M}(I,c)}
\]
\[
\vcenter{\hbox{\xymatrix@!0@=15mm{
RR\ot R\ar[r]^-{YC}\ar[d]_-{e1}\xtwocell[rd]{}<>{^y}&SS\ot S\ar[d]^-{e1}\\
R\ar[r]_-{Y}&S
}}}
\quad
\vcenter{\hbox{\xymatrix@!0@=15mm{
\ar@{<~>}[r]&
}}}
\quad
\vcenter{\hbox{\xymatrix@!0@C=12mm@R=5mm{
&S\ar[rd]^-{c}&\\
I\ar[ru]^-{X}\ar@/_4mm/[rr]_-{X}\xtwocell[rr]{}<>{^x}&&S\\
}}}
\]
where $\xymatrix@1@C=5mm{c:S\ar[r]&S}$ is the comonad induced by $C$ as in Remark~\ref{rem:InducedComonad}.
\end{cor}
\begin{proof}

Let $\xymatrix@1@C=5mm{s:SR\ar[r]&S}$ be the oplax right action that corresponds to the opmonoidal arrow $C$ under the equivalence in Corollary~\ref{cor:Opmon_is_OplaxAct_Local}. By Remark~\ref{rem:InducedComonad} $c$ is the comonad induced both by the opmonoidal arrow $C$ and by the oplax action $s$. Then there is an equivalence of categories,
\[
\rComod_C((R,e1),(S,e1))\cong\rComod_{\id_R}((R,i^*1),(S,s))\simeq\mathcal{M}(I,S)_{\mathcal{M}(I,c)}
\]
where the isomorphism is an instance of Corollary~\ref{cor:Comodules_Local} and the equivalence is an instance of Theorem~\ref{teo:SkewActions_are_Coactions}.
\end{proof}

We conclude with a remark about the motivating example $\mathcal{M}=\Mod_k$. In Lemma~\ref{lem:CoalgebroidsAsQuantum} we saw how to translate between $R|S$-coalgebroids and opmonoidal arrows between enveloping monoidales $\xymatrix@1@C=5mm{C:R\ot R\ar[r]&S\ot S}$ in $\Mod_k$. There is a standard definition of a comodule for an $R|S$-coalgebroid found for example in \cite[1.4]{Hai2008} or \cite[3.6]{Bohm2009}.

\begin{defi}
Let $R$ and $S$ be two $k$-algebras and $C$ an $R|S$-coalgebroid. A $C$-comodule $X$ is a comodule for the underlying comonoid in $S$-$\Mod$-$S$ of the coalgebroid $C$, i.e. a module $X$ in $\Mod$-$S$ together with a module morphism $\xymatrix@1@C=5mm{x:X\otimes_S C\ar[r]&X}$ in $\Mod$-$S$, called the $C$-coaction, which satisfies coassociative and counit laws.
\end{defi}

\begin{rem}
If we apply Corollary~\ref{cor:ComodulesAreComodules} to the case $\mathcal{M}=\Mod_k$, we recover \cite[Lemma~1.4.1]{Hai2008}, also found in \cite[Lemma~3.17]{Bohm2009}. This is an equivalence between comodules for coalgebroids, as defined above, and comodules for the opmonoidal arrows in $\Mod_k$ that correspond to coalgebroids as described in Example~\ref{exa:ComodulesForCoalgebroids}. Moreover, these two versions of comodules for coalgebroids are also equivalent to those defined via oplax actions as in Example~\ref{exa:ComoduleForOplaxAction}. The only difference between the definition of comodule for a coalgebroid via oplax actions and the standard definition is that in the former the underlying module is in $R$-$\Mod$-$S$ while in the latter is in $\Mod$-$S$.
\end{rem}

Now, a sufficient condition to have a monoidal structure on the category of comodules for a coalgebroid, is that the coalgebroid is in fact a bialgebroid \cite[Corollary 1.7.2]{Hai2008}. In our language, just as a coalgebroid means an opmonoidal arrow $\xymatrix@1@C=5mm{C:R\ot R\ar[r]&S\ot S}$, a bialgebroid is an opmonoidal monad $\xymatrix@1@C=5mm{B:R\ot R\ar[r]&R\ot R}$. This description of bialgebroids in the language of monoidal bicategories is due to \cite{Day2003a} and it is motivated by the work of \cite{Szlachanyi2003}. We get a monoidal structure in the category of comodules for opmonoidal monads on an enveloping monoidale $R\ot R$ in a similar way.

\begin{teo}
For every object $A$, every right skew monoidale $M$, every oplax right action $\xymatrix@1@C=5mm{a:AM\ar[r]&A}$, and every opmonoidal monad $\xymatrix@1@C=5mm{B:M\ar[r]&M}$ the category of right $B$-comodules $\rComod_B((A,a),(A,a))$ has a monoidal structure such that the forgetful functor
\[
\vcenter{\hbox{\xymatrix{
\rComod_B((A,a),(A,a))\ar[r]&\mathcal{M}(A,A)
}}}
\]
is strong monoidal. The tensor product and unit of $B$ comodules is calculated as shown.
\[
\vcenter{\hbox{\xymatrix@!0@=10mm{
&&AM\ar[rd]^-{Y11}&&\\
&AM\ar[ru]^-{1B}\ar[rd]|-{Y11}&&AM\ar[rd]^-{Z11}&\\
AM\ar[rr]_-{YB}\ar[dd]_-{a}\ar[ru]^-{1B}\ar@/^12mm/[rruu]^-{1B}\xtwocell[rrdd]{}<>{^y}\xtwocell[rruu]{}<>{^<-4>1\mu\ }&&AM\ar[rr]_-{ZB}\ar[dd]_-{a}\ar[ru]|-{1B}\xtwocell[rrdd]{}<>{^z}\ar@{}[uu]|*=0[@]{\cong}&&AM\ar[dd]^-{a}\\
\\
A\ar[rr]_Y&&A\ar[rr]_Z&&A
}
\qquad\quad
\xymatrix@!0@=20mm{
\\
AM\ar[r]_-{1}\ar[d]_-{a}\ar@/^7mm/[r]^-{1B}\xtwocell[r]{}<>{^<-2>1\eta\ }&AM\ar[d]^-{a}\\
A\ar[r]_-{1}&A
}}}
\]
\end{teo}
\begin{proof}

The associator and left and right unitor isomorphisms are induced by those of the horizontal composition of $\rComod(\mathcal{M})$. And the axioms for a monoidal category follow from the associativity and unitality of the monad structure of $B$ and from coherence axioms for the horizontal composition of $B$-comodules.
\end{proof}

If we let $A=R$, $M=R\ot R$, and $\xymatrix@1@C=5mm{a=e1:RR\ot R\ar[r]&R}$ we obtain the following result.
\begin{cor}
For every opmonoidal monad $\xymatrix@1@C=5mm{B:R\ot R\ar[r]&R\ot R}$ on an enveloping monoidale $R\ot R$ in $\mathcal{M}$ the category $\rComod_B((R,e1),(R,e1))$ of right $B$-comodules has a monoidal structure such that the forgetful functor
\[
\xymatrix{\rComod_B((R,e1),(R,e1))\ar[r]&\mathcal{M}(R,R)}
\]
is strong monoidal.
\end{cor}
\bibliographystyle{alpha}
\bibliography{ComodulesForCoalgebroids}

\begin{thebibliography}{{Mac}97}

\bibitem[AA17]{Abud2017}
Ramón Abud~Alcalá.
\newblock {\em Oplax actions and enriched icons with applications to
  coalgebroids and quantum categories}.
\newblock PhD thesis, Macquarie University, 2017.

\bibitem[B{\'e}n67]{Benabou1967}
Jean B{\'e}nabou.
\newblock Introduction to bicategories.
\newblock In {\em Reports of the Midwest Category Seminar}, volume~47 of {\em
  Lecture Notes in Mathematics}, pages 1--77, Berlin, Heidelberg, 1967.
  Springer Berlin Heidelberg.

\bibitem[B{\"{o}}h09]{Bohm2009}
Gabriella B{\"{o}}hm.
\newblock Hopf algebroids.
\newblock In M.~Hazewinkel, editor, {\em Handbook of Algebra}, volume~6, pages
  173--235. North-Holland, 2009.

\bibitem[DS97]{Day1997}
Brian Day and Ross Street.
\newblock Monoidal bicategories and {Hopf} algebroids.
\newblock {\em Advances in Mathematics}, 129(1):99--157, 1997.

\bibitem[DS04]{Day2003a}
Brian Day and Ross Street.
\newblock Quantum categories, star autonomy, and quantum groupoids.
\newblock In {\em Galois Theory, Hopf Algebras, and Semiabelian Categories},
  volume~43 of {\em Fields Institute Communications}, pages 187--225. American
  Mathematical Society, 2004.

\bibitem[GPS95]{Gordon1995}
R.~Gordon, A.~John Power, and Ross Street.
\newblock {\em Coherence for Tricategories}.
\newblock Number 558 in Memoirs of the American Mathematical Society. American
  Mathematical Society, 1995.

\bibitem[Gur13]{Gurski2013}
Nick Gurski.
\newblock {\em Coherence in Three-dimensional Category Theory}.
\newblock Number 201 in Cambridge Tracts in Mathematics. Cambridge University
  Press, 2013.

\bibitem[Kel64]{Kelly1964}
Gregory~Maxwell Kelly.
\newblock On {MacLane's} conditions for coherence of natural associativities,
  commutativities, etc.
\newblock {\em Journal of Algebra}, 1(4):397--402, 1964.

\bibitem[Kel74]{Kelly1974a}
Gregory~Maxwell Kelly.
\newblock Doctrinal adjunction.
\newblock In {\em Proceedings Sydney Category Theory Seminar 1972/1973}, volume
  420, pages 257--280. Springer-Verlag Berlin Heidelberg, 1974.

\bibitem[KS74]{Kelly1974}
Gregory~Maxwell Kelly and Ross Street.
\newblock Review of the elements of 2-categories.
\newblock In {\em Proceedings Sydney Category Theory Seminar 1972/1973}, volume
  420, pages 75--103. Springer-Verlag Berlin Heidelberg, 1974.

\bibitem[Lac10]{Lack2010a}
Stephen Lack.
\newblock A 2-categories companion.
\newblock In John~C. Baez and J.~Peter May, editors, {\em Towards Higher
  Categories}, pages 105--191. Springer, New York Dordrecht Heidelberg London,
  2010.

\bibitem[Lac14]{Lack2014c}
Stephen Lack.
\newblock Morita contexts as lax functors.
\newblock {\em Applied Categorical Structures}, 22(2):311--330, 2014.

\bibitem[LS12]{Lack2012}
Stephen Lack and Ross Street.
\newblock Skew monoidales, skew warpings and quantum categories.
\newblock {\em Theory and Applications of Categories}, 26(15):385--402, 2012.

\bibitem[LS14]{Lack2014b}
Stephen Lack and Ross Street.
\newblock On monads and warpings.
\newblock {\em Cahiers de Topologie et G{\'{e}}om{\'{e}}trie
  Diff{\'{e}}rentielle Cat{\'{e}}goriques}, LV(Fascicule 4):244--266, 2014.

\bibitem[{Mac}97]{MacLane1997}
Saunders {Mac Lane}.
\newblock {\em Categories for the Working Mathematician}.
\newblock Number~5 in Graduate Texts in Mathematics. Springer-Verlag New York
  Berlin Heidelberg, second edition edition, 1997.

\bibitem[McC02]{McCrudden2002}
Paddy McCrudden.
\newblock Opmonoidal monads.
\newblock {\em Theory and Applications of Categories}, 10(19):469--485, 2002.

\bibitem[MLP85]{MacLane1985}
Saunders Mac~Lane and Robert Par{\'{e}}.
\newblock Coherence for bicategories and indexed categories.
\newblock {\em Journal of Pure and Applied Algebra}, 37:59--80, 1985.

\bibitem[Moe02]{Moerdijk2002}
Ieke Moerdijk.
\newblock Monads on tensor categories.
\newblock {\em Journal of Pure and Applied Algebra}, 168(2-3):189--208, 2002.

\bibitem[Ph{\`{u}}08]{Hai2008}
H{\^{o}}~Hai Ph{\`{u}}ng.
\newblock {Tannaka-Krein} duality for {Hopf} algebroids.
\newblock {\em Israel Journal of Mathematics}, 167:193--225, 2008.

\bibitem[Pow90]{Power1990}
A.~John Power.
\newblock A 2-categorical pasting theorem.
\newblock {\em Journal of Algebra}, 129(2):439--445, 1990.

\bibitem[Sch98]{Schauenburg1998}
Peter Schauenburg.
\newblock Bialgebras over noncommutative rings and a structure theorem for
  {Hopf} bimodules.
\newblock {\em Applied Categorical Structures}, 6:193--222, 1998.

\bibitem[Str72]{Street1972}
Ross Street.
\newblock The formal theory of monads.
\newblock {\em Journal of Pure and Applied Algebra}, 2:149--168, 1972.

\bibitem[Str80]{Street1980}
Ross Street.
\newblock Fibrations in bicategories.
\newblock {\em Cahiers de Topologie et G{\'{e}}om{\'{e}}trie
  Diff{\'{e}}rentielle Cat{\'{e}}goriques}, 21:111--160, 1980.

\bibitem[Str07]{Street2007}
Ross Street.
\newblock {\em Quantum Groups: A Path to Current Algebra}.
\newblock Number~19 in Australian Mathematical Society Lecture Series.
  Cambridge University Press, 2007.

\bibitem[Swe74]{Sweedler1974}
Moss~E. Sweedler.
\newblock Groups of simple algebras.
\newblock {\em Publications Mathématiques de l'IH\'{E}S}, 44:79--189, 1974.

\bibitem[Szl03]{Szlachanyi2003}
Korn{\'{e}}l Szlach{\'{a}}nyi.
\newblock The monoidal {Eilenberg-Moore} construction and bialgebroids.
\newblock {\em Journal of Pure and Applied Algebra}, 182(2-3):287--315, 2003.

\bibitem[Szl05]{Szlachanyi2004}
Korn{\'{e}}l Szlach{\'{a}}nyi.
\newblock Monoidal morita equivalence.
\newblock In {\em Noncommutative Geometry and Representation Theory in
  Mathematical Physics}, pages 353--369. American Mathematical Society -
  Contemporary Mathematics, 2005.

\bibitem[Szl12]{Szlachanyi2012}
Korn{\'{e}}l Szlach{\'{a}}nyi.
\newblock Skew-monoidal categories and bialgebroids.
\newblock {\em Advances in Mathematics}, 231:1694--1730, 2012.

\bibitem[Tak77]{Takeuchi1977}
Mitsuhiro Takeuchi.
\newblock Groups of algebras over {$A\otimes\overline{A}$}.
\newblock {\em Journal of the Mathematical Society of Japan}, 29(3):459--492,
  1977.

\bibitem[Tak87]{Takeuchi1987}
Mitsuhiro Takeuchi.
\newblock {$\sqrt{\mathrm{Morita}}$} theory --- {Formal} ring laws and monoidal
  equivalences of categories of bimodules ---.
\newblock {\em Journal of the Mathematical Society of Japan}, 39(2):301--336,
  1987.

\bibitem[Ver92]{Verity1992}
Dominic Verity.
\newblock {\em Enriched Categories, Internal Categories and Change of Base}.
\newblock PhD thesis, Fitzwilliam College, Cambridge, 1992.

\end{thebibliography}
\end{document}